\documentclass[11pt, twoside, leqno, draft]{article}

\usepackage{amsfonts}

\usepackage{amssymb}
\usepackage{amsmath}
\usepackage{amsthm}
\usepackage{xcolor}
\usepackage{mathrsfs}
\usepackage{txfonts}

\usepackage{indentfirst}

\allowdisplaybreaks

\def\rr{{\mathbb R}}
\def\rn{{\mathbb{R}^n}}

\def\nn{{\mathbb N}}
\def\zz{{\mathbb Z}}
\def\cc{{\mathbb C}}
\def\CG{{\mathcal G}}
\def\CL{{\mathcal L}}
\def\CS{{\mathcal S}}
\def\CR{{\mathcal R}}
\def\CM{{\mathcal M}}
\def\CA{{\mathcal M}}
\def\CX{{\mathcal X}}
\def\CY{{\mathcal Y}}
\def\CB{{\mathcal B}}

\newcommand{\FT}{\mathfrak T}
\newcommand{\FR}{\mathfrak R}

\newcommand{\RZ}{\mathrm Z}

\newcommand{\CN}{\mathcal{N}}

\def\fz{\infty }
\def\az{\alpha}
\def\bz{\beta}
\def\dz{\delta}
\def\ez{\epsilon}
\def\kz{\kappa}
\def\thz{\theta}
\def\vz{\varphi}

\def\lf{\left}
\def\r{\right}
\def\ls{\lesssim}
\def\noz{\nonumber}
\def\wz{\widetilde}
\def\gfz{\genfrac{}{}{0pt}{}}

\def\loc{{\mathrm{loc}}}
\DeclareMathOperator{\supp}{supp}

\def\XXint#1#2#3{{\setbox0=\hbox{$#1{#2#3}{\int}$ }
\vcenter{\hbox{$#2#3$ }}\kern-.6\wd0}}

\def\lz{{\lambda}}

\def\CG{{\mathcal G}}
\def\CA{{\mathcal A}}
\def\CS{{\mathcal S}}
\def\CY{\mathcal Y}
\def\RY{{\mathrm Y}}
\def\gz{\gamma}

\newcommand{\RJ}{\mathrm J}

\DeclareMathOperator{\diam}{diam}

\newcommand{\wt}{\widetilde}

\def\gfz{\genfrac{}{}{0pt}{}}

\newtheorem{theorem}{Theorem}[section]
\newtheorem{lemma}[theorem]{Lemma}
\newtheorem{corollary}[theorem]{Corollary}
\newtheorem{proposition}[theorem]{Proposition}

\theoremstyle{definition}
\newtheorem{remark}[theorem]{Remark}
\newtheorem{definition}[theorem]{Definition}
\renewcommand{\appendix}{\par
   \setcounter{section}{0}%
   \setcounter{subsection}{0}%
   \setcounter{subsubsection}{0}%
   \gdef\thesection{\@Alph\c@section}%
   \gdef\thesubsection{\@Alph\c@section.\@arabic\c@subsection}%
   \gdef\theHsection{\@Alph\c@section.}%
   \gdef\theHsubsection{\@Alph\c@section.\@arabic\c@subsection}%
   \csname appendixmore\endcsname
 }
\newcommand{\OCG}{\mathring\CG}
\newcommand{\go}[1]{\CG_0^\eta(#1)}
\newcommand{\GOX}[1]{\OCG(#1)}
\newcommand{\GO}[1]{\mathring{\CG}(#1)}
\newcommand{\GOO}[1]{\mathring{\CG}^\eta_0(#1)}

\numberwithin{equation}{section}

\begin{document}
\title{\bf\Large
New Calder\'{o}n Reproducing Formulae with Exponential Decay on Spaces of Homogeneous Type
\footnotetext{\hspace{-0.35cm}
2010 {\it Mathematics Subject Classification}. Primary 42C40; Secondary 42B20, 42B25, 30L99. \endgraf
{\it Key words and phrases.} space of homogeneous type, Calder\'{o}n reproducing formula, approximation of
the identity,  wavelet, space of test functions, distribution.\endgraf
This project is supported by the National Natural Science Foundation of China (Grant Nos.
11771446, 11571039, 11726621, 11761131002 and 11471042).}}
\author{Ziyi He, Liguang Liu, Dachun Yang\footnote{Corresponding author/April 1, 2018/newest version.}\ \ and
Wen Yuan}
\date{}
\maketitle

\vspace{-0.8cm}

\begin{center}
\begin{minipage}{13cm}
{\small {\bf Abstract}\quad
Assume that $(X, d, \mu)$ is a space of homogeneous type in the sense of Coifman and Weiss.
In this article, motivated by the breakthrough work of P. Auscher and T. Hyt\"onen
on orthonormal bases of regular wavelets on spaces of homogeneous type, the authors introduce a new
kind of approximations of the identity  with
exponential decay (for short, $\exp$-ATI). Via such an $\exp$-ATI, motivated by another
creative idea of Y. Han et al. to merge
the aforementioned orthonormal bases of regular wavelets into the frame of
the existed distributional theory on spaces of homogeneous type,
the authors establish the
homogeneous continuous/discrete Calder\'{o}n reproducing formulae on $(X, d, \mu)$,
as well as their inhomogeneous counterparts.
The novelty of this article exists in that $d$ is only assumed to be a quasi-metric
and the underlying measure $\mu$ a doubling measure, not necessary to satisfy
the reverse doubling condition. It is well known that Calder\'{o}n reproducing formulae
are the cornerstone to develop analysis and, especially, harmonic analysis on spaces
of homogeneous type.}
\end{minipage}
\end{center}


\arraycolsep=1pt


\vspace{-0.2cm}

\section{Introduction}\label{intro}

It is well known that the space of homogeneous type introduced 
by Coifman and Weiss \cite{CW71, CW77} provides a natural setting for
the study of function spaces and the boundedness of operators. A \emph{quasi-metric space} $(X,d)$ is a
non-empty set $X$ equipped with a \emph{quasi-metric} $d$, that is, a non-negative function defined on
$X\times X$, satisfying that, for any $x,\ y,\ z\in X$,
\begin{enumerate}
\item $d(x,y)=0$ if and only if $x=y$;
\item $d(x,y)=d(y,x)$;
\item there exists a  constant $A_0\in[1,\infty)$ such that $d(x,z)\le A_0[d(x,y)+d(y,z)]$.
\end{enumerate}
The ball $B$ on $X$ centered at $x_0\in X$ with radius $r\in(0,\fz)$ is defined by setting
$$
B:=\{x\in X:\ d(x,x_0)<r\}=:B(x_0,r).
$$
For any ball $B$ and $\tau\in(0,\infty)$, denote by $\tau B$ the ball with the same center as that of $B$
but of radius $\tau$ times that of $B$.
Given a quasi-metric space $(X,d)$ and a non-negative measure $\mu$, we call $(X,d,\mu)$ a \emph{space of homogeneous type} if $\mu$ satisfies the following {\it doubling condition}:
there exists a positive constant $C_{(\mu)}\in[1,\fz)$ such that,
for any ball $B\subset X$,
$$\mu(2B)\le
C_{(\mu)}\mu(B).$$
The above doubling condition is equivalent to that, for any ball $B$ and any $\lz\in[1,\infty)$,
\begin{equation}\label{eq:doub}
\mu(\lz B)\le C_{(\mu)}\lz^\omega\mu(B),
\end{equation}
where $\omega:=\log_2{C_{(\mu)}}$  measures the \emph{upper dimension} of $X$.
If $A_0=1$, we call $(X, d, \mu)$ a metric measure space of homogeneous type or, simply, a \emph{doubling
measure metric space}.

To develop the real-variable theory of function spaces and the boundedness of
operators on  spaces of homogeneous type, one needs Calder\'{o}n reproducing formulae
which play a very important and fundamental
role. Indeed, Calder\'on reproducing formulae and their variants are
important tools to characterize Hardy spaces (see \cite{Uchi85, GLY08, YZ10, dy03, BT, BDL}),
to study Besov spaces and Triebel-Lizorkin spaces (see \cite{Pee75, FJ85, FJ90,FJW91}), and also more general
scale of function spaces (see \cite{YY08,YY10,YSY}), as well to establish $T1$ or $Tb$-theorems for Besov or Triebel-Lizorkin spaces (see \cite{DJS,  Han94, Han97, HS94, Ro92}).
All these have further applications in the study on
the boundedness of Calder\'on-Zygmund operators (see, for example, \cite{fcy17, lcfy17, lcfy18, LX}).
On the anisotropic Euclidean space, which is a special case of spaces of homogeneous type,
Calder\'{o}n reproducing formulae were  built and then used to study Besov spaces and
Triebel-Lizorkin spaces (see \cite{Bow05, BH06, BLYZ08,LWYY17}).
The reader can also find nice applications of Calder\'{o}n reproducing formulae in the study of Musielak-Orlicz
Hardy spaces and Hardy spaces with variable exponents (see, for example,
\cite{FHLY17, LHY12,YLK17,NS12, ZSY16}).

Classical Calder\'{o}n reproducing formulae on the Euclidean space $\rn$
gain its prototype from Calder\'{o}n \cite{Ca64}, in which problems of the intermediate space and the interpolation
were studied via a
reproducing method. In 1971, Heideman
\cite{Hei71} proved  the (in)homogeneous continuous Calder\'{o}n reproducing formulae in detail and used them
to consider the duality and the
fractional integral on Lipschitz spaces; see also  Calder\'on and Torchinsky \cite{CT75, Ca77}.
Denote by $\mathcal S'(\rn)/\mathcal P(\rn)$  the space $\mathcal S'(\rn)$ of Schwartz distributions modul
the polynomials space $\mathcal P(\rn)$, which is known topologically equivalent to $\CS_\fz'(\rn)$ (see
\cite[Proposition 8.1]{YSY}, \cite[Theorem 6.28]{nns16} and \cite[Theorem 3.1]{s17} for a recent proof), where
$\CS_\fz(\rn)$ denotes the set of all Schwartz functions having infinite vanishing moments and $\CS_\fz'(\rn)$
its dual space equipped with the weak-$*$ topology.
Then the homogeneous continuous Calder\'on reproducing formula in \cite{Hei71, CT75, Ca77} essentially has the
version
$$
f=\sum_{\nu\in\zz} \varphi_\nu\ast\psi_\nu\ast f \qquad \textup{in}\;\;\mathcal S'(\rn)/\mathcal P(\rn),
\leqno(\bf CCRF)
$$
where $\varphi,\ \psi$ belong to the Schwartz class $\mathcal S(\rn)$ with their Fourier transforms
$\widehat\varphi,\ \widehat\psi$ satisfying
$\supp\widehat\varphi,\ \supp\widehat\psi\subset\{\xi\in\rn:\, 1/4<|\xi|<4\},$
and $\varphi_\nu(\cdot):=2^{\nu n}\varphi(2^\nu\cdot)$ and $\psi_\nu(\cdot):=2^{\nu n}\psi(2^\nu \cdot)$ for any $\nu\in\zz$.
The homogeneous  discrete Calder\'on reproducing formula was first established by Frazier and Jawerth \cite{FJ85}, which reads as follows:
$$
f(\cdot)=\sum_{\nu\in\zz} 2^{-\nu n}\sum_{k\in\zz^n}(\varphi_\nu\ast f)(2^{-\nu}k)\,\psi_\nu(\cdot-2^{-\nu}k) \qquad \textup{in}\;\;\mathcal S'(\rn)/\mathcal P(\rn).
\leqno(\bf DCRF)
$$
The reader may also find corresponding inhomogeneous continuous/discrete Calder\'on reproducing formulae in literatures
(see, for example, \cite{FJ90, YSY}).

In the setting of \emph{Ahlfors-$n$ regular spaces} [that is, $\mu(B(x,r))\sim r^n$, for any $x\in X$ and
$r\in(0,\infty)$, with the equivalent positive constants independent of $x$ and $r$], upon assuming
that there exists $\thz\in(0,1]$ such that, for any $x,\ x',\ y\in X$,
\begin{align}\label{regular-d}
|d(x,y)-d(x',y)|\ls[d(x,x')]^\thz[d(x,y)+d(x',y)]^{1-\thz},
\end{align}
then Calder\'{o}n reproducing formulae were established in \cite{Han97, HLY01, HS94} and further used to study
Besov and Triebel-Lizorkin spaces as well as $T1$ or $Tb$-theorems.

Recall that an RD-\emph{space} $(X,d,\mu)$  is a doubling metric space
with $\mu$ further satisfying the following
\emph{reverse doubling
	condition}: there exists a positive constant $\wz{C}\in(0,1]$ and $\kz\in(0,\omega]$
such that, for any ball $B(x,r)$ with $r\in(0,\diam X/2)$ and $\lz\in[1,\diam X/(2r))$,
$$
\wz C\lz^\kz\mu(B(x,r))\le\mu(B(x,\lz r)),
$$
here and hereafter, $\diam X:=\sup\{d(x,y):\ x,\ y\in X\}$.
In RD-spaces, (in)homogeneous continuous/discrete Calder\'on reproducing formulae
analogous to ({\bf CCRF}) and ({\bf DCRF}) were established
in \cite{HMY08}, and further used to built a full theory of (in)homogeneous Besov and Triebel-Lizorkin spaces on RD-spaces.
Also, in \cite{GLY08, HMY06, YZ10}, Calder\'on reproducing formulae  were also used to establish
Littlewood-Paley function characterizations and various maximal function characterizations of
Hardy spaces on RD-spaces. See also \cite{YZ11} to use Calder\'on reproducing formulae and  local Hardy spaces
to obtain new characterizations of Besov and Triebel-Lizorkin spaces on RD-spaces.
More applications of Calder\'on reproducing formulae in analysis on RD-spaces can be found in
\cite{glmy, gly, kyz10, kyz11, yz08, YZ11,dh09}.

Notice that the reverse doubling condition of $\mu$ and the fact that $d$ is a metric play fundamental roles
in establishing the Calder\'on reproducing formulae in \cite{HMY08}.
It is  interesting to
ask whether or not Calder\'on reproducing formulae like ({\bf CCRF}) and ({\bf DCRF}) can be established
without the reverse doubling condition of $\mu$ and with $d$ being a quasi-metric.
This article gives an affirmative answer to this question. It is these
Calder\'{o}n reproducing formulae built in the present article
that impel us to establish,
\emph{without any additional geometrical condition},
various real-variable characterizations of Hardy spaces
on spaces homogeneous type,
which completely answers the question
asked by Coifman and Weiss \cite[p.\,642]{CW77}.
Due to the limited length of this article, the latter part is presented in \cite{hhlly}.

One motivation of this article is the breakthrough work of Auscher and Hyt\"onen \cite{AH13},
where an  orthonormal wavelet basis $\{\psi_\az^k:\ k\in\zz,\ \az\in\CG_k\}$  of $L^2(X)$
with H\"older continuity and exponential decay
on spaces of homogeneous type was constructed, by using the system of random dyadic cubes.
For any $k\in\zz$ and $x,\ y\in X$, if we let
$$Q_k(x,y):=\sum_{\az\in\CG_k}\psi_\az^k(x)\psi_\az^k(y),$$
then
\begin{align}\label{wrf}
f(\cdot)=\sum_{k\in\zz}Q_kf(\cdot)=\sum_{k\in\zz}\sum_{\az\in\CG_k}\lf<f,\psi_\az^k\r>\psi_\az^k(\cdot)
=\sum_{k\in\zz}\int_X Q_k(\cdot,y)f(y)\,d\mu(y)\qquad\textup{in}\; L^2(X)
\end{align}
and the kernels $\{Q_k\}_{k\in\zz}$ were proved in \cite[Lemma~10.1]{AH13} to satisfy conditions (ii), (iii)
and (v) of Definition \ref{def:eti} below.
It was essentially proved in \cite[Lemma~3.6 and (3.22)]{HLW16} that the kernels $\{Q_k\}_{k\in\zz}$ satisfy
the ``second difference regularity condition'', with exponentially decay. This inspires us to introduce a new
kind of \emph{approximations of the identity with exponential decay} (for short, $\exp$-ATI); see Definition
\ref{def:eti} below.

Recall that any  \emph{approximations of the identity} (for short, ATI) on RD-spaces or Ahlfors-$n$ regular
spaces appeared in the literature only has the polynomial decay (see \cite{HMY08,dh09}).
The $\exp$-ATI turns out to be an  approximations of the identity used in \cite{HMY08}.
However, as it is pointed out in Remark \ref{rem-add} below, even the approximations of the identity with
bounded supports can not provide the exponential decay factor like
$$ \exp\lf\{-\nu\lf[\frac{\max\{d(x, \CY^k)),\,d(x,\CY^k)\}}{\dz^k}\r]^a\r\}$$
in the right-hand sides of \eqref{eq:etisize}, \eqref{eq:etiregx} and \eqref{eq:etidreg} below; we explain
the symbols $a,\ \dz$ and $\CY^k$ in Section \ref{pre} below.
The evidence for the importance of such an exponential decay  factor can be found in \cite[Lemma~8.3]{AH13}
which establishes the following estimate
$$
\sum_{\dz^k\ge r}\frac 1{\mu(B(x, \dz^k))}\exp\lf\{-\nu\lf[\frac{d(x,\CY^k)}{\dz^k}\r]^a\r\}
\ls \frac 1{\mu(B(x,r))},\qquad \forall\, x\in X,\ r\in(0,\infty),
$$
with the implicit positive constant independent of $x\in X$ and $r\in(0,\fz)$. Observe that this estimate can
be used as a replacement of the reverse doubling condition of $\mu$.

Another motivation of this article is the creationary works of Han et al.\ \cite{HHL16,hhl17,HLW16}
in which they attempted to show that \eqref{wrf} holds true when $f$ belongs to the distribution space
on spaces of homogeneous type (see Section \ref{ati} below). Indeed, it was established in
\cite[Theorem 3.4]{HLW16} the following \emph{wavelet reproducing formula}: for any $f$ in the space
$\GO{\bz,\gz}$ of test functions with $0<\bz'<\bz\le\eta$ and $0<\gz'<\gz\le\eta$ (see Definition
\ref{def:test} below),
\begin{equation*}
f=\sum_{k\in\zz}\sum_{\az\in\CG_k}\lf<f,\psi_\az^k\r>\psi_\az^k \qquad \textup{in}\; \GO{\bz',\gz'},
\end{equation*}
where $\eta\in(0,1)$ denotes the regularity exponent of the wavelets from \cite{AH13}.
This wavelet reproducing formula was used in \cite{HLW16} to obtain a Littlewood-Paley theory of Hardy spaces
on product spaces. It is the creative combination of the wavelet theory of \cite{AH13} with the existed
distributional theory on spaces of homogeneous type, which motivates us to consider analogous versions of
({\bf CCRF}) and ({\bf DCRF})  on spaces of homogeneous type in the sense of distributions.

Let us mention here that a Calder\'{o}n reproducing formula
for functions in the intersection of $L^2(X)$ and the Hardy space $H^p(X)$ was established in
\cite[Proposition 2.5]{HHL16}, and then was used to obtain atomic decompositions of Hardy spaces defined via
the Littlewood-Paley wavelet square functions. A deficit of \cite[Proposition 2.5]{HHL16} is that it does not
have exactly the analogous version of ({\bf DCRF}).
One might also mention here that the range of $p\in(\omega/(\omega+\eta),1]$ in \cite[Proposition 2.5]{HHL16}
seems to be \emph{problematic}. This is because the regularity exponent of the approximations of the identity in
\cite[p.\ 3438]{HHL16} is $\thz$ [indeed, $\thz$ is from the regularity of the quasi-metric $d$ in
\eqref{regular-d}], which leads to that the regularity exponent in \cite[(2.6)]{HHL16} should be
$\min\{\thz,\eta\}$ and hence the correct range of $p$ in \cite[Proposition 2.5]{HHL16} (indeed, all results of
\cite{HHL16}) seems to be $(\omega/[\omega+\min\{\thz,\eta\}],1]$. This range of $p$ is not optimal.

Via the aforementioned newly introduced $\exp$-ATI, we follow the Coifman idea in \cite{DJS} (see also
\cite{HMY08}) to establish the (in)homogeneous continuous/discrete Calder\'on reproducing formulae.
Let $\{Q_k\}_{k\in\zz}$ bs an  $\exp$-ATI  as in Definition \ref{def:eti}. Then, for any $N\in\nn$, we write
$$
I=\lf(\sum_{k=-\fz}^\fz Q_k\r)\lf(\sum_{l=-\fz}^\fz Q_l\r)=\sum_{|l|>N}\sum_{k=-\fz}^\fz Q_{k+l}Q_k
+\sum_{k=-\fz}^{\fz}Q_k^N Q_k=:R_N+T_N,
$$
where $I$ denotes the \emph{identity operator}.
When $N$ is sufficiently large, if we can prove that the operators norms of $R_N$ on both $L^2(X)$ and the
space $\GOO{\bz,\gz}$ of test functions (see Definition \ref{def:test} below) are all smaller than $1$, then
$T_N$ is invertible in both $L^2(X)$ and $\GOO{\bz,\gz}$, where $\bz,\ \gz\in(0,\eta)$. After setting
$\wz{Q}_k:=T_N^{-1}Q_k$ for any $k\in\zz$, we then have
\begin{equation}\label{eq:fcrf}
I=\sum_{k=-\fz}^\fz\wz{Q}_kQ_k,
\end{equation}
which is the homogeneous continuous Calder\'on reproducing formula. Moreover, \eqref{eq:fcrf}
holds true in the space of test functions, as well as its dual space, and also in $L^p(X)$ with any given
$p\in(1,\fz)$.

The difficulty to establish these Calder\'{o}n reproducing formulae lies in the treatment of $R_N$. This is
mainly because of the lack of the regularity of a quasi-metric.
For any $x_0\in X$ and $r\in(0,\infty)$, let $\GO{x_0, r, \bz,\gz}$ be the space of test functions (see
Definition \ref{def:test} below).
Recall that, in the setting of RD-spaces, the boundedness of $R_N$ on  $\GO{x_0, r, \bz,\gz}$ was ensured by
\cite[Theorem 2.18]{HMY08}. However,
the proof of \cite[Theorem 2.18]{HMY08} needs the existence of the
$1$-ATI with bounded support (see \cite[Definition 2.3 and Theorem 2.6]{HMY08}).
For a space of homogeneous type, the existence of the $\eta$-ATI with bounded support
is  hard to prove due to the lack of the regularity of the quasi-metric $d$. Indeed, it is still unknown
whether or not a corresponding theorem similar to \cite[Theorem 2.18]{HMY08} still holds true on a space of
homogeneous type. To overcome this essential difficulty, we observe that, for any $f\in\GO{x_0, r, \bz,\gz}$
and $x\in X$,
$$R_Nf(x)=\lim_{M\to\fz} R_{N,M}f(x)$$
holds true both in $L^2(X)$ and locally uniformly (see Lemma \ref{lem:ccrf3} below),
where each $R_{N,M}$  is associated to an integral kernel, still denoted by $R_{N,M}$, in the following way
$$
R_{N,M}g(x)=\int_X R_{N,M}(x,y)g(y)\,d\mu(y), \qquad \forall\, g\in\bigcup_{p\in[1,\fz]} L^p(X),\ \forall\,x\in X,
$$
with $R_{N,M}$ being a standard Calder\'on-Zygmund kernel and satisfying the ``second difference regularity
condition''. Thus, the boundedness of $R_N$ on $\GO{x_0, r, \bz,\gz}$ can then be reduced to the
corresponding boundedness of operators like $R_{N,M}$, while the latter is obtained in Theorem \ref{thm:Kbdd}
below. This is the key creative point used in this article to obtain the desired homogeneous continuous
Calder\'on reproducing formulae.

The above discussion mainly works for the proof of homogeneous continuous Calder\'{o}n reproducing formula.
For the homogeneous discrete one, we formally apply the mean value
theorem to \eqref{eq:fcrf}. For the inhomogeneous ones, the
difficulties we meet are similar to those for homogeneous ones.
For their detailed proofs, see Sections \ref{hdrf} and \ref{idrf} below, respectively.

Compared with the Calder\'{o}n reproducing formulae on RD-spaces (or Ahlfors-$n$ regular spaces) in
\cite{HMY08}, which holds true in  the space $\mathring\CG_0^\ez(\bz,\gz)$ of test functions  with $\ez$
\emph{strictly} smaller than the H\"older regularity exponent $\eta$
of the approximations of the identity, here
all Calder\'{o}n reproducing formulae obtained in this article
hold true in $\mathring\CG_0^\eta(\bz,\gz)$, that is, the regularity exponent of the space of test functions
can attain the corresponding one of the approximations of the identity.

Following \cite[pp.\ 587--588]{CW77}, throughout this article, we always make the following assumptions:
for any point $x\in X$, assume that the balls $\{B(x,r)\}_{r\in(0,\infty)}$
form a basis of open neighborhoods of $x$; assume that $\mu$ is Borel regular, which means that open sets are measurable and every set $A\subset X$
is contained in a Borel set $E$ satisfying that $\mu(A)=\mu(E)$; we also
assume that $\mu(B(x, r))\in(0,\fz)$ for any $x\in X$ and $r\in(0,\infty)$.
For the presentation concision,
we always assume that  $(X,d,\mu)$ is non-atomic [namely, $\mu(\{x\})=0$ for any
$x\in X$] and $\diam (X)=\fz$. It is known that $\diam (X)=\fz$ implies that
$\mu(X)=\fz$ (see, for example, \cite[Lemma 8.1]{AH13}).

The organization of this article is as follows.

Section \ref{pre} deals with approximations of the identity on $(X, d, \mu)$.
In Section \ref{ati}, we recall the notions of both the approximations of the identity with polynomial decay
(for short, ATI) and the space of test functions from \cite{HMY08}, and then state some often
used related estimates. In Section \ref{eti}, motivated by the wavelet theory established in \cite{AH13}, we
introduce a new kind of approximations of the identity with exponential decay (for short, $\exp$-ATI), and then
establish several equivalent characterizations of $\exp$-ATIs and discuss 
the relationship between $\exp$-ATIs and ATIs.

Section \ref{BDD} concerns the boundedness of Calder\'{o}n-Zygmund-type operators on spaces of test functions.
In Section \ref{s3.1}, we show that Calder\'{o}n-Zygmund operators whose kernels satisfying the second
difference regularity condition and some other size and regularity conditions
are bounded on spaces of test functions with cancellation, whose proof is long and  separated into two
subsections (see Sections \ref{size} and \ref{reg}). Section \ref{s3.4} deals with the boundedness of
Calder\'{o}n-Zygmund-type  operators on spaces of test functions without cancellation.
Compared with \cite[Theorem 2.18]{HMY08}, the condition used here is a little bit stronger,
but the proof is easier and enough for us to  build the Calder\'{o}n reproducing formulae in Sections
\ref{hcrf}, \ref{hdrf} and \ref{irf} below.

In Section \ref{hcrf}, we start our discussion by dividing the identity into a main operator $T_N$ and a
remainder $R_N$. In Section \ref{com}, we prove that compositions of two $\exp$-ATIs have properties similar
to those of an $\exp$-ATI.
With this and the conclusions in Section \ref{BDD}, we prove, in Section \ref{RN}, that the operator norms of
$R_N$ on both $L^2(X)$ and spaces of test functions can be small enough if
$N$ is sufficiently large. This ensures the existence of $T_N^{-1}$ which leads to the
homogeneous continuous Calder\'{o}n reproducing formulae in Section \ref{pr}.

In Section \ref{hdrf}, by a method similar to that used in \cite[Section 4]{HMY08}, we split the $k$-level
dyadic cubes $Q_\az^k$ into a sum of dyadic subcubes in level $k+j$.
The remainder for the discrete case contains $R_N$ and another part $G_N$ [see \eqref{eq:defGR} below].
In Section \ref{sec5.1}, we treat the boundedness of $G_N$ on both $L^2(X)$ and spaces of test functions, and
further establish homogeneous discrete Calder\'on reproducing formulae in Section \ref{pr2}.

In Section \ref{irf}, we obtain inhomogeneous continuous and discrete Calder\'on reproducing formulae, whose
proofs are similar to those of homogeneous ones presented in Sections \ref{hcrf} and \ref{hdrf}.

Let us make some conventions on notation. Throughout this article, we use $A_0$ to denote the coefficient
appearing in the \emph{quasi-triangular inequality} of $d$, the parameter $\omega$ means the \emph{upper
dimension constant} in \eqref{eq:doub}, and $\eta$ is defined to be the
\emph{smoothness index of wavelets} (see Theorem \ref{thm:wave} below). Denote by $\dz$ a
\emph{small positive number}, for example, $\dz\le (2A_0)^{-10}$, which comes from constructing the dyadic
cubes on $X$ (see Theorem \ref{thm:dys} below).
For any $p\in[1,\fz]$, we use $p'$ to denote its \emph{conjugate index}, namely, $1/p+1/p'=1$.
The \emph{symbol $C$} denotes a positive constant which is independent of the main parameters involved, but
may vary from line to line. We use $C_{(\az,\bz,\ldots)}$ to denote a positive constant depending on the
indicated parameters $\az$, $\bz,\ldots$. The \emph{symbol $A \ls B$} means that
$A \le CB$ for some positive constant $C$, while  $A \sim B$ is used as an abbreviation of $A \ls B \ls A$. We
also use $A\ls_{\az,\bz,\ldots}B$ to indicate that here the implicit positive constant depends on $\az$,
$\bz$, \ldots and, similarly, $A\sim_{\az,\bz,\ldots}B$.
For any (quasi)-Banach spaces $\mathcal X,\,\mathcal Y$  and any operator $T$, we use
$\|T\|_{\mathcal X\to\mathcal Y}$ to denote the \emph{operator norm} of $T$ from $\mathcal X$ to $\mathcal Y$.
For any $j,\ k\in\rr$, let $j\wedge k:=\min\{j,k\}$.

\section{Approximations of the identity}\label{pre}

This section concerns approximations of the identity on $(X, d, \mu)$.
In Section \ref{ati}, we recall the notions of both the approximations of the identity with polynomial decay
and the space of test functions from \cite{HMY08}, and then state some often used related estimates.
In Section \ref{eti}, we recall the dyadic systems established in \cite{HK12} and the wavelet systems built in
\cite{AH13}, which further motivate us to introduce a new kind of approximations of the identity with
exponential decay. Equivalence definitions and properties of $\exp$-ATIs are discussed in Section \ref{eti}.

\subsection{Approximations of the identity with polynomial decay}\label{ati}

For any $x,\ y\in X$ and $r\in(0,\infty)$, we adopt the notation
$$V(x,y):=\mu(B(x,d(x,y)))
\qquad \textup{and}\qquad V_r(x):=\mu(B(x,r)).
$$
We recall the following notion of approximations of the identity constructed in \cite{HMY08}.

\begin{definition}\label{def:ati}
Let $\bz\in(0,1]$ and $\gz\in(0,\fz)$. A sequence $\{P_k\}_{k\in\zz}$ of bounded linear integral operators
on $L^2(X)$ is called an \emph{approximation of the identity of order $(\bz,\gz)$} [for short,
$(\bz,\gz)$-ATI] if there exists a positive constant $C$ such that, for any $k\in\zz$, the kernel of operator
$P_k$, a function on $X\times X$, which is also denoted by $P_k$, satisfying
\begin{enumerate}
\item (the \emph{size condition}) for any $x,\ y\in X$,
\begin{equation}\label{eq:atisize}
|P_k(x,y)|\le C\frac 1{V_{\dz^k}(x)+V(x,y)}\lf[\frac{\dz^k}{\dz^k+d(x,y)}\r]^\gz;
\end{equation}
\item (the \emph{regularity condition}) if $d(x,x')\le (2A_0)^{-1}[\dz^k+d(x,y)]$, then
\begin{align}\label{eq:atisregx}
&|P_k(x,y)-P_k(x',y)|+|P_k(y,x)-P_k(y, x')|\\
&\quad\le C\lf[\frac{d(x,x')}{\dz^k+d(x,y)}\r]^\bz\frac 1{V_{\dz^k}(x)+V(x,y)}
\lf[\frac{\dz^k}{\dz^k+d(x,y)}\r]^\gz;\notag
\end{align}
\item (the \emph{second difference regularity condition}) if $d(x,x')\le (2A_0)^{-2}[\dz^k+d(x,y)]$ and
$d(y,y')\le (2A_0)^{-2}[\dz^k+d(x,y)]$, then
\begin{align}\label{eq:atidreg}
&|[P_k(x,y)-P_k(x',y)]-[P_k(x,y')-P_k(x',y')]|\\
&\quad \le C\lf[\frac{d(x,x')}{\dz^k+d(x,y)}\r]^\bz
\lf[\frac{d(y,y')}{\dz^k+d(x,y)}\r]^\bz\frac 1{V_{\dz^k}(x)+V(x,y)}\lf[\frac{\dz^k}{\dz^k+d(x,y)}\r]^\gz;\noz
\end{align}
\item for any $x,\ y\in X$,
\begin{equation*}
\int_X P_k(x,y')\,d\mu(y')=1=\int_X P_k(x',y)\,d\mu(x').
\end{equation*}
\end{enumerate}
\end{definition}

Let $L_\loc^1(X)$ be the space of all locally integrable functions on $X$.
Denote by $\CM$ the \emph{Hardy-Littlewood maximal operator} defined by setting, for any $f\in L_\loc^1(X)$ and
$x\in X$,
\begin{equation}\label{eq:defmax}
\CM(f)(x):=\sup_{r\in(0,\fz)}\frac 1{\mu(B(x,r))}\int_{B(x,r)} |f(y)|\,d\mu(y).
\end{equation}
For any $p\in(0,\fz]$, we use the \emph{symbol $L^p(X)$} to denote the set of all Lebesgue measurable functions
$f$ such that
$$
\|f\|_{L^p(X)}:=\lf[\int_X |f(x)|^p\,d\mu(x)\r]^{\frac 1p}<\fz
$$
with the usual modification made when $p=\fz$.

Now we list some basic properties of $(\bz,\gz)$-ATIs. For their proofs, see \cite[Proposition 2.7]{HMY08}.
\begin{proposition}\label{prop:basic}
Let $\{P_k\}_{k\in\zz}$ be an $(\bz,\gz)$-{\rm ATI}, with $\bz\in(0,1]$ and $\gz\in(0,\fz)$, and $p\in[1,\fz)$.
Then there exists a positive constant $C$ such that
\begin{enumerate}
\item for any $x,\ y\in X$,
$\int_X|P_k(x,y')|\,d\mu(y')\le C $ and $\int_X|P_k(x',y)|\,d\mu(x')\le C$;
\item for any $f\in L^1_\loc(X)$ and $x\in X$, $|P_kf(x)|\le C\CM(f)(x)$;
\item for any $f\in L^p(X)$, $\|P_kf\|_{L^p(X)}\le C\|f\|_{L^p(X)}$, which also holds true when $p=\fz$;
\item for any $f\in L^p(X)$, $\|f-P_kf\|_{L^p(X)}\to 0$ as $k\to\fz$.
\end{enumerate}
\end{proposition}

Let us recall the notions of both the space of test functions and the space of distributions, whose
following versions were originally given in \cite{HMY08} (see also \cite{HMY06}).

\begin{definition}[test functions]\label{def:test}
Let $x_1\in X$, $r\in(0,\fz)$, $\bz\in(0,1]$ and $\gz\in(0,\fz)$. A function $f$ on $X$ is called a
\emph{test function of type $(x_1,r,\bz,\gz)$}, denoted by $f\in\CG(x_1,r,\bz,\gz)$, if there exists a positive
constant $C$ such that
\begin{enumerate}
\item (the \emph{size condition}) for any $x\in X$,
\begin{equation}\label{eq:size}
|f(x)|\le C\frac{1}{V_r(x_1)+V(x_1,x)}\lf[\frac r{r+d(x_1,x)}\r]^\gz;
\end{equation}

\item (the \emph{regularity condition}) for any $x,\ y\in X$ satisfying $d(x,y)\le (2A_0)^{-1}[r+d(x_1,x)]$,
\begin{equation}\label{eq:reg}
|f(x)-f(y)|\le C\lf[\frac{d(x,y)}{r+d(x_1,x)}\r]^\bz
\frac{1}{V_r(x_1)+V(x_1,x)}\lf[\frac r{r+d(x_1,x)}\r]^\gz.
\end{equation}
\end{enumerate}
For any $f\in\CG(x_1,r,\bz,\gz)$, define the norm
$$
\|f\|_{\CG(x_1,r,\bz,\gz)}:=\inf\{C\in(0,\fz):\ \text{\eqref{eq:size} and \eqref{eq:reg} hold true}\}.
$$
Define
$$
\mathring{\CG}(x_1,r,\bz,\gz):=\lf\{f\in\CG(x_1,r,\bz,\gz):\ \int_X f(x)\,d\mu(x)=0\r\}
$$
equipped with the norm $\|\cdot\|_{\mathring{\CG}(x_1,r,\bz,\gz)}:=\|\cdot\|_{\CG(x_1,r,\bz,\gz)}$. Fixed
$x_0\in X$,
we denote $\CG(x_0,1,\bz,\gz)$ [resp., $\mathring{\CG}(x_0,1,\bz,\gz)$], simply, by $\CG(\bz,\gz)$ [resp.,
$\mathring{\CG}(\bz,\gz)$].
\end{definition}

Fix $x_0\in X$. For any $x\in X$ and $r\in(0,\fz)$, we find that $\CG(x,r,\bz,\gz)=\CG(x_0,1,\bz,\gz)$
with equivalent norms, but
the equivalent positive constants depend on $x$ and $r$.

Fix $\epsilon\in(0,1]$ and $\bz,\ \gz\in(0,\epsilon]$. Let $\CG^\epsilon_0(\bz,\gz)$ [resp.,
$\mathring\CG^\epsilon_0(\bz,\gz)$] be the completion of the set $\CG(\epsilon,\epsilon)$ [resp.,
$\mathring\CG(\ez,\ez)$] in $\CG(\bz,\gz)$. Furthermore, if $f\in\CG^\epsilon_0(\bz,\gz)$ [resp.,
$f\in\mathring\CG^\epsilon_0(\bz,\gz)$], we then let $\|f\|_{\CG^\epsilon_0(\bz,\gz)}:=\|f\|_{\CG(\bz,\gz)}$
[resp., $\|f\|_{\mathring\CG^\epsilon_0(\bz,\gz)}:=\|f\|_{\CG(\bz,\gz)}$].
The \emph{dual space} $(\CG^\epsilon_0(\bz,\gz))'$ [resp., $(\mathring{\CG}^\epsilon_0(\bz,\gz))'$] is defined
to be the set of all continuous linear functionals from $\CG^\epsilon_0(\bz,\gz)$ [resp.,
$(\mathring{\CG}^\epsilon_0(\bz,\gz))'$] to $\cc$, equipped with the weak-$*$ topology. The spaces
$(\CG^\epsilon_0(\bz,\gz))'$ and  $(\mathring{\CG}^\epsilon_0(\bz,\gz))'$ are called the \emph{spaces of
distributions}.

We conclude this subsection  with some estimates from \cite[Lemma~2.1]{HMY08}, which are proved by using
\eqref{eq:doub}.

\begin{lemma}\label{lem-add}
Let $\bz,\ \gz\in(0,\infty)$.
\begin{enumerate}
\item For any $x,\ y\in X$ and $r\in(0,\fz)$, $V(x,y)\sim  V(y,x)$ and
$$
V_r(x)+V_r(y)+V(x,y)\sim  V_r(x)+V(x,y)\sim  V_r(y)+V(x,y)\sim  \mu(B(x, r+d(x,y))),
$$
where the equivalent positive constants are independent of $x$, $y$ and $r$.
\item There exists a positive constant $C$ such that, for any $x_1\in X$ and $r\in(0,\infty)$,
$$\int_X\frac{1}{V_r(x_1)+V(x_1,y)}\lf[\frac r{r+d(x_1,y)}\r]^\gz\,d\mu(y)\le C.$$
\item There exists a positive constant $C$ such that, for any $x\in X$ and $R\in(0,\infty)$,
$$\int_{d(x,y)\le R}\frac 1{V(x,y)}\lf[\frac{d(x,y)}{R}\r]^\bz\,d\mu(y)\le C
\quad \textup{and}\quad
\int_{d(x, y)\ge R}\frac{1}{V(x,y)}\lf[\frac R{d(x,y)}\r]^\bz\,d\mu(y)\le C.$$
\item There exists a positive constant $C$ such that, for any $x_1\in X$ and $r,\ R\in(0,\infty)$,
$$
\int_{d(x, x_1)\ge R}\frac{1}{V_r(x_1)+V(x_1,x)}\lf[\frac r{r+d(x_1,x)}\r]^\gz\,d\mu(x)
\le C\lf(\frac{r}{r+R}\r)^\gz.
$$
\end{enumerate}
\end{lemma}

\subsection{Approximations of the identity with exponential decay}\label{eti}

The main aim of this section is to introduce the approximations of the identity with exponential decay.
Recall that
Hyt\"{o}nen and Kariema \cite{HK12} established a system of dyadic cubes, which is re-stated in the following
theorem.

\begin{theorem}[{\cite[Theorem 2.2]{HK12}}]\label{thm:dys}
Fix constants $0<c_0\le C_0<\fz$ and $\dz\in(0,1)$ such that $12A_0^3C_0\dz\le c_0$. Assume that
a set of points, $\{z_\az^k:\ k\in\zz,\ \az\in\CA_k\}\subset X$ with $\CA_k$ for any $k\in\zz$ being a set
of indices, has the following properties: for any $k\in\zz$,
\begin{enumerate}
\item[\rm (i)] $d(z_\az^k,z_\bz^k)\ge c_0\dz^k$ if $\az\neq\bz$;
\item[\rm (ii)] $\min_{\az\in\CA_k} d(x,z_\az^k)\le C_0\dz^k$ for any $x\in X$.
\end{enumerate}
Then there exists a family of sets, $\{Q_\az^k:\  k\in\zz,\ \az\in\CA_k\}$, satisfying
\begin{enumerate}
\item[\rm (iii)] for any $k\in\zz$, $\bigcup_{\az\in\CA_k} Q_\az^k=X$ and $\{ Q_\az^k:\;\; {\az\in\CA_k}\}$ is disjoint;
\item[\rm (iv)] if $k,\ l\in\zz$ and $l\ge k$, then either $Q_\bz^l\subset Q_\az^k$ or
$Q_\bz^l\cap Q_\az^k=\emptyset$;
\item[\rm (v)]  for any $k\in\zz$ and $\az\in\CA_k$, $B(z_\az^k,c_\natural\dz^k)\subset Q_\az^k\subset B(z_\az^k,C^\natural\dz^k)$,
where $c_\natural:=(3A_0^2)^{-1}c_0$, $C^\natural:=2A_0C_0$ and $z_\az^k$ is called ``the center'' of
$Q_\az^k$.
\end{enumerate}
\end{theorem}

Throughout this article, we keep the following notation from Theorem \ref{thm:dys}. For any $k\in\zz$, let
$$
\CX^k:=\{z_\az^k\}_{\az\in\CA_k},\qquad \CG_k:=\CA_{k+1}\setminus\CA_k\qquad \textup{and}\qquad
\CY^k:=\{z_\az^{k+1}\}_{\az\in\CG_k}=:\{y_\az^{k}\}_{\az\in\CG_k}.
$$

Recall that Auscher and Hyt\"{o}nen \cite{AH13} constructed the points $\{z_\az^k:\ k\in\zz,\ \az\in\CA_k\}$
satisfying Theorem \ref{thm:dys} and further built a system of random dyadic cubes having the properties (iii),
(iv) and (v) of Theorem \ref{thm:dys}.
With those random dyadic cubes, it was constructed in \cite{AH13} the spline functions and then
the wavelet functions on $(X, d,\mu)$.

\begin{theorem}[{\cite[Theorem 7.1 and Corollary 10.4]{AH13}}]\label{thm:wave}
There exist constants $a\in(0,1]$, $\eta\in(0,1)$, $C,\ \nu\in(0,\fz)$ and wavelet functions
$\{\psi_\az^k:\ k\in\zz,\ \az\in\CG_k\}$ satisfying that, for any $k\in\zz$ and $\az\in\CG_k$,
\begin{enumerate}
\item (the decay condition) for any $x\in X$,
$$
\lf|\psi_\az^k(x)\r|\le
\frac C{\sqrt{V_{\dz^k}(y_\az^k)}}\exp\lf\{-\nu\lf[\frac{d(x,y_\az^k)}{\dz^k}\r]^a\r\};
$$
\item (the H\"{o}lder-regularity condition) if $d(x,x')\le\dz^k$, then
$$
\lf|\psi_\az^k(x)-\psi_\az^k(x')\r|\le\frac C{\sqrt{V_{\dz^k}(y_\az^k)}}\lf[\frac{d(x,x')}{\dz^k}\r]^\eta
\exp\lf\{-\nu\lf[\frac{d(x,y_\az^k)}{\dz^k}\r]^a\r\};
$$
\item (the cancellation condition)
$$
\int_X \psi_\az^k(x)\,d\mu(x)=0.
$$\end{enumerate}
Moreover, the wavelets $\{\psi_\az^k\}_{k,\ \az}$ form an orthonormal basis of $L^2(X)$ and an unconditional
basis of $L^p(X)$ with any given $p\in(1,\fz)\setminus\{2\}$.
\end{theorem}

For any $x,\ y\in X$ and $k\in\zz$, if we let
\begin{equation*}
Q_k(x,y):=\sum_{\az\in\CG_k}\psi_\az^k(x)\psi_\az^k(y),
\end{equation*}
then $Q_k$ was proved to satisfy all conditions (i) through (v) of Definition \ref{def:eti} below; see
\cite[Lemma~10.1]{AH13} and \cite[Lemma~3.6 and (3.22)]{HLW16}.
This inspires us to introduce a new kind of \emph{approximations of the identity with exponential decay}.

\begin{definition}\label{def:eti}
A sequence $\{Q_k\}_{k\in\zz}$ of bounded linear integral operators on $L^2(X)$ is called an
\emph{approximation of the identity with exponential decay} (for short, $\exp$-ATI) if there exist constants
$C,\ \nu\in(0,\fz)$, $a\in(0,1]$ and $\eta\in(0,1)$ such that, for any $k\in\zz$, the kernel of operator $Q_k$,
a function on $X\times X$, which is still denoted by $Q_k$, satisfying
\begin{enumerate}
\item (the \emph{identity condition}) $\sum_{k=-\fz}^\fz Q_k=I$ in $L^2(X)$, where $I$ is the identity operator
on $L^2(X)$.
\item (the \emph{size condition}) for any $x,\ y\in X$,
\begin{align}\label{eq:etisize}
|Q_k(x,y)|&\le C\frac1{\sqrt{V_{\dz^k}(x)\,V_{\dz^k}(y)}}
\exp\lf\{-\nu\lf[\frac{d(x,y)}{\dz^k}\r]^a\r\}\\
&\quad\times\exp\lf\{-\nu\lf[\frac{\max\{d(x, \CY^k),\,d(y,\CY^k)\}}{\dz^k}\r]^a\r\};\noz
\end{align}
\item (the \emph{regularity condition}) if $d(x,x')\le\dz^k$, then
\begin{align}\label{eq:etiregx}
&|Q_k(x,y)-Q_k(x',y)|+|Q_k(y,x)-Q_k(y, x')|\\
&\quad\le C\lf[\frac{d(x,x')}{\dz^k}\r]^\eta
\frac1{\sqrt{V_{\dz^k}(x)\,V_{\dz^k}(y)}} \exp\lf\{-\nu\lf[\frac{d(x,y)}{\dz^k}\r]^a\r\}\noz\\
&\qquad\times\exp\lf\{-\nu\lf[\frac{\max\{d(x, \CY^k),\,d(y,\CY^k)\}}{\dz^k}\r]^a\r\};\noz
\end{align}
\item (the \emph{second difference regularity condition}) if $d(x,x')\le\dz^k$ and $d(y,y')\le\dz^k$, then
\begin{align}\label{eq:etidreg}
&|[Q_k(x,y)-Q_k(x',y)]-[Q_k(x,y')-Q_k(x',y')]|\\
&\quad \le C\lf[\frac{d(x,x')}{\dz^k}\r]^\eta\lf[\frac{d(y,y')}{\dz^k}\r]^\eta
\frac1{\sqrt{V_{\dz^k}(x)\,V_{\dz^k}(y)}} \exp\lf\{-\nu\lf[\frac{d(x,y)}{\dz^k}\r]^a\r\}\noz\\
&\qquad\times\exp\lf\{-\nu\lf[\frac{\max\{d(x, \CY^k),\,d(y,\CY^k)\}}{\dz^k}\r]^a\r\},\noz
\end{align}
\item (the \emph{cancellation condition}) for any $x,\ y\in X$,
\begin{equation*}
\int_X Q_k(x,y')\,d\mu(y')=0=\int_X Q_k(x',y)\,d\mu(x').
\end{equation*}
\end{enumerate}
\end{definition}
\begin{remark}
In Definition \ref{def:eti}, the existence of $\nu$, which is independent of $k$, $x$, $y$, $x'$ and $y'$, is
far more important than the quantity of $\nu$. Therefore, in the following discussion, we sometimes do not
distinguish the quantity of $\nu$ but focus on the existence of such a $\nu$, which is independent of $k$, $x$,
$y$, $x'$ and $y'$.
\end{remark}

\begin{remark}\label{rem:andef}
Notice that equivalent definitions of $\exp$-ATIs can be given in the following ways:
\begin{enumerate}
\item Due to \eqref{eq:doub},  the factor $\frac1{\sqrt{V_{\dz^k}(x)\,V_{\dz^k}(y)}}$ in
\eqref{eq:etisize}, \eqref{eq:etiregx} and \eqref{eq:etidreg}
can be replaced by $\frac1{V_{\dz^k}(x)}$ or $\frac1{V_{\dz^k}(y)}$, with the factor  $\exp\{-\nu[\frac{d(x,y)}{\dz^k}]^a\}$
replaced by $\exp\{-\nu'[\frac{d(x,y)}{\dz^k}]^a\}$ for some $\nu'\in(0,\nu)$.

\item Since $a\in(0,1]$, the factor $\max\{d(x, \CY^k),\,d(y,\CY^k)\}$ in \eqref{eq:etisize},
\eqref{eq:etiregx} and \eqref{eq:etidreg}
can be replaced by $d(x, \CY^k)$ or $d(y,\CY^k)$, again with the  factor  $\exp\{-\nu[\frac{d(x,y)}{\dz^k}]^a\}$
replaced by $\exp\{-\nu'[\frac{d(x,y)}{\dz^k}]^a\}$ for some $\nu'\in(0,\nu)$.

\item According to the proposition below, conditions (iii) and (iv) of Definition \ref{def:eti} can be equivalently replaced by

\begin{enumerate}
\item[\rm (iii)$'$]if $d(x,x')\le(2A_0)^{-1}[\dz^k+d(x,y)]$, then
\begin{align*}
&|Q_k(x,y)-Q_k(x',y)|+|Q_k(y,x)-Q_k(y, x')|\\
&\quad\le C\lf[\frac{d(x,x')}{\dz^k+d(x,y)}\r]^\eta
\frac1{\sqrt{V_{\dz^k}(x)\,V_{\dz^k}(y)}} \exp\lf\{-\nu'\lf[\frac{d(x,y)}{\dz^k}\r]^a\r\}\noz\\
&\qquad\times\exp\lf\{-\nu'\lf[\frac{\max\{d(x, \CY^k),\,d(y,\CY^k)\}}{\dz^k}\r]^a\r\};\noz
\end{align*}

\item[\rm (iv)$'$]
if $d(x,x')\le(2A_0)^{-2}[\dz^k+d(x,y)]$ and $d(y,y')\le(2A_0)^{-2}[\dz^k+d(x,y)]$, then
\begin{align*}
&|[Q_k(x,y)-Q_k(x',y)]-[Q_k(x,y')-Q_k(x',y')]|\\
&\quad \le C\lf[\frac{d(x,x')}{\dz^k+d(x,y)}\r]^\eta\lf[\frac{d(y,y')}{\dz^k+d(x,y)}\r]^\eta
\frac1{\sqrt{V_{\dz^k}(x)\,V_{\dz^k}(y)}} \exp\lf\{-\nu'\lf[\frac{d(x,y)}{\dz^k}\r]^a\r\}\noz\\
&\qquad\times\exp\lf\{-\nu'\lf[\frac{\max\{d(x, \CY^k),\,d(y,\CY^k)\}}{\dz^k}\r]^a\r\},\noz
\end{align*}
\end{enumerate}
where $\nu'\in(0,\nu)$.
\end{enumerate}
\end{remark}

\begin{proposition}\label{prop:etoa}
Conditions (iii) and (iv) of Definition \ref{def:eti} can be equivalently replaced by $(iii)'$ and $(iv)'$
of Remark \ref{rem:andef}(iii).
Consequently, any $\exp$-{\rm ATI} satisfies \eqref{eq:atisize}, \eqref{eq:atisregx} and \eqref{eq:atidreg} with
$\bz:=\eta$ and any given $\gz\in(0,\fz)$.
\end{proposition}

\begin{proof}
Let us first prove (iii)$'$.
It suffices to consider the case $\dz^k< d(x,x')\le(2A_0)^{-1}[\dz^k+d(x,y)]$.
In this case, we have $d(x,x')\le A_0^{-1}\dz^k+d(x',y)$ and hence $d(x,y)\le \dz^k+2A_0d(x',y)$,
which, combined with the inequality
\begin{align*}
A^a\le (A-B)^a+B^a,\qquad\forall\, A>B>0\; \textup{and}\; a\in(0,1],
\end{align*}
implies that
$$[d(x',y)]^a\ge (2A_0)^{-a}[d(x,y)-\dz^k]^a \ge (2A_0)^{-a}\lf\{[d(x,y)]^a-[\dz^k]^a\r\}.$$
Similarly, we have
\begin{align*}
[\max\{d(x, \CY^k),\,d(y,\CY^k)\}]^a
&\le [\max\{\dz^k+2A_0d(x', \CY^k),\,d(y,\CY^k)\}]^a\\
&\le (\dz^k)^a+[2A_0\max\{d(x', \CY^k),\,d(y,\CY^k)\}]^a.
\end{align*}
Thus, we obtain
\begin{align}\label{eq-add5}
&\lf[\frac{\dz^k+d(x,y)}{d(x,x')}\r]^\eta\exp\lf\{-\nu\lf[\frac{d(x',y)}{\dz^k}\r]^a\r\} \exp\lf\{-\nu\lf[\frac{\max\{d(x', \CY^k),\,d(y,\CY^k)\}}{\dz^k}\r]^a\r\}\\
&\quad \le \lf[\frac{\dz^k+d(x,y)}{\dz^k}\r]^\eta\exp\lf\{2(2A_0)^{-a}\nu\r\}\exp\lf\{-\nu\lf[\frac{d(x,y)}{2A_0\dz^k}\r]^a\r\} \notag\\
&\qquad\times\exp\lf\{-\nu\lf[\frac{\max\{d(x, \CY^k),\,d(y,\CY^k)\}}{2A_0\dz^k}\r]^a\r\}.\notag
\end{align}
From this and \eqref{eq:etisize}, we easily deduce (iii)$'$.

Now we prove (iv)$'$. Let us consider the case $\dz^k<d(x,x')\le(2A_0)^{-2}[\dz^k+d(x,y)]$ and
$\dz^k<d(y,y')\le(2A_0)^{-2}[\dz^k+d(x,y)]$ as an example to explain the idea. Write
\begin{align*}
|[Q_k(x,y)-Q_k(x',y)]-[Q_k(x,y')-Q_k(x',y')]|
&\le |Q_k(x,y)-Q_k(x',y)|+|Q_k(x,y')-Q_k(x',y')|.
\end{align*}
For the first term, since $d(x,x')\le(2A_0)^{-1}[\dz^k+d(x,y)]$, then, from (iii)$'$, it follows that
\begin{align}\label{eq-add4}
|Q_k(x,y)-Q_k(x',y)|&\ls\lf[\frac{d(x,x')}{\dz^k+d(x,y)}\r]^\eta\lf[\frac{d(y,y')}{\dz^k}\r]^\eta
\frac1{\sqrt{V_{\dz^k}(x)\,V_{\dz^k}(y)}} \exp\lf\{-\nu'\lf[\frac{d(x,y)}{\dz^k}\r]^a\r\}\\
&\quad\times\exp\lf\{-\nu'\lf[\frac{\max\{d(x, \CY^k),\,d(y,\CY^k)\}}{\dz^k}\r]^a\r\}\notag\\
&\ls\lf[\frac{d(x,x')}{\dz^k+d(x,y)}\r]^\eta\lf[\frac{d(y,y')}{\dz^k+d(x,y)}\r]^\eta
\frac1{\sqrt{V_{\dz^k}(x)\,V_{\dz^k}(y)}}\noz\\
&\quad\times \exp\lf\{-\nu''\lf[\frac{d(x,y)}{\dz^k}\r]^a\r\}
\exp\lf\{-\nu'\lf[\frac{\max\{d(x, \CY^k),\,d(y,\CY^k)\}}{\dz^k}\r]^a\r\},\notag
\end{align}
where  $\nu''\in(0,\nu')\subset(0,\nu)$.

For the  term $|Q_k(x,y')-Q_k(x',y')|$, notice that $d(y,y')\le(2A_0)^{-2}[\dz^k+d(x,y)]$. Thus,
$$
d(x,y)\le A_0[d(y,y')+d(x,y)]\le [\dz^k+d(x,y)]/2+A_0d(x,y'),
$$
which implies that $d(x,y)\le\dz^k+2A_0d(x,y')$ and hence
$$
d(x,x')\le(2A_0)^{-2}[\dz^k+d(x,y)]\le(2A_0)^{-2}[2\dz^k+2A_0d(x,y')]\le(2A_0)^{-1}[\dz^k+d(x,y')].
$$
Therefore, using these, \eqref{eq-add4} with $y$ therein replaced by $y'$ and also \eqref{eq-add5}, we conclude
that
\begin{align*}
|Q_k(x,y')-Q_k(x',y')|&\ls\lf[\frac{d(x,x')}{\dz^k+d(x,y')}\r]^\eta\lf[\frac{d(y,y')}{\dz^k+d(x,y')}\r]^\eta
\frac1{\sqrt{V_{\dz^k}(x)\,V_{\dz^k}(y')}} \exp\lf\{-\nu''\lf[\frac{d(x,y')}{\dz^k}\r]^a\r\}\noz\\
&\quad\times\exp\lf\{-\nu'\lf[\frac{\max\{d(x, \CY^k),\,d(y',\CY^k)\}}{\dz^k}\r]^a\r\}\\
&\ls\lf[\frac{d(x,x')}{\dz^k+d(x,y)}\r]^\eta\lf[\frac{d(y,y')}{\dz^k+d(x,y)}\r]^\eta
\frac1{\sqrt{V_{\dz^k}(x)\,V_{\dz^k}(y)}}\exp\lf\{-\frac{\nu''}2\lf[\frac{d(x,y')}{2A_0\dz^k}\r]^a\r\}\noz\\
&\quad\times\exp\lf\{-\nu'\lf[\frac{\max\{d(x, \CY^k),\,d(y',\CY^k)\}}{2A_0\dz^k}\r]^a\r\}.
\end{align*}
The other cases are similar, the details being omitted. This proves (iv)$'$.

Finally, we fix $\gamma\in(0,\fz)$. Notice that the doubling condition \eqref{eq:doub} implies that
$$
\frac{V_{\dz^k}(x)+V(x,y)}{\sqrt{V_{\dz^k}(x)\,V_{\dz^k}(y)}}
\ls \lf[\frac{\dz^k+d(x,y)}{\dz^k}\r]^{\omega}
$$
with $\omega$ as in \eqref{eq:doub}. From this, \ref{def:eti}(ii) and the above (iii)$'$
and (iv)$'$, we deduce that
$$
\lf[\frac{\dz^k+d(x,y)}{\dz^k}\r]^{\omega}\exp\lf\{-\nu'\lf[\frac{d(x,y)}{\dz^k}\r]^a\r\}
\ls \lf[\frac{\dz^k}{\dz^k+d(x,y)}\r]^{\gamma}.$$
This implies that any $\exp$-ATI satisfies \eqref{eq:atisize}, \eqref{eq:atisregx} and \eqref{eq:atidreg},
which completes the proof of Proposition \ref{prop:etoa}.
\end{proof}

\begin{remark}\label{rem-add}
Observe that, even if $\{P_k\}_{k=-\fz}^\fz$ is an $(\eta,\gamma)$-ATI with $\eta$ as in Definition
\ref{def:eti} and $\gz\in(0,\fz)$ as in Definition \ref{def:ati},
the family $\{P_k-P_{k-1}\}_{k=-\fz}^\fz$ may not be an $\exp$-ATI due to the additional exponential factor
$$\exp\{-\nu[\max\{d(x, \CY^k)),\,d(x,\CY^k)\}/\dz^k]^a\}$$ in Definition \ref{def:eti}.
This exponential factor is important because it is a substitute of the reverse doubling condition;
see \cite[Lemma~8.3]{AH13} or Lemma \ref{lem:expsum} below.
\end{remark}


\section{Calder\'on-Zygmund-type operators on spaces of test functions}\label{BDD}

For any $s\in(0,\eta]$, the \emph{H\"older space} $C^s(X)$ consists of all $f\in L^\fz(X)$ such that
$$
\|f\|_{\dot{C}^s(X)}:=\sup_{x\neq y}\frac{|f(x)-f(y)|}{[d(x,y)]^s}<\fz.
$$
Denote by the \emph{symbol $C^s_b(X)$} the space of all functions in $C^s(X)$ with bounded support, equipped
with the strict inductive limit topology. Write the dual space of $C^s_b(X)$ as $(C^s_b(X))'$, equipped with
the weak-$*$ topology.

For any $s\in(0,\eta]$,  a function $K:\; (X\times X)\setminus\{(x,x):\ x\in X\}\to\cc$ is
called an \emph{$s$-Calder\'on-Zygmund kernel} if there exists a positive constant $C_T$ such that
\begin{enumerate}
\item for any $x,\ y\in X$ with $x\neq y$,
\begin{equation}\label{eq:Ksize}
|K(x,y)|\le C_T\frac 1{V(x,y)};
\end{equation}
\item if $d(x,x')\le(2A_0)^{-1}d(x,y)$ with $x\neq y$, then
\begin{equation}\label{eq:Kreg}
|K(x,y)-K(x',y)|+|K(y,x)-K(y,x')|\le C_T\lf[\frac{d(x,x')}{d(x,y)}\r]^{s}\frac 1{V(x,y)}.
\end{equation}
\end{enumerate}
A linear operator $T:\ C^s_b(X)\rightarrow (C^s_b(X))'$
is called an \emph{$s$-Calder\'on-Zygmund
operator} if $T$ can be extended to a bounded linear operator on $L^2(X)$ and if there exists
an $s$-Calder\'on-Zygmund kernel $K$ such that, for any $f\in C^s_b(X)$,
$$
Tf(x):=  \int_{X}K(x,y)f(y)\,d\mu(y),\quad \forall\; x\notin\supp f.
$$
For more studies on Calder\'{o}n-Zygmund operators over spaces of homogeneous type, we refer the reader to
\cite{CW71,HMY08,AH13,dh09}.

The main aim of this section is to prove the following boundedness of a Calder\'{o}n-Zygmund-type operator on spaces of test functions with or without cancellation.

\subsection{Boundedness of Calder\'{o}n-Zygmund-type operators on $\GO{\bz,\gz}$}\label{s3.1}

The main aim of this section is to prove the following boundedness of Calder\'{o}n-Zygmund-type operators on $\GO{x_1,r,\bz,\gz}$.

\begin{theorem}\label{thm:Kbdd}
Suppose that $s\in(0,\eta)$, $\bz,\ \gz\in(0,s)$ and $T$ is an $s$-Calder\'{o}n-Zygmund-operator
with its kernel $K$ satisfying the following additional conditions:
\begin{enumerate}
\item[{\rm (a)}] there exsits a positive constant $C_T$ such that, if $d(x,x')\le(2A_0)^{-2}d(x,y)$ and
$d(y,y')\le(2A_0)^{-2}d(x,y)$ with $x\neq y$, then
\begin{equation}\label{eq:Kdreg}
|[K(x,y)-K(x',y)]-[K(x,y')-K(x',y')]|\le C_T\lf[\frac{d(x,x')}{d(x,y)}\r]^{s}
\lf[\frac{d(y,y')}{d(x,y)}\r]^{s}\frac 1{V(x,y)};
\end{equation}
\item[{\rm (b)}] for any $f\in C^\bz(X)$ and  $x\in X$,
\begin{equation*}
Tf(x)=\int_X K(x,y)f(y)\,d\mu(y);
\end{equation*}
\item[{\rm (c)}] there exists a constant $c_0$ such that, for any $x\in X$.
\begin{equation}\label{eq:Kcany}
\int_X K(x,y)\,d\mu(y)=c_0.
\end{equation}
\end{enumerate}
Then there exists a positive constant $C_{(\bz,\gz)}$ such that, for any $f\in\mathring{\CG}(x_1,r,\bz,\gz)$
with $x_1\in X$ and $r\in(0,\fz)$,
\begin{equation*}
\|Tf\|_{\CG(x_1,r,\bz,\gz)}\le C_{(\bz,\gz)}\lf[C_T+\|T\|_{L^2(X)\to L^2(X)}+|c_0|\r]\|f\|_{\CG(x_1,r,\bz,\gz)},
\end{equation*}
where $C_{(\bz,\gz)}$ depends on $\bz$, $\gz$ but not on $x_1$, $r$ or $c_0$.
\end{theorem}

\begin{remark}\label{rem:bdd}
Theorem \ref{thm:Kbdd} remains true if $d$ is a metric, the condition (b) is removed and the condition (c) is
replaced by the assumption that $T1=0$ in $(\mathring{C}^\bz_b(X))'$, which means $\langle T1,f\rangle=0$ for
any
$$
f\in\mathring{C}^\bz_b(X):=\lf\{f\in C^\bz_b(X):\ \int_X f(x)\,d\mu(x)=0\r\};
$$
see \cite[Theorem 2.18]{HMY08}. It is still unknown whether or not \cite[Theorem 2.18]{HMY08} still holds true
if $d$ is just a quasi-metric.
\end{remark}

To prove Theorem \ref{thm:Kbdd}, we need the following two lemmas.

\begin{lemma}[{\cite[Corollary 4.2]{AH13}}]\label{lem:scf}
Suppose $x\in X$ and $r\in(0,\fz)$. Then there exists a function $f$ such that
$\chi_{B(x,r)}\le f\le\chi_{B(x,2A_0r)}$ and $\|f\|_{\dot{C}^\eta(X)}\le Cr^{-\eta}$, where $C$ is a positive
constant independent of $x$ and $r$.
\end{lemma}

\begin{lemma}\label{lem:3.24}
Let $T$ be a Calder\'{o}n-Zygmund operator with its kernel $K$ satisfying that, for any $x\in X$,
$$
\int_X K(x,y)\,d\mu(y)=c_0,
$$
where $c_0\in\cc$ is a constant. Assume that $x_0\in X$, $t\in(0,\fz)$ and $\thz$ is a function satisfying
$|\thz|\le\chi_{B(x_0,t)}$ and $\|\thz\|_{\dot{C}^\eta(X)}\le t^{-\eta}$. Then there exists a positive constant
$C$, independent of $x_0$, $t$ and $c_0$, such that, for any $x\in X$,
\begin{equation}\label{eq:claim}
|T\thz(x)|\le C\lf[C_T+\|T\|_{L^2(X)\to L^2(X)}+|c_0|\r].
\end{equation}
\end{lemma}

\begin{proof}
With the condition $\int_X K(x,y)\,d\mu(y)=c_0$ for any $x\in X$ replaced by $T1=0$ in
$(\mathring{C}^{\bz}_b(X))'$, the proof of this lemma was essentially given in \cite[(3.24)]{HLW16} (see also
\cite[Lemma~2.15]{HMY08} on the setting of RD-spaces).
Following the proof of \cite[(3.24)]{HLW16}, we observe that here the condition
$$
\int_X K(x,y)\,d\mu(y)=c_0,\quad\forall\, x\in X
$$
leads to the additional term $|c_0|$ in \eqref{eq:claim}, the details being omitted. This finishes the proof of
Lemma \ref{lem:3.24}.
\end{proof}
Using Lemmas \ref{lem:scf} and \ref{lem:3.24}, we now show Theorem \ref{thm:Kbdd}.
\begin{proof}[Proof of Theorem \ref{thm:Kbdd}]
Let all the notation be as in Theorem \ref{thm:Kbdd}. Without loss of generality, we may as well assume that
$\|f\|_{\CG(x_1,r,\bz,\gz)}=1$ and
$\max\{C_T,\,\|T\|_{L^2(X)\to L^2(X)},\, |c_0|\}\le1$. To show Theorem \ref{thm:Kbdd}, it suffices to prove
that, for any $x\in X$,
\begin{equation}\label{eq:Tfsize}
|Tf(x)|\ls\frac 1{V_r(x_1)+V(x_1,x)}\lf[\frac r{r+d(x_1,x)}\r]^\gz
\end{equation}
and, for any $x,\ x'\in X$ satisfying $d(x,x')\le (2A_0)^{-1}[r+d(x_1,x)]$,
\begin{equation}\label{eq:Tfreg}
|Tf(x)-Tf(x')|\ls\lf[\frac {d(x,x')}{r+d(x_1,x)}\r]^\bz\frac 1{V_r(x_1)+V(x_1,x)}
\lf[\frac r{r+d(x_1,x)}\r]^\gz.
\end{equation}
We conclude the proof of Theorem \ref{thm:Kbdd} by proving \eqref{eq:Tfsize} and \eqref{eq:Tfreg} in Sections
\ref{size} and \ref{reg} below, respectively.
\end{proof}

\subsection{Proof of the size estimate \eqref{eq:Tfsize}}\label{size}

\begin{proof}[Proof of \eqref{eq:Tfsize}]
The proof is given by considering  two cases: $d(x_1,x)<2A_0r$ and $d(x_1,x)\ge 2A_0r$.

{\it Case 1) $d(x_1,x)<2A_0r$}. In this case, the right-hand side of \eqref{eq:Tfsize} is comparable with
$\frac1{V_r(x_1)}\sim \frac1{V_r(x)}$. By Lemma \ref{lem:scf}, there exists a function $\theta$ such that
$\chi_{B(x_1,4A_0^2r)}\le\thz\le\chi_{B(x_1,8A_0^3r)}$ and $\|\thz\|_{\dot{C}^\eta(X)}\ls r^{-\eta}$. Let
$\psi:=1-\thz$. By the fact that $f\in \CG(x_1,r,\bz,\gz)$ belongs to $C^\beta(X)$, we can use (b) to write,
for any $x\in X$,
\begin{align*}
Tf(x)&=\int_X K(x,y)f(y)\,d\mu(y)\\
&=\int_X K(x,y)[f(y)-f(x)]\thz(y)\,d\mu(y)+\int_X K(x,y)f(y)\psi(y)\,d\mu(y)
+f(x)\int_X K(x,y)\thz(y)\,d\mu(y)\\
&=:\RZ_1+\RZ_2+\RZ_3
\end{align*}

From the fact $\supp\theta\subset B(x_1,8A_0^3r)$ and the assumption $d(x_1,x)<r$, we deduce that any $y$ with
$\theta(y)\neq0$ satisfies $d(x,y)\le 10A_0^4r$.
Then, applying \eqref{eq:Ksize}, the H\"{o}lder regularity condition of $f$  and Lemma \ref{lem-add}(iii), the
size condition of $f$ and \eqref{eq:doub}, we conclude that
\begin{align*}
|\RZ_1|&\le\int_{d(x,y)\le(2A_0)^{-1}[r+d(x_1,x)]}|K(x,y)||f(y)-f(x)|\,d\mu(y)\\
&\quad+\int_{\gfz{d(x,y)>(2A_0)^{-1}[r+d(x_0,x)]}{d(x,y)<10A_0^4r}}|K(x,y)|[|f(y)|+|f(x)|]\,d\mu(y)\\
&\ls \frac 1{V_r(x_1)}\int_{d(x,y)\le(2A_0)^{-1}[r+d(x_1,x)]}\frac 1{V(x,y)}\lf[\frac{d(x,y)}{r+d(x_1,x)}\r]^\bz
\,d\mu(y)\\
&\quad+\frac 1{V_r(x_1)}\int_{(2A_0)^{-1}r<d(x,y)<10A_0^4r}\frac 1{V(x,y)}\,d\mu(y)
\ls \frac 1{V_r(x_1)},
\end{align*}
as desired.

Notice that $\psi(y)\neq 0$ implies that $d(x_1,y)\ge 4A_0^2r$, which further implies $d(x,y)\ge 2A_0r$.
By this, \eqref{eq:Ksize} and Lemma \ref{lem-add}(ii), we know that
$$
|\RZ_2|\ls\int_{d(x_1,y)\ge 4A_0^2r}\frac 1{V(x,y)}|f(y)|\,d\mu(y)
\ls\frac 1{V_r(x)}\int_X |f(y)|\,d\mu(y)\ls \frac 1{V_r(x_1)}.
$$
From \eqref{eq:claim} and the size condition of $f$, we deduce that
$
|\RZ_3|\ls|f(x)|\ls\frac 1{V_r(x_1)}.
$
Thus, \eqref{eq:Tfsize} holds true when $d(x_1,x)<2A_0r$.

{\it Case 2) $R:=d(x_1,x)\ge 2A_0r$}. In this case, by Lemma \ref{lem:scf}, we can choose $u_1$ and $u_2$ such
that $\chi_{B(x,(2A_0)^{-3}R)}\le u_1\le\chi_{B(x,(2A_0)^{-2}R)}$,
$\|u_1\|_{\dot{C}^\eta(X)}\ls R^{-\eta}$, $\chi_{B(x_1,(2A_0)^{-3}R)}\le u_2\le\chi_{B(x_1,(2A_0)^{-2}R)}$ and
$\|u_2\|_{\dot{C}^\eta(X)}\ls R^{-\eta}$.
Let $u_3:=1-u_1-u_2$.
Write $f=f_1+f_2+f_3$, where $f_i:=fu_i$ for any $i\in\{1,2,3\}$. We claim that,
for any $y,\ y'\in X$,
\begin{gather}
|f_1(y)|\ls\frac 1{V_R(x_1)}\lf(\frac rR\r)^\gz;\label{eq:1}\\
|f_1(y)-f_1(y')|\ls\lf[\frac{d(y,y')}R\r]^\bz\frac 1{V_R(x_1)}\lf(\frac rR\r)^\gz;\label{eq:2}\\
|f_3(y)|\ls\frac 1{V_r(x_1)+V(x_1,y)}\lf[\frac r{r+d(x_1,y)}\r]^\gz
\chi_{\{y\in X:\ d(x,y)\ge(2A_0)^{-3}R\}}(y);\label{eq:3}\\
\int_X |f_3(y)|\,d\mu(y)\ls \lf(\frac rR\r)^\gz;\label{eq:4}\\
\lf|\int_X f_2(y)\,d\mu(y)\r|\ls \lf(\frac rR\r)^\gz.\label{eq:5}
\end{gather}

We now prove \eqref{eq:1} through \eqref{eq:5}. If $y\in\supp u_1\subset B(x,(2A_0)^{-2}R)$, then
$$
R=d(x_1,x)\le A_0[d(x_1,y)+d(y,x)]<A_0d(x_1,y)+d(x_1,x)/2,
$$
which implies that $R\le 2A_0 d(x_1,y)$. From this and the size condition of $f$, it follows that
$$
|f_1(y)|\le |f(y)|\chi_{B(x,(2A_0)^{-2}R)}(y)\le\frac 1{V_r(x_1)+V(x_1,y)}\lf[\frac r{r+d(x_1,y)}\r]^\gz
\ls\frac 1{V_R(x_1)}\lf(\frac rR\r)^\gz,
$$
which proves \eqref{eq:1}.

Next we prove \eqref{eq:2}.
By $\supp u_1\subset B(x,(2A_0)^{-2}R)$, we know that $|f_1(y)-f_1(y')|\neq0$ implies either
$y\in B(x,(2A_0)^{-2}R)$ or $y'\in B(x,(2A_0)^{-2}R)$.
Due to the symmetry, we may as well assume that $y\in B(x,(2A_0)^{-2}R)$. Then, like the discussion in the
estimation of \eqref{eq:1}, we still have $R\le 2A_0 d(x_1,y)$. If $d(y,y')> (2A_0)^{-1}[r+d(x_1,y)]$, then, by
\eqref{eq:1} and $R\le 2A_0 d(x_1,y)$, we obtain
$$
|f_1(y)-f_1(y')|\le|f(y)|+|f(y')|\ls\lf[\frac{d(y,y')}{r+d(x_1,y)}\r]^\bz\frac 1{V_R(x_1)}
\lf(\frac rR\r)^\gz\ls\lf[\frac{d(y,y')}R\r]^\bz\frac 1{V_R(x_1)}\lf(\frac rR\r)^\gz.
$$
If $d(y,y')\le (2A_0)^{-1}[r+d(x_1,y)]$, then, by the definition of $u_1$ and the regularity condition of $f$,
we find that
\begin{align*}
|f_1(y)-f_1(y')|&\le|f(y)||u_1(y)-u_1(y')|+|u_1(y')||f(y)-f(y')|\\
&\ls\frac 1{V_1(x_1 )+V(x_1 ,y)}\lf[\frac r{r+d(x_1,y)}\r]^\gz
\min\lf\{1,\;\lf[\frac{d(y,y')}R\r]^\eta\r\}\\
&\quad+\lf[\frac{d(y,y')}{r+d(x_1,y)}\r]^\bz\frac 1{V_r(x_1)+V(x_1,y)}
	\lf[\frac r{r+d(x_1 ,y)}\r]^\gz\\
	&\ls\lf[\frac{d(y,y')}R\r]^\bz\frac 1{V_R(x_1 )}\lf(\frac rR\r)^\gz.
	\end{align*}
This proves  \eqref{eq:2}.

Since $B(x,(2A_0)^{-2}R)\cap B(x_1,(2A_0)^{-2}R)=\emptyset$, it follows that
$0\le u_3\le 1$ and
$
u_3(x)=0$ if $x\in B(x,(2A_0)^{-3}R)\cup B(x_1,(2A_0)^{-3}R)$. Combining this with the size condition of $f$ implies \eqref{eq:3}.
From \eqref{eq:3} and Lemma \ref{lem-add}(iv), it follows directly \eqref{eq:4}.

To prove \eqref{eq:5}, since  $f\in\mathring{\CG}(x_1,r,\bz,\gz)$, we deduce that $\int_X f(y)\,d\mu(y)=0$,
which, together with  \eqref{eq:4}, \eqref{eq:1} and $\supp f_1\subset B(x,(2A_0)^{-2}R)$, implies that
$$
\lf|\int_X f_2(y)\,d\mu(y)\r|=\lf|\int_X [f_1(y)+f_3(y)]\,d\mu(y)\r|\ls
\int_{B(x,(2A_0)^{-2}R)}|f_1(y)|\,d\mu(y)+\lf(\frac rR\r)^\gz\ls \lf(\frac rR\r)^\gz.
$$
Therefore, we complete the proofs of \eqref{eq:1} through \eqref{eq:5}.
	
Now we continue the proof of \eqref{eq:Tfsize} in the case $R=d(x_1, x)\ge 2A_0r$. Notice that, in this case,
the right-hand side of \eqref{eq:Tfsize} is comparable with $\frac 1{V_R(x_1)}(\frac rR)^\gz$.
By Lemma \ref{lem:scf}, we find a function $u$ such that
$\chi_{B(x,(2A_0)^{-2}R)}\le u\le\chi_{B(x, (2A_0)^{-1}R)}$ and $\|u\|_{\dot{C}^\eta(X)}\ls R^{-\eta}$. Then $f_1=uf_1$ and
\begin{align*}
Tf_1(x)&=\int_ X K(x,y)u(y)f_1(y)\,d\mu(y)\\
&=\int_X K(x,y)u(y)[f_1(y)-f_1(x)]\,d\mu(y)+f_1(x)\int_X K(x,y)u(y)\,d\mu(y)=:\Gamma_1+\Gamma_2.
\end{align*}
By Lemma \ref{lem:3.24} and \eqref{eq:1}, we conclude that
$|\Gamma_2|\ls|f_1(x)|\ls\frac 1{V_R(x_1)}(\frac rR)^\gz$.
From the fact that $\supp u\subset B(x,(2A_0)^{-1}R)$, \eqref{eq:Ksize}, \eqref{eq:2} and Lemma
\ref{lem-add}(iii), we deduce that
\begin{align*}
|\Gamma_1|&\le \int_{d(x,y)<(2A_0)^{-1}R}|K(x,y)||f_1(y)-f_1(x)|\,d\mu(y)\\
&\ls\frac 1{V_R(x_1)}\lf(\frac rR\r)^\gz
\int_{d(x,y)<(2A_0)^{-1}R}\frac 1{V(x,y)}\lf[\frac{d(x,y)}R\r]^\bz\,d\mu(y)\ls\frac 1{V_R(x_1)}\lf(\frac rR\r)^\gz.
\end{align*}
Combining the estimates of $\Gamma_1$ and $\Gamma_2$, we obtain the desired estimate of $Tf_1(x)$.

Now we estimate $Tf_2(x)$. Indeed, from \eqref{eq:Ksize}, \eqref{eq:Kreg}, \eqref{eq:5}, $|f_2|\le |f|$, the
size condition of $f$ and Lemma \ref{lem-add}(iii),  we deduce that
\begin{align*}
|Tf_2(x)|&\le\lf|\int_X [K(x,y)-K(x,x_1)]f_2(y)\,d\mu(y)\r|+|K(x,x_1)|\lf|\int_X f_2(y)\,d\mu(y)\r|\\
&\ls \int_{d(x_1,y)<(2A_0)^{-2}R} |K(x,y)-K(x,x_1)||f_2(y)|\,d\mu(y)+\frac 1{V(x,x_1)}\lf(\frac rR\r)^\gz\\
&\ls\int_{d(x_1,y)<(2A_0)^{-2}R}\lf[\frac{d(x_1,y)}R\r]^{s}\frac 1{V(x,x_1)}\lf(\frac r{r+d(x_1,y)}\r)^\gz
\frac 1{V(x_1,y)}\,d\mu(y)+\frac 1{V_R(x_1)}\lf(\frac rR\r)^\gz\\
&\ls\frac 1{V_R(x_1)}\lf(\frac rR\r)^\gz\lf\{\int_{d(x_1,y)<(2A_0)^{-2}R}\lf[\frac{d(x_1,y)}R\r]^{s-\gz}
\frac 1{V(x_1,y)}\,d\mu(y)+1\r\}\sim\frac 1{V_R(x_1)}\lf(\frac rR\r)^\gz.
\end{align*}

Finally, by \eqref{eq:3}, \eqref{eq:Ksize} and \eqref{eq:4}, we conclude that
\begin{align*}
|Tf_3(x)|&\le\int_{d(x,y)\ge (2A_0)^{-3}R}|K(x,y)||f_3(y)|\,d\mu(y)\\
&
\ls\int_{d(x,y)\ge (2A_0)^{-3}R} \frac 1{V(x,y)}|f_3(y)|\,d\mu(y)\ls\frac 1{V_R(x_1)}\lf(\frac rR\r)^\gz.
\end{align*}

Since $Tf=Tf_1+Tf_2+Tf_3$, we summarize the estimates of $\{Tf_i(x)\}_{i=1}^3$ and then complete
the proof of \eqref{eq:Tfsize} in the case $R=d(x_1, x)\ge 2A_0r$. This finishes the proof of
\eqref{eq:Tfsize}.
\end{proof}


\subsection{Proof of the regularity estimate \eqref{eq:Tfreg}}\label{reg}

\begin{proof}[Proof of \eqref{eq:Tfreg}]
Assume that $\rho:=d(x,x')\le(2A_0)^{-1}[r+d(x_1,x)]$, we show  \eqref{eq:Tfreg} by considering the following three cases.

{\it Case 1) $\rho\ge (3A_0)^{-4}[r+d(x_1,x)]$.} In this case, we have
$d(x,x')\sim r+d(x_1,x)\sim r+d(x_1,x')$,  so that the size condition \eqref{eq:Tfsize} and the doubling
condition \eqref{eq:doub} directly imply \eqref{eq:Tfreg}.

{\it Case 2) $\rho<(3A_0)^{-4}[r+d(x_1,x)]$ and $R:=d(x_1,x)<r$.}
In this case, notice that $d(x_1, x)<r$ implies that the right-hand side of \eqref{eq:Tfsize} is comparable
with $(\frac\rho{r})^\bz\frac 1{V_r(x_1)}$.
By Lemma \ref{lem:scf}, we choose a function $\vz$ such that
$\chi_{B(x,(2A_0)^2\rho)}\le\vz\le\chi_{B(x,(2A_0)^3\rho)}$ and $\|\vz\|_{\dot{C}^\eta(X)}\ls\rho^{-\eta}$.
Let $\zeta:=1-\vz$. Using (b), we write
\begin{align*}
Tf(x)&=\int_X K(x,y)[f(y)-f(x)]\vz(y)\,d\mu(y)+\lf\{\int_X K(x,y)f(y)\zeta(y)\,d\mu(y)\r.\\
&\quad\lf.+f(x)\int_X K(x,y)\vz(y)\,d\mu(y)\r\}=:\Gamma_1(x)+\Gamma_2(x).
\end{align*}

Let us first estimate $|\Gamma_1(x)-\Gamma_1(x')|$. Noticing that $\supp\varphi\subset B(x,(2A_0)^3\rho)$, we
then have
\begin{align*}
|\Gamma_1(x)|&\le
\int_{\gfz{d(x,y)<8A_0^3\rho}{d(x,y)\le(2A_0)^{-1}[r+d(x_1,x)]}}|K(x,y)||f(y)-f(x)|\,d\mu(y)\\
&\quad+\int_{\gfz{d(x,y)<8A_0^3\rho}{d(x,y)>(2A_0)^{-1}[r+d(x_1,x)]}
}|K(x,y)[|f(y)|+|f(x)|]\,d\mu(y)=:\Gamma_{1,1}+\Gamma_{1,2}.
\end{align*}
By \eqref{eq:Ksize}, the  regularity condition of $f$ and Lemma \ref{lem-add}(ii), we conclude that
\begin{align*}
\Gamma_{1,1}&\ls\int_{d(x,y)<8A_0^3\rho}\frac 1{V(x,y)}\lf[\frac{d(x,y)}{r+d(x_1,x)}\r]^\bz
\frac 1{V_r(x_1)}\,d\mu(y)\ls \lf(\frac\rho{r}\r)^\bz\frac 1{V_r(x_1)}.
\end{align*}
Likewise, using the size condition of $f$,  we obtain
\begin{align*}
\Gamma_{1,2}&\ls\frac 1{V_r(x_1)}
\int_{d(x,y)<8A_0^3\rho}\frac 1{V(x,y)}\lf[\frac{d(x,y)}{r+d(x_1,x)}\r]^\bz\,d\mu(y)\ls \lf(\frac\rho{r}\r)^\bz\frac 1{V_r(x_1)}.
\end{align*}
Therefore,
$
|\Gamma_1(x)|\ls (\frac\rho{r})^\bz\frac 1{V_r(x_1)}.
$

Notice that, if $y\in\supp\varphi\subset B(x,(2A_0)^3\rho)$, then  $d(x',y)<9A_0^4\rho$, so, by an argument
similar to the estimation of $\Gamma_1(x)$, we also conclude that
$|\Gamma_1(x')|\ls (\frac\rho{r})^\bz\frac 1{V_r(x_1)}$.
Thus, we have
$$|\Gamma_1(x)-\Gamma_1(x')|\le |\Gamma_1(x)|+|\Gamma_1(x')|\ls \lf(\frac\rho{r}\r)^\bz\frac 1{V_r(x_1)}.$$

To estimate $|\Gamma_2(x)-\Gamma_2(x')|$, by \eqref{eq:Kcany} and the fact $\int_X K(x,y)\,d\mu(y)=\int_X K(x',y)\,d\mu(y)$, we write
\begin{align*}
\Gamma_2(x)-\Gamma_2(x')
&=\int_X [K(x,y)-K(x',y)][f(y)-f(x)]\zeta(y)\,d\mu(y)\\
&\quad+[f(x)-f(x')]\int_X K( x',y)\vz(y)\,d\mu(y)=:\Gamma_{2,1}+\Gamma_{2,2}.
\end{align*}
To estimate $\Gamma_{2,1}$, by  $0\le\zeta\le 1$ and the support condition of $\zeta$, we have
\begin{align*}
|\Gamma_{2,1}|&\le\int_{\gfz{d(x,y)\ge(2A_0)^2\rho}{d(x,y)\le (2A_0)^{-1}[r+d(x_1,x)]}}
|K(x,y)-K(x',y)||f(y)-f(x)|\,d\mu(y)\\
&\quad+\int_{\gfz{d(x,y)\ge(2A_0)^2\rho}{(x,y)>(2A_0)^{-1}[r+d(x_1,x)]}}
|K(x,y)-K(x',y)|[|f(y)|+|f(x)|]\,d\mu(y)=:\Gamma_{2,1,1}+\Gamma_{2,1,2}.
\end{align*}
By \eqref{eq:Ksize}, the regularity condition of $f$, $d(x,x_1)<r$ and Lemma \ref{lem-add}(iii), we have
\begin{align*}
\Gamma_{2,1,1}&\ls\frac 1{V_r(x_1)}
\int_{d(x,y)\ge (2A_0)^2\rho}\lf[\frac{d(x,x')}{d(x,y)}\r]^{s}\frac 1{V(x,y)}
\lf[\frac{d(x,y)}{r}\r]^\bz\,d\mu(y)\\
&\ls\lf(\frac\rho{r}\r)^\bz\frac 1{V_r(x_1)}
\int_{d(x,y)\ge (2A_0)^2\rho}\lf[\frac\rho{d(x,y)}\r]^{s-\bz}\frac 1{V(x,y)}
\,d\mu(y)\ls\lf(\frac\rho{r}\r)^\bz\frac 1{V_r(x_1)}.
\end{align*}
From \eqref{eq:Ksize}, the size condition of $f$, $d(x,x_1)<r$ and Lemma \ref{lem-add}(iii), we deduce that
\begin{align*}
\Gamma_{2,1,2}&\ls\frac 1{V_r(x_1)}
\int_{\gfz{d(x,y)\ge(2A_0)^2\rho }{ d(x,y)>(2A_0)^{-1}[r+d(x_1,x)]}}
\lf[\frac{d(x,x')}{d(x,y)}\r]^{s}\frac 1{V(x,y)}\,d\mu(y)\\
&\ls \lf(\frac\rho{r}\r)^\bz\frac 1{V_r(x_1)}
\int_{
d(x,y)\ge(2A_0)^2\rho}\lf[\frac{\rho}{d(x,y)}\r]^{s-\bz}\frac 1{V(x,y)}\,d\mu(y)\ls \lf(\frac\rho{r}\r)^\bz\frac 1{V_r(x_1)}.
\end{align*}
We therefore obtain $
|\Gamma_{2,1}|\ls(\frac\rho{r})^\bz\frac 1{V_r(x_1)}.$

Now we estimate $\Gamma_{2,2}$. By Lemma \ref{lem:3.24}, the regularity of $f$ and $d(x, x_1)<r$, we know that
$$
|\Gamma_{2,2}|\ls|f(x)-f(x')|\ls\lf[\frac{d(x,x')}{r+d(x_1,x)}\r]^\bz\frac 1{V_r(x_1)+V(x_1,x)}
\lf[\frac r{r+d(x_1,x)}\r]^\gz\sim \lf(\frac\rho{r}\r)^\bz\frac 1{V_r(x_1)}.
$$
Therefore, $|\Gamma_2(x)-\Gamma_2(x')|\ls(\frac\rho{r})^\bz\frac 1{V_r(x_1)}$. Altogether, we find that
\eqref{eq:Tfreg} holds true when $\rho<(3A_0)^{-4}[r+d(x_1,x)]$ and $d(x_1,x)< r$.

{\it Case 3) $\rho<(3A_0)^{-4}[r+d(x_1,x)]$ and $R=d(x_1,x)\ge r$.}  The proof for this case was essentially
given in \cite[(3.12)]{HLW16}. We include some details here for the completeness of this article.
Notice that, in this case, the right-hand side of \eqref{eq:Tfreg} is comparable with
$(\frac\rho R)^\bz\frac 1{V_R(x_1)}(\frac rR)^\gz$. Let $u_i$ and $f_i$ be
defined as in Section \ref{size}. Then
$$
|Tf(x)-Tf(x')|\le \sum_{i=1}^3 |Tf_i(x)-Tf_i(x')|.
$$

To estimate the term $|Tf_1(x)-Tf_1(x')|$,  we write
\begin{align*}
Tf_1(x)&=\int_X K(x,y)f_1(y)\,d\mu(y)\\
&=\int_X K(x,y)\wt u(y)\lf[f_1(y)-f_1(x)\r]\,d\mu(y)\\
&\quad+\lf[\int_X K(x,y)\wt v(y)f_1(y)\,d\mu(y)+f_1(x)\int_X K(x,y)\wt u(y)\,d\mu(y)\r]
=:\Gamma_3(x)+\Gamma_4(x),
\end{align*}
where $\wt u$ is a function satisfying $\chi_{B(x,2A_0\rho)}\le\wt u\le\chi_{B(x,4A_0^2\rho)}$ and
$\|\wt{u}\|_{\dot{C}^\eta(X)}\ls\rho^{-\eta}$, and $\wt v:=1-\wt u$.
By $\supp u\subset B(x,4A_0^2\rho)$, \eqref{eq:Ksize},  \eqref{eq:1} and Lemma \ref{lem-add}(iii), we conclude that
\begin{align*}
|\Gamma_3(x)-\Gamma_3(x')|
&\le\lf|\Gamma_3(x)\r|+\lf|\Gamma_3(x')\r|\\
&\le \int_{d(x,y)<4A_0^2\rho}|K(x,y)|\lf|f_1(y)-f(x)\r|\,d\mu(y)\\
&\quad+\int_{d(x,y)<4A_0^2\rho}|K(x',y)|\lf|f_1(y)-f(x')\r|\,d\mu(y)\\
&\ls\int_{d(x,y)<4A_0^2\rho}\frac 1{V(x,y)}
\lf[\frac{d(x,y)}R\r]^\bz\frac 1{V_R(x_1)}\lf(\frac rR\r)^\gz\,d\mu(y)\\
&\quad+\int_{d(x',y)<5A_0^3\rho}\frac 1{V(x',y)}\lf[\frac{d(x',y)}R\r]^\bz
\frac 1{V_R(x_1)}\lf(\frac rR\r)^\gz\,d\mu(y)
\ls\lf(\frac\rho R\r)^\bz
\frac 1{V_R(x_1)}\lf(\frac rR\r)^\gz.
\end{align*}
To deal with the term $\Gamma_4$, using the assumption that $\int_X K(x,y)\,d\mu(y)
=\int_X K(x',y)\,d\mu(y)$, we write
\begin{align*}
\Gamma_4(x)-\Gamma_4(x')
&=\int_X [K(x,y)-K(x',y)]\wt v(y)\lf[f_1(y)-f_1(x)\r]\,d\mu(y)\\
&\quad+\lf[f_1(x)-f_1(x')\r]\int_X K(x',y)\wt u(y)\,d\mu(y)=:\Gamma_{4,1}+\Gamma_{4,2}.
\end{align*}
From $\supp\wt v\subset B(x, 2A_0\rho)^\complement$, \eqref{eq:Kreg}, \eqref{eq:2}, $\bz<s$ and Lemma
\ref{lem-add}(iii), we deduce that
\begin{align*}
|\Gamma_{4,1}|&\le \int_{d(x,y)\ge 2A_0\rho} |K(x,y)-K(x',y)|\lf|f_1(y)-f_1(x)\r|\,d\mu(y)\\
&\ls\int_{d(x,y)\ge 2A_0\rho}\lf[\frac\rho{d(x,y)}\r]^{s}\frac 1{V(x,y)}
\lf[\frac{d(x,y)}R\r]^\bz\frac 1{V_R(x_1)}\lf(\frac rR\r)^\gz\,d\mu(y)\\
&\sim\lf(\frac\rho R\r)^\bz\frac 1{V_R(x_1)}\lf(\frac rR\r)^\gz
\int_{d(x,y)\ge 2A_0\rho}\frac 1{V(x,y)}\lf[\frac\rho{d(x,y)}\r]^{s-\bz}\,d\mu(y)
\ls\lf(\frac\rho R\r)^\bz\frac 1{V_R(x_1)}\lf(\frac rR\r)^\gz.
\end{align*}
By \eqref{eq:2} and Lemma \ref{lem:3.24}, we know that
$$
|\Gamma_{4,2}|\ls\lf|f_1(x)-f_1(x')\r|\ls\lf(\frac\rho R\r)^\bz\frac 1{V_R(x_1)}\lf(\frac rR\r)^\gz,
$$
so does $|\Gamma_4(x)-\Gamma_4(x')|$. This shows that $|Tf_1(x)-Tf_1(x')|$ satisfies the desired estimate.

Now we consider the term $|Tf_2(x)-Tf_2(x')|$. Indeed, we have
\begin{align}\label{eq-reg1}
\lf|Tf_2(x)-Tf_2(x')\r|&=\lf|\int_X[K(x,y)-K(x',y)]f_2(y)\,d\mu(y)\r|\\
&\le\int_X |K(x,y)-K(x',y)-K(x,x_1)-K(x',x_1)|\lf|f_2(y)\r|\,d\mu(y)\noz\\
&\quad+|K(x,x_1)-K(x',x_1)|\lf|\int_X f_2(y)\,d\mu(y)\r|=:\Gamma_5+\Gamma_6.\noz
\end{align}
By $\supp f_2\subset B(x_1, (2A_0)^{-2}R)$, the fact that $\rho=d(x, x')\le (2A_0)^{-2}R$, \eqref{eq:Kdreg},
 the size condition of $f$,  $\gz<s$ and Lemma \ref{lem-add}(iii), we conclude that
\begin{align*}
\Gamma_5&=\int_{d(x_1,y)\le(2A_0)^{-2}R}|K(x,x_1)-K(x',x_1)-K(x,y)+K(x',y)||f(y)|\,d\mu(y)\\
&\ls\int_{d(x_1,y)\le(2A_0)^{-2}R}\frac 1{V(x, x_1)}\lf[\frac{d(x,x')} {d(x, x_1)}\r]^{s}\lf[\frac{d(x_1,y)}{d(x, x_1)}\r]^{s}
\frac1{V_r(x_1)+V(x_1,y)}\lf[\frac r{r+d(x_1,y)}\r]^\gz\,d\mu(y)\\
&\ls\lf(\frac\rho R\r)^\bz\frac 1{V_R(x)}\lf(\frac rR\r)^\gz
\int_{d(x_1,y)\le(2A_0)^{-2}R}\frac 1{V(x_1,y)}\lf[\frac{d(x_1,y)}R\r]^{s-\gz}\,d\mu(y)
\ls\lf(\frac\rho R\r)^\bz\frac 1{V_R(x_1)}\lf(\frac rR\r)^\gz.
\end{align*}
From the fact that $\rho=d(x, x')\le (2A_0)^{-1}R$, \eqref{eq:Kreg} and \eqref{eq:5}, we deduce that
$$
\Gamma_6\ls \frac 1{V(x, x_1)}\lf[\frac{d(x,x')} {d(x, x_1)}\r]^{s}\lf(\frac rR\r)^\gz
\ls\lf(\frac\rho R\r)^\bz\frac 1{V_R(x_1)}\lf(\frac rR\r)^\gz.
$$
Combining the estimates of $\Gamma_5$ and $\Gamma_6$, we know that
$$
\lf|Tf_2(x)-Tf_2(x')\r|\ls\lf(\frac\rho R\r)^\bz\frac 1{V_R(x_1)}\lf(\frac rR\r)^\gz.
$$

Finally, by \eqref{eq:Kreg}, \eqref{eq:3} and \eqref{eq:4}, we conclude that
\begin{align*}
\lf|Tf_3(x)-Tf_3(x')\r|&\le\int_{d(x,y)\ge (2A_0)^{-3}R\ge 2A_0\rho}
|K(x,y)-K(x',y)|\lf|f_3(y)\r|\,d\mu(y)\\
&\ls\int_{d(x,y)\ge (2A_0)^{-3}R\ge 2A_0\rho}\frac 1{V(x,y)}\lf[\frac{d(x,x')}{d(x,y)}\r]^s\lf|f_3(y)\r|
\,d\mu(y)\\
&\ls\lf(\frac\rho R\r)^\bz\frac 1{V_R(x)}\int_X \lf|f_3(y)\r|\,d\mu(y)
\ls\lf(\frac\rho R\r)^\bz\frac 1{V_R(x_1)}\lf(\frac rR\r)^\gz.
\end{align*}
Thus,  when $d(x,x')\le (3A_0)^{-4}[1+d(x_0,x)]$ and $R\ge\rho$, we also have
$$
|Tf(x)-Tf(x')|\ls\lf(\frac\rho R\r)^\bz\frac 1{V_R(x_1)}\lf(\frac rR\r)^\gz,
$$
as desired.

Summarizing all three cases, we complete  the proof of \eqref{eq:Tfreg} and hence of Theorem \ref{thm:Kbdd}.
\end{proof}

\subsection{Boundedness of Calder\'{o}n-Zygmund-type operators on $\CG(\bz,\gz)$}\label{s3.4}

The aim of this section is to show the following boundedness of Calder\'{o}n-Zygmund-type operators on $\CG(\bz,\gz)$, whose proof is
similar to that of Theorem \ref{thm:Kbdd}.

\begin{theorem}\label{thm:Kibdd}
Let all the notation be
as in Theorem \ref{thm:Kbdd}. Assume that there exist positive constants $C_T$,
$r_0$ and $\sigma$ such that
\begin{enumerate}
\item[{\rm (d)}] if $d(x,y)\ge r_0$, then
\begin{equation}\label{eq:Ksizeb}
|K(x,y)|\le C_T\frac 1{V(x,y)}\lf[\frac{r_0}{d(x,y)}\r]^\sigma;
\end{equation}
\item[{\rm (e)}] if $d(x,y)\ge r_0$ and $d(x,x')\le(2A_0)^{-1}d(x,y)$, then
\begin{equation}\label{eq:Kregb}
|K(x,y)-K(x',y)|\le C_T\lf[\frac{d(x,x')}{d(x,y)}\r]^{s}\frac 1{V(x,y)}\lf[\frac{r_0}{d(x,y)}\r]^\sigma.
\end{equation}
\end{enumerate}
Let $\bz\in(0,s)$ and $\gz\in(0,s)\cap(0,\sigma]$. Then, for any
 $f\in\CG(x_1,r_0,\bz,\gz)$ with $x_1\in X$,
$$
\|Tf\|_{\CG(x_1,r_0,\bz,\gz)}\le C\lf[C_T+\|T\|_{L^2(X)\to L^2(X)}+|c_0|\r]\|f\|_{\CG(x_1,r_0,\bz,\gz)},
$$
where $C$ is a positive constant independent of $x_1, r_0$, $c_0$ and $f$.
Moreover, if $T$ further  satisfies that $
\int_X K(x,y)\,d\mu(x)=0$ for any $y\in X$,
then $Tf\in\mathring{\CG}(x_1,r_0,\bz,\gz)$ for any $f\in\CG(x_1,r_0,\bz,\gz)$.
\end{theorem}

\begin{proof}
Similarly to the proof of Theorem \ref{thm:Kbdd}, we only need  to prove \eqref{eq:Tfsize} and \eqref{eq:Tfreg}, but now $f\in\CG(x_1,r_0,\bz,\gz)$
has no cancellation.

To show \eqref{eq:Tfsize}, following the arguments used in Section \ref{size}, we consider the cases
$R:=d(x_1,x)<2A_0r_0$ and $R:=d(x_1,x)\ge 2A_0r_0$, respectively. Observe that the cancellation condition of
$f$ is only used in the proof of \eqref{eq:5} and hence in the estimation of $Tf_2(x)$ (this is for the case
$R\ge 2A_0r_0$).
So we need to re-estimate  $Tf_2(x)$ without using  \eqref{eq:5}.   Indeed, if
$f_2(y)\neq 0$, then $d(x_1,y)< (2A_0)^{-2}R$ and hence $d(y,x)\ge(2A_0)^{-1}R\ge r_0$. From this, assumption
(b) of Theorem \ref{thm:Kbdd}, \eqref{eq:Ksizeb}, the size condition of $f$ and
Lemma \ref{lem-add}(ii), we deduce that
\begin{align*}
|Tf_2(x)|&\le \int_{d(x_1,y)\le (2A_0)^{-2}R} |K(x,y)f_2(y)|\,d\mu(y)
\ls\int_{d(x_1,y)\le (2A_0)^{-2}R} \frac 1{V(x,y)}\lf[\frac{r_0}{d(x,y)}\r]^\sigma|f(y)|\,d\mu(y)\\
&\ls\frac 1{V_{R}(x)}\lf(\frac{r_0}{R}\r)^\sigma\ls\frac 1{V_{r_0}(x_1)+V(x_1,x)}
\lf[\frac{r_0}{r_0+d(x_1,x)}\r]^\gz.
\end{align*}
This finishes the proof of the size condition \eqref{eq:Tfsize}.

Next we assume that  $\rho:=d(x,x')\le(2A_0)^{-1}[r_0+d(x_1,x)]$ and show \eqref{eq:Tfreg} by following the
arguments used in Section \ref{reg}.
Again, the  cancellation condition of $f$ was only used in dealing with the regularity of $Tf_2$ [see
\eqref{eq-reg1} in Case 3) therein].
So we re-estimate $|Tf_2(x)-Tf_2(x')|$ as follows.
Indeed, with the assumption of Case) 3, if $y\in X$ satisfies $f_2(y)\neq 0$, then $d(x,y)\ge(2A_0)^{-1}R>r_0$
and hence $d(x,x')=\rho\le (3A_0)^{-4}[r_0+R]\le (2A_0)^{-1}d(x,y)$. By this,
\eqref{eq:Kregb} and Lemma \ref{lem-add}(ii), we conclude that
\begin{align*}
|Tf_2(x)-Tf_2(x')|&\le\int_{d(x_1,y)\le (2A_0)^{-2}R} |K(x,y)-K(x',y)||f_2(y)|\,d\mu(y)\\
&\ls\int_{d(x_1,y)\le (2A_0)^{-2}R}\lf[\frac{d(x,x')}{d(x,y)}\r]^s
\frac 1{V(x,y)}\lf[\frac{r_0}{d(x,y)}\r]^\sigma|f(y)|\,d\mu(y)\\
&\ls\lf[\frac{d(x,x')}{R}\r]^s\frac 1{V_R(x)}\lf(\frac{r_0}{R}\r)^\sigma\\
&\ls\lf[\frac{d(x,x')}{r_0+d(x_1,x)}\r]^s\frac 1{V_{r_0}(x_1)+V(x_1,x)}
\lf[\frac{r_0}{r_0+d(x_1,x)}\r]^\gz.
\end{align*}
This proves \eqref{eq:Tfreg} and hence finishes the proof of Theorem \ref{thm:Kibdd}.
\end{proof}


\section{Homogeneous continuous Calder\'{o}n reproducing formulae}\label{hcrf}

Suppose that $\{Q_k\}_{k=-\fz}^\fz$ is an $\exp$-ATI. For any fixed $N\in\nn$, we have
\begin{equation}\label{eq:defRT}
I=\sum_{k=-\fz}^\fz Q_k=\sum_{k=-\fz}^\fz\sum_{l=-\fz}^\fz Q_{k+l}Q_k=
\sum_{k=-\fz}^\fz Q_k^NQ_k+\sum_{k=-\fz}^\fz\sum_{|l|>N} Q_{k+l}Q_k=:T_N+R_N
\quad\text{in}\; L^2(X),
\end{equation}
where $Q_k^N:=\sum_{|l|\le N}Q_{k+l}$ for any $k\in\zz$. To deal with the remainder $R_N$
(see Section \ref{RN}), we need to estimate compositions of $\exp$-ATIs first (see Section \ref{com}).
Then the homogeneous continuous reproducing formula is established in  Section \ref{pr}.

\subsection{Compositions of two exp-ATIs}\label{com}

In this subsection, we consider the compositions of two $\exp$-ATIs.
\begin{lemma}\label{lem:ccrf1}
Let $\{\wz Q_j\}_{j\in\zz}$ and $\{{Q}_k\}_{k\in\zz}$ be two $\exp$-{\rm ATI}s.
Fix $\eta'\in(0,\eta)$. Then, for any $j,\ k\in\zz$, the kernel of
$\wz{Q}_jQ_k$, still denoted by $\wz{Q}_jQ_k$, has the following properties:
\begin{enumerate}
\item for any $x,\ y\in X$,
\begin{align}\label{eq:mixsize}
\lf|\wz{Q}_jQ_k(x,y)\r|&\le C\dz^{|j-k|\eta}\frac 1{V_{\dz^{j\wedge k}}(x)}
\exp\lf\{-c\lf[\frac{d(x,y)}{\dz^{j\wedge k}}\r]^a\r\}\exp\lf\{-c\lf[\frac{d(x,\CY^{j\wedge k})}{\dz^{j\wedge k}}\r]^a\r\};
\end{align}
\item if $d(x,x')\le(2A_0)^{-2}d(x,y)$ with $x\neq y$, then
\begin{align}\label{eq:mixregx}
&\lf|\wz{Q}_jQ_k(x,y)-\wz{Q}_jQ_k(x',y)\r|+\lf|\wz{Q}_jQ_k(y,x)-\wz{Q}_jQ_k(y,x')\r|\\
&\quad\le C\dz^{|j-k|(\eta-\eta')}\lf[\frac{d(x,x')}{\dz^{j\wedge k}}\r]^{\eta'}
\frac 1{V_{\dz^{j\wedge k}}(x)}
\exp\lf\{-c\lf[\frac{d(x,y)}{\dz^{j\wedge k}}\r]^a\r\}\exp\lf\{-c\lf[\frac{d(x,\CY^{j\wedge k})}{\dz^{j\wedge k}}\r]^a\r\};\noz
\end{align}
\item if $d(x,x')\le(2A_0)^{-3}d(x,y)$ and $d(y,y')\le(2A_0)^{-3}d(x,y)$ with $x\neq y$, then
\begin{align}\label{eq:mixdreg}
&\lf|\lf[\wz{Q}_jQ_k(x,y)-\wz{Q}_jQ_k(x',y)\r]-\lf[\wz{Q}_jQ_k(x,y')-\wz{Q}_jQ_k(x',y')\r]\r|\\
&\noz\quad\le C\dz^{|j-k|(\eta-\eta')}\lf[\frac{d(x,x')}{\dz^{j\wedge k}}\r]^{\eta'}
\lf[\frac{d(y,y')}{\dz^{j\wedge k}}\r]^{\eta'}
\frac 1{V_{\dz^{j\wedge k}}(x)}
\exp\lf\{-c\lf[\frac{d(x,y)}{\dz^{j\wedge k}}\r]^a\r\}\\
&\qquad\times\exp\lf\{-c\lf[\frac{d(x,\CY^{j\wedge k})}{\dz^{j\wedge k}}\r]^a\r\};
\noz
\end{align}
\item for any $x,\ y\in X$,
\begin{equation}\label{eq:mixcan}
\int_X \wz{Q}_jQ_k(x,y')\,d\mu(y')=0=\int_X \wz{Q}_jQ_k(x',y)\,d\mu(x'),
\end{equation}
\end{enumerate}
where, in (i), (ii) and (iii), $C$ and $c$ are positive constants independent of $j$, $k$, $x$, $y$, $x'$ and
$y'$.
Moreover, according to Remark \ref{rem:andef}, in (i), (ii) and (iii), the factors ${V_{\dz^{j\wedge k}}(x)}$
and $d(x,\CY^{j\wedge k})$ can be replaced, respectively, by ${V_{\dz^{j\wedge k}}(y)}$
and $d(y,\CY^{j\wedge k})$, but with the factor $\exp\{-c[\frac{d(x,y)}{\dz^{j\wedge k}}]^a\}$
replaced by $\exp\{-c'[\frac{d(x,y)}{\dz^{j\wedge k}}]^a\}$ for some $c'\in(0,c)$ independent of $j$, $k$, $x$,
$y$, $x'$ and $y'$.
\end{lemma}
\begin{proof}
For any $j,\ k\in\zz$ and $x,\ y\in X$, we have $\wz{Q}_jQ_k(x,y)=\int_X\wz{Q}_j(x,z)Q_k(z,y)\,d\mu(z)$.
By \eqref{eq:etisize}, the dominated convergence theorem and the cancellation property of
$\wz{Q}_j$ and $Q_k$ (see Definition \ref{def:eti}), we then have \eqref{eq:mixcan}.
So it remains to prove \eqref{eq:mixsize}, \eqref{eq:mixregx} and \eqref{eq:mixdreg}.
By symmetry, we may as well assume that $j\ge k$.

First we show  \eqref{eq:mixsize}. By the cancellation of $\wz{Q}_j$, we have
\begin{align*}
\lf|\wz{Q}_jQ_k(x,y)\r|& 
=\lf|\int_X\wz{Q}_j(x,z)[Q_k(z,y)-Q_k(x,y)]\,d\mu(z)\r|\\
&\le\int_{d(x,z)\le\dz^k}\lf|\wz{Q}_j(x,z)\r||Q_k(z,y)-Q_k(x,y)|\,d\mu(z)+
\int_{d(x,z)>\dz^k}\lf|\wz{Q}_j(x,z)\r||Q_k(z,y)|\,d\mu(z)\\
&\quad+|Q_k(x,y)|\int_{d(x,z)>\dz^k}\lf|\wz{Q}_j(x,z)\r|\,d\mu(z)=:\RY_1+\RY_2+\RY_3.
\end{align*}
From the size condition of $\wz{Q}_j$, the smooth condition of $Q_k$ and Lemma \ref{lem-add}(ii), we deduce
that
\begin{align*}
\RY_1&\ls\int_{d(x,z)\le\dz^k}\frac 1{\sqrt{V_{\dz^j}(x)V_{\dz^j}(z)}}
\exp\lf\{-\nu\lf[\frac{d(x,z)}{\dz^j}\r]^a\r\}
\lf[\frac{d(x,z)}{\dz^k}\r]^\eta\frac 1{\sqrt{V_{\dz^k}(x)V_{\dz^k}(y)}}\\
&\quad\times\exp\lf\{-\nu\lf[\frac{d(x,y)}{\dz^k}\r]^a\r\}\exp\lf\{-\nu\lf[\frac{d(x,\CY^k)}{\dz^k}\r]^a\r\}
\,d\mu(z)\\
&\ls\dz^{(j-k)\eta}\frac 1{\sqrt{V_{\dz^k}(x)V_{\dz^k}(y)}}\exp\lf\{-\nu\lf[\frac{d(x,y)}{\dz^k}\r]^a\r\}
\exp\lf\{-\nu\lf[\frac{d(x,\CY^k)}{\dz^k}\r]^a\r\}\\
&\quad\times\int_{d(x,z)\le\dz^k}\frac 1{\sqrt{V_{\dz^j}(x)V_{\dz^j}(z)}}\lf[\frac{d(x,z)}{\dz^j}\r]^\eta
\exp\lf\{-\nu\lf[\frac{d(x,z)}{\dz^j}\r]^a\r\}\,d\mu(z)\\
&\ls\dz^{(j-k)\eta}\frac 1{\sqrt{V_{\dz^k}(x)V_{\dz^k}(y)}}\exp\lf\{-\nu\lf[\frac{d(x,y)}{\dz^k}\r]^a\r\}
\exp\lf\{-\nu\lf[\frac{d(x,\CY^k)}{\dz^k}\r]^a\r\}.
\end{align*}
By the size conditions of $\wz{Q}_j$, $Q_k$ and Lemma \ref{lem-add}(ii), we conclude that
\begin{align*}
\RY_2&\ls\int_{d(x,z)>\dz^k} \frac 1{\sqrt{V_{\dz^j}(x)V_{\dz^j}(z)}}\exp\lf\{-\nu\lf[\frac{d(x,z)}{\dz^j}\r]^a\r\}\\
&\quad\times
\frac 1{\sqrt{V_{\dz^k}(x)V_{\dz^k}(y)}}\exp\lf\{-\nu\lf[\frac{d(y,z)}{\dz^k}\r]^a\r\}
\exp\lf\{-\nu\lf[\frac{d(y,\CY^k)}{\dz^k}\r]^a\r\}\,d\mu(z).
\end{align*}
Notice that $a\in(0,1]$. From the inequality $[d(x,y)]^a\le A_0^a([d(x,z)]^a+[d(y,z)]^a)$, it follows that
\begin{align}\label{eq-add7}
\exp\lf\{-\frac\nu2\lf[\frac{d(x,z)}{\dz^j}\r]^a\r\}\exp\lf\{-\frac\nu2\lf[\frac{d(y,z)}{\dz^k}\r]^a\r\}\le \exp\lf\{-\frac\nu2\lf[\frac{d(x,y)}{A_0\dz^k}\r]^a\r\}.
\end{align}
Similarly, from  the fact $[d(x,\CY^k)]^a\le A_0^a([d(x,y)]^a+[d(y,\CY^k)]^a)$, it also follows that
\begin{align}\label{eq-add8}
\exp\lf\{-\frac\nu4\lf[\frac{d(x,y)}{A_0\dz^k}\r]^a\r\}
\exp\lf\{-\nu\lf[\frac{d(y,\CY^k)}{\dz^k}\r]^a\r\}
\le \exp\lf\{-\frac\nu4\lf[\frac{d(x,\CY^k)}{A_0\dz^k}\r]^a\r\}.
\end{align}
Combining the above two formulae and Lemma \ref{lem-add}(ii), together with the fact
$\exp\{-\frac\nu4[\frac{d(x,z)}{\dz^j}]^a\}\le  \exp\{-\frac\nu4\dz^{(k-j)a}\}\ls \dz^{(j-k)\eta}$, we find that
\begin{align*}
\RY_2&\ls\dz^{(j-k)\eta}\frac 1{\sqrt{V_{\dz^k}(x)V_{\dz^k}(y)}} \exp\lf\{-\frac\nu4\lf[\frac{d(x,y)}{A_0\dz^k}\r]^a\r\}\exp\lf\{-\frac\nu4\lf[\frac{d(x,\CY^k)}{A_0\dz^k}\r]^a\r\}.
\end{align*}
Again, by the size conditions of $\wz{Q}_j$ and $Q_k$, and Lemma \ref{lem-add}(ii), we obtain
\begin{align*}
\RY_3&\ls\frac 1{\sqrt{V_{\dz^k}(x)V_{\dz^k}(y)}}\exp\lf\{-\nu\lf[\frac{d(x,\CY^k)}{\dz^k}\r]^a\r\}
\exp\lf\{-\nu\lf[\frac{d(x,y)}{\dz^k}\r]^a\r\}\\
&\quad\times\int_{d(x,z)>\dz^k}\frac 1{\sqrt{V_{\dz^j}(x)V_{\dz^j}(z)}}\exp\lf\{-\nu\lf[\frac{d(x,z)}{\dz^j}\r]^a\r\}\,d\mu(z)\\
&\ls\exp\lf\{-\frac\nu 2 \dz^{(k-j)a}\r\}\frac 1{\sqrt{V_{\dz^k}(x)V_{\dz^k}(y)}}\exp\lf\{-\nu\lf[\frac{d(x,\CY^k)}{\dz^k}\r]^a\r\}
\exp\lf\{-\nu\lf[\frac{d(x,y)}{\dz^k}\r]^a\r\},
\end{align*}
which is desired, because $\exp\{-\frac\nu 2 \dz^{(k-j)a}\}\ls \dz^{(j-k)\eta}$.
Therefore, we obtain the size condition \eqref{eq:mixsize}.

Next we consider $|\wz{Q}_jQ_k(x,y)-\wz{Q}_jQ_k(x',y)|$ in \eqref{eq:mixregx}, where
$d(x,x')\le(2A_0)^{-2}d(x,y)$ with $x\neq y$. 
Notice that $d(x,x')\le(2A_0)^{-2}d(x,y)$ implies that  $(\frac 43A_0)^{-1}d(x',y)\le d(x,y)\le \frac 43A_0d(x',y)$.
On one hand, by this and \eqref{eq:mixsize}, we find that
\begin{align}\label{eq-add6}
&\lf|\wz{Q}_jQ_k(x,y)-\wz{Q}_jQ_k(x',y)\r|\\
&\quad\le\lf|\wz{Q}_jQ_k(x,y)\r|+\lf|\wz{Q}_jQ_k(x',y)\r|\notag\\
&\quad\ls\dz^{(j-k)\eta}\frac 1{\sqrt{V_{\dz^k}(x)V_{\dz^k}(y)}} \exp\lf\{-c\lf[\frac{d(x,y)}{\dz^k}\r]^a\r\}\exp\lf\{-c\lf[\frac{d(x,\CY^k)}{\dz^k}\r]^a\r\}.\notag
\end{align}
On the other hand, we write
\begin{align*}
&\lf|\wz{Q}_jQ_k(x,y)-\wz{Q}_jQ_k(x',y)\r|\\
&\quad=\lf|\int_X\lf[\wz{Q}_j(x,z)-\wz{Q}_j(x',z)\r][Q_k(z,y)-Q_k(x,y)]\,d\mu(z)\r|\\
&\quad\le\int_{d(x,z)<2A_0d(x,x')}\lf|\wz{Q}_j(x,z)-\wz{Q}_j(x',z)\r||Q_k(z,y)-Q_k(x,y)|\,d\mu(z)\\
&\qquad+\int_{2A_0d(x,x')\le d(x,z)\le(2A_0)^{-1}d(x,y)}\cdots
+\int_{d(x,z)>(2A_0)^{-1}d(x,y)}\cdots=:\RY_4+\RY_5+\RY_6.
\end{align*}
For the term $\RY_4$, noticing that $d(x,z)<2A_0d(x,x')\le (2A_0)^{-1}d(x,y)$, we apply the
regularity condition to $|Q_k(z,y)-Q_k(x,y)|$ [see Remark \ref{rem:andef}(iii)], and then use Lemma
\ref{lem-add}(ii) to deduce that
\begin{align*}
\RY_4&\ls\frac 1{\sqrt{V_{\dz^k}(x)V_{\dz^k}(y)}}\exp\lf\{-\nu'\lf[\frac{d(x,y)}{\dz^k}\r]^a\r\}
\exp\lf\{-\nu'\lf[\frac{d(x,\CY^k)}{\dz^k}\r]^a\r\}\\
&\quad\times\int_{d(x,z)<2A_0d(x,x')}
\lf[\lf|\wz{Q}_j(x,z)\r|+\lf|\wz{Q}_j(x',z)\r|\r]\lf[\frac{d(x,z)}{\dz^k}\r]^\eta\,d\mu(z)
\\
&\ls\lf[\frac{d(x,x')}{\dz^k}\r]^\eta\frac 1{\sqrt{V_{\dz^k}(x)V_{\dz^k}(y)}}\exp\lf\{-\nu'\lf[\frac{d(x,y)}{\dz^k}\r]^a\r\}
\exp\lf\{-\nu'\lf[\frac{d(x,\CY^k)}{\dz^k}\r]^a\r\}.
\end{align*}
To deal with $\RY_5$, we first use the regularity conditions to estimate both $|\wz{Q}_j(x,z)-\wz{Q}_j(x',z)|$
and $|Q_k(z,y)-Q_k(x,y)|$ [see Remark \ref{rem:andef}(iii)], and then  use Lemma \ref{lem-add}(ii) to conclude
that
\begin{align*}
\RY_5&\ls\frac 1{\sqrt{V_{\dz^k}(x)V_{\dz^k}(y)}}\exp\lf\{-\nu'\lf[\frac{d(x,y)}{\dz^k}\r]^a\r\}
\exp\lf\{-\nu'\lf[\frac{d(x,\CY^k)}{\dz^k}\r]^a\r\} \\
&\quad\times\int_X \lf[\frac{d(x,x')}{\dz^j}\r]^\eta
\frac 1{\sqrt{V_{\dz^j}(x)V_{\dz^j}(z)}}\exp\lf\{-\nu'\lf[\frac{d(x,z)}{\dz^j}\r]^a\r\}
\lf[\frac{d(x,z)}{\dz^k}\r]^\eta\,d\mu(z)\\
&\ls\lf[\frac{d(x,x')}{\dz^k}\r]^\eta\frac 1{\sqrt{V_{\dz^k}(x)V_{\dz^k}(y)}}\exp\lf\{-\nu'\lf[\frac{d(x,y)}{\dz^k}\r]^a\r\}
\exp\lf\{-\nu'\lf[\frac{d(x,\CY^k)}{\dz^k}\r]^a\r\}.
\end{align*}
For the term $\RY_6$, noticing that $d(x,x')\le (2A_0)^{-2}d(x,y)<(2A_0)^{-1}d(x,z)$, we  apply the regularity of $|\wz{Q}_j(x,z)-\wz{Q}_j(x',z)|$ [see Remark \ref{rem:andef}(iii)],
together with the size conditions of $Q_k$ when $d(x,z)> \dz^k$ and the regularity of $Q_k$ when $d(x,z)\le \dz^k$, 
to obtain
\begin{align*}
\RY_6&\ls\int_{d(x,z)>(2A_0)^{-1}d(x,y)}
\lf[\frac{d(x,x')}{\dz^j}\r]^\eta\frac 1{\sqrt{V_{\dz^j}(x)V_{\dz^j}(z)}}\exp\lf\{-\nu'\lf[\frac{d(x,z)}{\dz^j}\r]^a\r\}
\min\lf\{1,\;\lf[\frac{d(x,z)}{\dz^k}\r]^\eta\r\}\\
&\quad\times\frac 1{V_{\dz^k}(y)}\lf(\exp\lf\{-\nu'\lf[\frac{d(y,z)}{\dz^k}\r]^a\r\}
+\exp\lf\{-\nu'\lf[\frac{d(x,y)}{\dz^k}\r]^a\r\}\r)\exp\lf\{-\nu'\lf[\frac{d(y,\CY^k)}{\dz^k}\r]^a\r\}\,d\mu(z).
\end{align*}
As in the estimation of \eqref{eq-add7}, we always have
$$
\exp\lf\{-\frac{\nu'} 2\lf[\frac{d(x,z)}{\dz^j}\r]^a\r\}
\lf(\exp\lf\{-\nu'\lf[\frac{d(y,z)}{\dz^k}\r]^a\r\}
+\exp\lf\{-\nu'\lf[\frac{d(x,y)}{\dz^k}\r]^a\r\}\r)
\ls \exp\lf\{-\frac{\nu'}2\lf[\frac{d(x,y)}{A_0\dz^k}\r]^a\r\}
$$
for some positive constant $c$ independent of $x,\ y,\ z$ and $k,\ j$. Consequently,
\begin{align*}
\RY_6&\ls
\lf[\frac{d(x,x')}{\dz^k}\r]^\eta\frac 1{V_{\dz^k}(y)}\exp\lf\{-\frac{\nu'}2\lf[\frac{d(x,y)}{A_0\dz^k}\r]^a\r\} \exp\lf\{-\nu'\lf[\frac{d(y,\CY^k)}{\dz^k}\r]^a\r\}\\
&\quad\times
\int_{d(x,z)>(2A_0)^{-1}d(x,y)}
\lf[\frac{d(x,z)}{\dz^j}\r]^\eta\frac 1{\sqrt{V_{\dz^j}(x)V_{\dz^j}(z)}}\exp\lf\{-\frac{\nu'} 2\lf[\frac{d(x,z)}{\dz^j}\r]^a\r\}\,d\mu(z)\\
&\ls\lf[\frac{d(x,x')}{\dz^k}\r]^\eta\frac 1{\sqrt{V_{\dz^k}(x)V_{\dz^k}(y)}}\exp\lf\{-\frac{\nu'}4\lf[\frac{d(x,y)}{A_0^2\dz^k}\r]^a\r\} \exp\lf\{-\frac{\nu'}4\lf[\frac{d(x,\CY^k)}{A_0^2\dz^k}\r]^a\r\},
\end{align*}
where in the last step we also used \eqref{eq-add8}.

Combining the estimates of $\RY_4$ through $\RY_6$, we conclude that
\begin{align}\label{eq-add9}
&\lf|\wz{Q}_jQ_k(x,y)-\wz{Q}_jQ_k(x',y)\r|\\
&\quad\ls \lf[\frac{d(x,x')}{\dz^k}\r]^\eta\frac 1{\sqrt{V_{\dz^k}(x)V_{\dz^k}(y)}}\exp\lf\{-\frac{\nu'}4\lf[\frac{d(x,y)}{A_0\dz^k}\r]^a\r\} \exp\lf\{-\nu'\lf[\frac{d(x,\CY^k)}{A_0\dz^k}\r]^a\r\}.\notag
\end{align}
Taking the geometric means between \eqref{eq-add6} and \eqref{eq-add9}, we then obtain the desired estimate of
$|\wz{Q}_jQ_k(x,y)-\wz{Q}_jQ_k(x',y)|$ in \eqref{eq:mixregx}.

Now we estimate $|\wz{Q}_jQ_k(y,x)-\wz{Q}_jQ_k(y,x')|$ in  \eqref{eq:mixregx}.
Again $d(x,x')\le(2A_0)^{-2}d(x,y)$ implies that  $(\frac 43A_0)^{-1}d(x',y)\le d(x,y)\le \frac 43A_0d(x',y)$.
On one hand, by  \eqref{eq:mixsize}, we have
\begin{align}\label{eq-add66}
\lf|\wz{Q}_jQ_k(y,x)-\wz{Q}_jQ_k(y,x')\r|\ls\frac {\dz^{(j-k)\eta}}{\sqrt{V_{\dz^k}(x)V_{\dz^k}(y)}} \exp\lf\{-c\lf[\frac{d(x,y)}{\dz^k}\r]^a\r\}\exp\lf\{-c\lf[\frac{d(x,\CY^k)}{\dz^k}\r]^a\r\}.
\end{align}
On the other hand, using the size condition of $\wz Q_j$, as well the size and the regularity conditions of
$Q_k$, and invoking (i) and (ii) of Remark \ref{rem:andef}, we find that
\begin{align*}
&\lf|\wz{Q}_jQ_k(y,x)-\wz{Q}_jQ_k(y,x')\r|\\
&\quad=\lf|\int_X\wz{Q}_j(y,z)[Q_k(z,x)-Q_k(z,x')]\,d\mu(z)\r|\\
&\quad\ls \int_X \frac1{V_{\dz^j}(y)V_{\dz^k}(z)}
\min\lf\{1,\lf[\frac{d(x,x')}{\dz^k}\r]^\eta\r\}\exp\lf\{-\nu'\lf[\frac{d(z,\CY^k)}{\dz^k}\r]^a\r\}\\
&\qquad\times
\exp\lf\{-\nu'\lf[\frac{d(y,z)}{\dz^j}\r]^a\r\}
\lf(\exp\lf\{-\nu'\lf[\frac{d(z,x)}{\dz^k}\r]^a\r\}
+\exp\lf\{-\nu'\lf[\frac{d(z,x')}{\dz^k}\r]^a\r\}\r)\,d\mu(z).
\end{align*}
Noticing that $k\le j$, by \eqref{eq-add8}, we have
\begin{align*}
&\min\lf\{1,\lf[\frac{d(x,x')}{\dz^k}\r]^\eta\r\}\exp\lf\{-\nu'\lf[\frac{d(z,\CY^k)}{\dz^k}\r]^a\r\}
\exp\lf\{-\frac{\nu'}{4}\lf[\frac{d(y,z)}{\dz^j}\r]^a\r\}\\
&\quad\ls\lf[\frac{d(x,x')}{\dz^k}\r]^\eta \exp\lf\{-\nu'\lf[\frac{d(z,\CY^k)}{\dz^k}\r]^a\r\}
\exp\lf\{-\frac{\nu'}{4}\lf[\frac{d(y,z)}{\dz^k}\r]^a\r\}\\
&\quad\ls\lf[\frac{d(x,x')}{\dz^k}\r]^\eta \exp\lf\{-\frac{\nu'}{4}\lf[\frac{d(y,\CY^k)}{A_0\dz^k}\r]^a\r\}.
\end{align*}
By \eqref{eq:doub}, we obtain
$
\frac1{V_{\dz^k}(z)}\ls \frac1{V_{\dz^k}(x)} [1+\frac{d(x,z)}{\dz^k}]^\omega,
$
which, together with $d(x,x')\ls d(x,y)$, further implies that
\begin{align*}
&\frac1{V_{\dz^k}(z)}\lf(\exp\lf\{-\nu'\lf[\frac{d(z,x)}{\dz^k}\r]^a\r\}
+\exp\lf\{-\nu'\lf[\frac{d(z,x')}{\dz^k}\r]^a\r\}\r)\\
&\quad\ls \frac1{V_{\dz^k}(x)}\lf[1+\frac{d(x,y)}{\dz^k}\r]^\omega \lf(\exp\lf\{-\frac{\nu'}2\lf[\frac{d(z,x)}{\dz^k}\r]^a\r\}
+\exp\lf\{-\frac{\nu'}2\lf[\frac{d(z,x')}{\dz^k}\r]^a\r\}\r).
\end{align*}
As in the estimation of \eqref{eq-add7}, we always have
\begin{align*}
&\exp\lf\{-\frac{\nu'}2\lf[\frac{d(y,z)}{\dz^j}\r]^a\r\}\lf(\exp\lf\{-\frac{\nu'}2\lf[\frac{d(z,x)}{\dz^k}\r]^a\r\}
+\exp\lf\{-\frac{\nu'}2\lf[\frac{d(z,x')}{\dz^k}\r]^a\r\}\r)\\
&\quad\le \exp\lf\{-\frac{\nu'}2\lf[\frac{d(y,x)}{A_0\dz^k}\r]^a\r\}
+\exp\lf\{-\frac{\nu'}2\lf[\frac{d(y,x')}{A_0\dz^k}\r]^a\r\} \le 2\exp\lf\{-\frac{\nu'}2\lf[\frac{d(y,x)}{2A_0^2\dz^k}\r]^a\r\}.
\end{align*}
Applying the above estimates, Lemma \ref{lem-add}(ii) and \eqref{eq-add8}, we conclude that
\begin{align*}
&\lf|\wz{Q}_jQ_k(y,x)-\wz{Q}_jQ_k(y,x')\r|\\
&\quad\ls \lf[\frac{d(x,x')}{\dz^k}\r]^\eta
\frac1{V_{\dz^k}(x)}\exp\lf\{-\frac{\nu'}4\lf[\frac{d(y,\CY^k)}{\dz^k}\r]^a\r\}
\exp\lf\{-\frac{\nu'}4\lf[\frac{d(y,x)}{2A_0^2\dz^k}\r]^a\r\}\\
&\qquad\times\int_X \frac1{V_{\dz^j}(y)}
\exp\lf\{-\frac{\nu'} 4\lf[\frac{d(y,z)}{\dz^j}\r]^a\r\}\,d\mu(z)\\
&\quad\ls \lf[\frac{d(x,x')}{\dz^k}\r]^\eta \frac1{V_{\dz^k}(x)}
\exp\lf\{-c\lf[\frac{d(x,\CY^k)}{\dz^k}\r]^a\r\}\exp\lf\{-c\lf[\frac{d(y,x)}{\dz^k}\r]^a\r\}.
\end{align*}
Taking the  geometric means between the last estimate and \eqref{eq-add66}, we obtain the desired estimate of
$|\wz{Q}_jQ_k(y,x)-\wz{Q}_jQ_k(y,x')|$ in \eqref{eq:mixregx}.

Finally, we prove \eqref{eq:mixdreg}.
If $d(x,x')\le(2A_0)^{-3}d(x,y)$ and $d(y,y')\le(2A_0)^{-3}d(x,y)$ with $x\neq y$, by \eqref{eq:mixsize} and
(i) and (ii) of Remark \ref{rem:andef}, we conclude that
\begin{align}\label{eq:*1}
&\lf|\lf[\wz{Q}_jQ_k(x,y)-\wz{Q}_jQ_k(x',y)\r]-\lf[\wz{Q}_jQ_k(x,y')-\wz{Q}_jQ_k(x',y')\r]\r|\\
&\quad\ls\dz^{(j-k)\eta}\frac{1}{V_{\dz^k}(x)}\exp\lf\{-c\lf[\frac{d(x,y)}{\dz^k}\r]^a\r\}
\exp\lf\{-c\lf[\frac{d(y,\CY^k)}{\dz^k}\r]^a\r\}.\noz
\end{align}
Once we have proved that, when $d(x,x')\le(2A_0)^{-3}d(x,y)$ and $d(y,y')\le(2A_0)^{-3}d(x,y)$,
\begin{align}\label{eq:mixdreg1}
&\lf|\lf[\wz{Q}_jQ_k(x,y)-\wz{Q}_jQ_k(x',y)\r]-\lf[\wz{Q}_jQ_k(x,y')-\wz{Q}_jQ_k(x',y')\r]\r|\\
&\quad\ls\lf[\frac{d(x,x')}{\dz^k}\r]^\eta\lf[\frac{d(y,y')}{\dz^k}\r]^\eta\frac{1}{V_{\dz^k}(x)}
\exp\lf\{-c\lf[\frac{d(x,y)}{\dz^k}\r]^a\r\}\exp\lf\{-c\lf[\frac{d(y,\CY^k)}{\dz^k}\r]^a\r\},\noz
\end{align}
then taking the geometric mean between \eqref{eq:*1} and \eqref{eq:mixdreg1} gives \eqref{eq:mixdreg}.
Thus, it remains to prove \eqref{eq:mixdreg1}.

To show \eqref{eq:mixdreg1}, we observe that
\begin{align*}
&\lf|\lf[\wz{Q}_jQ_k(x,y)-\wz{Q}_jQ_k(x',y)\r]-\lf[\wz{Q}_jQ_k(x,y')-\wz{Q}_jQ_k(x',y')\r]\r|\\
&\quad\le\int_X\lf|\wz{Q}_j(x,z)-\wz{Q}_j(x',z)\r||Q_k(x,y)-Q_k(z,y)-Q_k(x,y')+Q_k(z,y')|\,d\mu(z)\\
&\quad=\sum_{i=1}^3\int_{W_i}\lf|\wz{Q}_j(x,z)-\wz{Q}_j(x',z)\r|
|Q_k(x,y)-Q_k(z,y)-Q_k(x,y')+Q_k(z,y')|\,d\mu(z)=:\sum_{i=1}^3\RZ_i,
\end{align*}
where
\begin{align*}
W_1&:=\{z\in X:\ d(x,z)<2A_0d(x,x')\},\\
W_2&:=\{z\in X:\ 2A_0d(x,x')\le d(x,z)\le  (2A_0)^{-2}d(x,y)\}
\end{align*}
and
$$
W_3:=\{z\in X:\ d(x,z)\ge(2A_0)^{-2}d(x,y)\}.
$$
Notice that $d(x,z)<2A_0d(x,x')\le(2A_0)^{-2}d(x,y)$ when $z\in W_1$. By this, (i) and (iii) of Remark
\ref{rem:andef} for $Q_k$ and Lemma \ref{lem-add}(iii), we conclude that
\begin{align*}
\RZ_1&\ls\lf[\frac{d(y,y')}{\dz^k}\r]^\eta\frac 1{V_{\dz^k}(x)}
\exp\lf\{-\frac{\nu'}2\lf[\frac{d(x,y)}{\dz^k}\r]^a\r\}
\exp\lf\{-\nu'\lf[\frac{d(x,\CY^k)}{\dz^k}\r]^a\r\}\\
&\quad\times\int_{W_1}\lf[\lf|\wz{Q}_j(x,z)\r|+\lf|\wz{Q}_j(x',z)\r|\r]\lf[\frac{d(x,z)}{\dz^k}\r]^\eta
\,d\mu(z)\\
&\ls\lf[\frac{d(x,x')}{\dz^k}\r]^\eta\frac 1{V_{\dz^k}(x)}\lf[\frac{d(y,y')}{\dz^k}\r]^\eta
\exp\lf\{-\nu'\lf[\frac{d(x,y)}{\dz^k}\r]^a\r\}\exp\lf\{-\frac{\nu'}2\lf[\frac{d(x,\CY^k)}{\dz^k}\r]^a\r\}.
\end{align*}

From (i) and (iii) of Remark \ref{rem:andef} for $Q_k$ and $\wz{Q}_j$, together with Lemma \ref{lem-add}(ii),
we deduce that
\begin{align*}
\RZ_2&\ls\lf[\frac{d(y,y')}{\dz^k}\r]^\eta\frac 1{V_{\dz^k}(x)}\exp\lf\{-\nu'\lf[\frac{d(x,y)}{\dz^k}\r]^a\r\}
\exp\lf\{-\nu'\lf[\frac{d(x,\CY^k)}{\dz^k}\r]^a\r\}\\
&\quad\times\int_{W_2}\lf[\frac{d(x,x')}{\dz^j}\r]^\eta\frac{1}{V_{\dz^j}(x)}
\exp\lf\{-\nu'\lf[\frac{d(x,z)}{\dz^j}\r]^a\r\}\lf[\frac{d(x,z)}{\dz^k}\r]^\eta\,d\mu(z)\\
&\ls\lf[\frac{d(x,x')}{\dz^k}\r]^\eta\lf[\frac{d(y,y')}{\dz^k}\r]^\eta\frac 1{V_{\dz^k}(x)}
\exp\lf\{-\nu'\lf[\frac{d(x,y)}{\dz^k}\r]^a\r\}\exp\lf\{-\nu'\lf[\frac{d(x,\CY^k)}{\dz^k}\r]^a\r\}.
\end{align*}

To estimate $\RZ_3$, we claim that, for any $z\in X$,
\begin{align}\label{eq-x0}
&|[Q_k(x,y)-Q_k(x,y')]-[Q_k(z,y)-Q_k(z,y')]|\\
&\quad\ls\lf[\frac{d(y,y')}{\dz^k}\r]^\eta\lf[\frac{d(x,z)}{\dz^k}\r]^\eta
\lf[\frac{1}{V_{\dz^k}(y)}
\exp\lf\{-\nu'\lf[\frac{d(x,y)}{\dz^k}\r]^a\r\}\exp\lf\{-\nu'\lf[\frac{d(y,\CY^k)}{\dz^k}\r]^a\r\}\r.\noz\\
&\qquad
+\frac{1}{V_{\dz^k}(y')}
\exp\lf\{-\nu'\lf[\frac{d(x,y')}{\dz^k}\r]^a\r\}\exp\lf\{-\nu'\lf[\frac{d(y',\CY^k)}{\dz^k}\r]^a\r\}\noz\\
&\qquad
+\frac{1}{V_{\dz^k}(y)}
\exp\lf\{-\nu'\lf[\frac{d(z,y)}{\dz^k}\r]^a\r\}\exp\lf\{-\nu'\lf[\frac{d(y,\CY^k)}{\dz^k}\r]^a\r\}\noz\\
&\qquad
\lf.+\frac{1}{V_{\dz^k}(y')}
\exp\lf\{-\nu'\lf[\frac{d(z,y')}{\dz^k}\r]^a\r\}\exp\lf\{-\nu'\lf[\frac{d(y',\CY^k)}{\dz^k}\r]^a\r\}\r].\noz
\end{align}

Indeed, if  $d(y,y')\le\dz^k$ and $d(x,z)\le \dz^k$, then  \eqref{eq-x0} follows from the second difference
regularity condition of $Q_k$ and Remark \ref{rem:andef}(i).
If $d(y,y')>\dz^k$ and $d(x,z)>\dz^k$, then we use the size condition of $Q_k$ and Remark \ref{rem:andef}(i)
to obtain \eqref{eq-x0}. If $d(y,y')\le\dz^k$ and $d(x,z)> \dz^k$, then, by the regularity of $Q_k$ and Remark
\ref{rem:andef}(i), we have
\begin{align}\label{eq-x1}
&|[Q_k(x,y)-Q_k(x,y')]-[Q_k(z,y)-Q_k(z,y')]|\\
&\quad\le |Q_k(x,y)-Q_k(x,y')|+|Q_k(z,y)-Q_k(z,y')|\noz\\
&\quad\ls\lf[\frac{d(y,y')}{\dz^k}\r]^\eta \exp\lf\{-\nu'\lf[\frac{d(y,\CY^k)}{\dz^k}\r]^a\r\}\frac{1}{V_{\dz^k}(y)}
\noz\\
&\qquad\times\lf(\exp\lf\{-\nu'\lf[\frac{d(x,y)}{\dz^k}\r]^a\r\}+\exp\lf\{-\nu'\lf[\frac{d(z,y)}{\dz^k}\r]^a\r\}\r).\noz
\end{align}
Using $d(x,z)> \dz^k$, we directly multiply $[d(x,z)/\dz^k]^\eta$ in the right-hand side
of \eqref{eq-x1} to obtain \eqref{eq-x0}. When $d(y,y')>\dz^k$ and  $d(x,z)\le \dz^k$, a symmetric argument
also leads to \eqref{eq-x0}.

The treatment for the four terms in the bracket of the right-hand side of \eqref{eq-x0} is similar. We only
consider the last term, which is also the
most difficult one. Thus, the estimation of $\RZ_3$ is reduced to the estimation of
\begin{align*}
\wz\RZ_3
:={}&\int_{W_3} \lf|\wz{Q}_j(x,z)-\wz{Q}_j(x',z)\r| \lf[\frac{d(y,y')}{\dz^k}\r]^\eta\lf[\frac{d(x,z)}{\dz^k}\r]^\eta \\
&\times
\frac{1}{V_{\dz^k}(y')}
\exp\lf\{-\nu'\lf[\frac{d(z,y')}{\dz^k}\r]^a\r\}\exp\lf\{-\nu'\lf[\frac{d(y',\CY^k)}{\dz^k}\r]^a\r\}\,d\mu(z).
\end{align*}
When $z\in W_3$, we have $d(x,x')\le(2A_0)^{-3}d(x,y)\le (2A_0)^{-1}d(x,z)$, so that we can apply (i) and
(iii) of Remark \ref{rem:andef} for $\wz Q_j$, \eqref{eq-add7} and Lemma \ref{lem-add}(ii) to conclude that
\begin{align*}
\wz\RZ_3
&\ls \int_{W_3} \lf[\frac{d(x,x')}{\dz^j}\r]^\eta \lf[\frac{d(y,y')}{\dz^k}\r]^\eta\lf[\frac{d(x,z)}{\dz^k}\r]^\eta \frac{1}{V_{\dz^j}(x)V_{\dz^k}(y')} \\
&\quad\times
\exp\lf\{-\nu'\lf[\frac{d(x,z)}{\dz^j}\r]^a\r\}
\exp\lf\{-\nu'\lf[\frac{d(z,y')}{\dz^k}\r]^a\r\}\exp\lf\{-\nu'\lf[\frac{d(y',\CY^k)}{\dz^k}\r]^a\r\}\,d\mu(z)\\
&\ls\lf[\frac{d(x,x')}{\dz^k}\r]^\eta \lf[\frac{d(y,y')}{\dz^k}\r]^\eta
\frac{1}{V_{\dz^k}(y')} \exp\lf\{-\frac{\nu'} 2\lf[\frac{d(x,y')}{A_0\dz^k}\r]^a\r\}\exp\lf\{-\nu'\lf[\frac{d(y',\CY^k)}{\dz^k}\r]^a\r\}.
\end{align*}
By $d(y,y')\le (2A_0)^{-2}d(x,y)$, we have $(\frac 43A_0)^{-1}d(x,y)\le d(x,y')\le \frac 43A_0 d(x,y)$ and hence
$$d(x, \CY^k)\le A_0^2[d(x,y)+d(y,y')+d(y',\CY^k)]\le 2A_0^2 [d(x,y)+d(y',\CY^k)],$$
which further implies that
$$
\exp\lf\{-\frac{\nu'} 2\lf[\frac{d(x,y')}{A_0\dz^k}\r]^a\r\}\exp\lf\{-\nu'\lf[\frac{d(y',\CY^k)}{\dz^k}\r]^a\r\}
\le \exp\lf\{-\frac{\nu'} 8\lf[\frac{d(x,y)}{2A_0^2\dz^k}\r]^a\r\}\exp\lf\{-\frac{\nu'} 8\lf[\frac{d(x,\CY^k)}{2A_0^2\dz^k}\r]^a\r\}.
$$
Also, the condition $d(y,y')\le (2A_0)^{-2}d(x,y)$ and \eqref{eq:doub} imply that
$$
\frac{1}{V_{\dz^k}(y')}\ls \frac{1}{V_{\dz^k}(x)} \lf[\frac{\dz^k+d(y',x)}{\dz^k}\r]^\omega \ls
\frac{1}{V_{\dz^k}(y)} \lf[1+\frac{d(x,y)}{\dz^k}\r]^\omega.
$$
Therefore, we obtain
\begin{align*}
\wz\RZ_3
\ls\lf[\frac{d(x,x')}{\dz^k}\r]^\eta \lf[\frac{d(y,y')}{\dz^k}\r]^\eta
\frac{1}{V_{\dz^k}(x)} \exp\lf\{-c\lf[\frac{d(x,y)}{\dz^k}\r]^a\r\}\exp\lf\{-c\lf[\frac{d(x,\CY^k)}{\dz^k}\r]^a\r\},
\end{align*}
so does $\RZ_3$. Combining the estimates of $\RZ_1$, $\RZ_2$ and $\RZ_3$, we obtain
\eqref{eq:mixdreg1} and hence \eqref{eq:mixdreg}. This finishes the proof of Lemma \ref{lem:ccrf1}.
\end{proof}

\begin{corollary}\label{cor:mixb}
Let all the notation be as in Lemma \ref{lem:ccrf1}. Then $\wz{Q}_jQ_k$ satisfies the following estimates:
\begin{enumerate}
\item if $d(x,x')\le(2A_0)^{-1}d(x,y)$ with $x\neq y$, then
\begin{align}\label{eq:mixregxb}
&\lf|\wz{Q}_jQ_k(x,y)-\wz{Q}_jQ_k(x',y)\r|+\lf|\wz{Q}_jQ_k(y,x)-\wz{Q}_jQ_k(y,x')\r|\\
&\quad\le C\dz^{|j-k|(\eta-\eta')}\lf[\frac{d(x,x')}{d(x,y)}\r]^{\eta'}
\frac 1{V_{\dz^{j\wedge k}}(x)}\exp\lf\{-c'\lf[\frac{d(x,y)}{\dz^{j\wedge k}}\r]^a\r\}\noz\\
&\qquad\times\exp\lf\{-c'\lf[\frac{d(x,\CY^{j\wedge k})}{\dz^{j\wedge k}}\r]^a\r\};\noz
\end{align}
\item if $d(x,x')\le(2A_0)^{-2}d(x,y)$ and $d(y,y')\le(2A_0)^{-2}d(x,y)$ with $x\neq y$, then
\begin{align}\label{eq:mixdregb}
&\lf|\lf[\wz{Q}_jQ_k(x,y)-\wz{Q}_jQ_k(x',y)\r]-\lf[\wz{Q}_jQ_k(x,y')-\wz{Q}_jQ_k(x',y')\r]\r|\\
&\noz\quad\le C\dz^{|j-k|(\eta-\eta')}\lf[\frac{d(x,x')}{\dz^{j\wedge k}}\r]^{\eta'}
\lf[\frac{d(y,y')}{\dz^{j\wedge k}}\r]^{\eta'}
\frac 1{V_{\dz^{j\wedge k}}(x)}\exp\lf\{-c'\lf[\frac{d(x,y)}{\dz^{j\wedge k}}\r]^a\r\}\\
&\qquad\times\exp\lf\{-c'\lf[\frac{d(x,\CY^{j\wedge k})}{\dz^{j\wedge k}}\r]^a\r\},
\noz
\end{align}
\end{enumerate}
where $C$ and $c'$ are positive constants independent of $k,\ j\in\zz$ and $x,\ y,\ x',\ y'\in X$.
\end{corollary}

\begin{proof}
For (i), if $d(x,x')\le(2A_0)^{-2}d(x,y)$, then \eqref{eq:mixregxb} holds true by \eqref{eq:mixregx} and
\begin{equation}\label{eq:*2}
\lf[\frac{d(x,y)}{\dz^{j\wedge k}}\r]^\eta\exp\lf\{-c\lf[\frac{d(x,y)}{\dz^{j\wedge k}}\r]^a\r\}
\ls\exp\lf\{-c'\lf[\frac{d(x,y)}{\dz^{j\wedge k}}\r]^a\r\}.
\end{equation}
If $(2A_0)^{-2}d(x,y)<d(x,x')\le(2A_0)^{-1}d(x,y)$, then \eqref{eq:mixregxb} remains true by
using \eqref{eq:mixsize} and \eqref{eq:*2}. The proof of (ii) is similar by using (i) and \eqref{eq:*2}.
This finishes the proof of Corollary \ref{cor:mixb}.
\end{proof}

\begin{remark}\label{rem:mix}
Checking the proof of Lemma \ref{lem:ccrf1} and Corollary \ref{cor:mixb}, we observe that, if $j\ge k$, then
we only use the cancellation of $\wz{Q}_j$ while the cancellation of $Q_k$ is not necessary. This observation
will be used later in Section \ref{irf}.
\end{remark}

\subsection{Boundedness of the remainder $R_N$}\label{RN}

For any Banach space $\CB$, denote by $\CL(\CB)$ the set of all bounded linear operators on $\CB$.
The aim of this section is to estimate the operator norms of the remainder $R_N$ on both $L^2(X)$ and
$\GO{\bz,\gz}$. To this end, we need the following two lemmas.

\begin{lemma}[{\cite{CW77}}]\label{lem:HL}
For any given $p\in(1,\fz]$, the Hardy-Littlewood maximal operator $\CM$ defined in \eqref{eq:defmax} is
bounded on $L^p(X)$.
\end{lemma}

\begin{lemma}[{\cite[pp.\ 279--280]{Stein93}}]\label{lem:CSlem}
Let
$\{\gamma(j)\}_{j=-\fz}^\fz\subset(0,\fz)$ be such that $A:=\sum_{j=-\fz}^\fz \gz(j)<\fz$.
Suppose that $\{T_j\}_{j=-\fz}^\fz$ is a sequence of bounded linear operators on $L^2(X)$  satisfying that,
for any $j,\ k\in\zz$,
$$
\lf\|T_j^*T_k\r\|_{L^2(X)\to L^2(X)}\le[\gz(j-k)]^2\quad\text{and}\quad
\lf\|T_jT_k^*\r\|_{L^2(X)\to L^2(X)}\le[\gz(j-k)]^2.
$$
Then, for any $f\in L^2(X)$, the series $\sum_{j=-\fz}^\fz T_jf$ converges to an element in $L^2(X)$, denoted by $Tf$.
Moreover, the operator $T$ is bounded on $L^2(X)$ and $\|T\|_{L^2(X)\to L^2(X)}\le A$.
\end{lemma}
We first establish the boundedness of $R_N$ on $L^2(X)$ and, moreover, estimate its operator norm.
\begin{lemma}\label{lem:ccrf2}
For any $n\in\nn$, suppose that $R_N$ is defined as in \eqref{eq:defRT} and $\eta'\in(0,\eta)$. Then there
exists a positive constant $C$, independent of $N$, such that
\begin{equation}\label{eq:RNL2}
\|R_N\|_{L^2(X)\to L^2(X)}\le C\dz^{\eta' N}.
\end{equation}
\end{lemma}
\begin{proof}
By \eqref{eq:mixsize}, Proposition \ref{prop:basic}(ii) and Lemma \ref{lem:HL}, we know that, for any
$k,\ l\in\zz$ and $f\in L^2(X)$,
$$
\lf\|Q_{k+l}Q_k(f)\r\|_{L^2(X)}\ls\dz^{\eta|l|}\|\CM(f)\|_{L^2(X)}\ls\dz^{\eta|l|}\|f\|_{L^2(X)}.
$$
Therefore, for any $k_1,\ k_2,\ l_1,\ l_2\in\zz$, we have
$$
\lf\|Q_{k_1+l_1}Q_{k_1}\lf(Q_{k_2+l_2}Q_{k_2}\r)^*\r\|_{L^2(X)\to L^2(X)}\ls\dz^{|l_1|\eta}\dz^{|l_2|\eta}.
$$
On the other hand,
$$
\lf\|Q_{k_1+l_1}Q_{k_1}\lf(Q_{k_2+l_2}Q_{k_2}\r)^*\r\|_{L^2(X)\to L^2(X)}
=\lf\|Q_{k_1+l_1}Q_{k_1}Q_{k_2}^*Q_{k_2+l_2}^*\r\|_{L^2(X)\to L^2(X)}\ls\dz^{|k_1-k_2|\eta}.
$$
Taking the geometric means between the above two formulae, we conclude that, for any $\thz\in(0,1)$,
\begin{equation*}
\lf\|Q_{k_1+l_1}Q_{k_1}\lf(Q_{k_2+l_2}Q_{k_2}\r)^*\r\|_{L^2(X)\to L^2(X)}
\ls\dz^{(|l_1|+|l_2|)\eta\thz}\dz^{|k_1-k_2|\eta(1-\thz)}.
\end{equation*}
Consequently,
\begin{align}\label{eq:**}
&\lf\|\lf(\sum_{|l_1|>N}Q_{k_1+l_1}Q_{k_1}\r)\lf(\sum_{|l_2|>N}Q_{k_2+l_2}Q_{k_2}\r)^*\r\|_{L^2(X)\to L^2(X)}\\
&\quad\le\sum_{|l_1|>N}\sum_{|l_2|>N}
\lf\|Q_{k_1+l_1}Q_{k_1}\lf(Q_{k_2+l_2}Q_{k_2}\r)^*\r\|_{L^2(X)\to L^2(X)}\noz\\
&\quad\ls\sum_{|l_1|>N}\sum_{|l_2|>N}\dz^{(|l_1|+|l_2|)\eta\thz}\dz^{|k_1-k_2|\eta(1-\thz)}\noz
\sim \dz^{2N\eta\thz}\dz^{|k_1-k_2|\eta(1-\thz)}.
\end{align}
Combining this with Lemma \ref{lem:CSlem} implies that
$
\|R_N\|_{L^2(X)\to L^2(X)}\ls \dz^{N\eta\thz}.
$
Taking $\eta':=\eta\thz$, we then complete the proof of Lemma \ref{lem:ccrf2}.
\end{proof}

To consider the boundedness of $R_N$ on spaces of test functions, we begin with the following proposition.

\begin{proposition}\label{prop:sizeRN}
For any $N\in\nn$, suppose that $R_N$ is defined as in \eqref{eq:defRT} and $\eta'\in(0,\eta)$. Then
$R_N$ satisfies \eqref{eq:Ksize}, \eqref{eq:Kreg} and \eqref{eq:Kdreg}
with $C_T$ therein replaced by $C\dz^{(\eta-\eta')N}$, where $C$ is a positive
constant $C$ independent of $N$.
\end{proposition}

The proof of Proposition \ref{prop:sizeRN} is based on the following several lemmas.
\begin{lemma}[{\cite[Lemma 8.3]{AH13}}]\label{lem:expsum}
For any fixed $a,\ c\in(0,\fz)$, there exists a positive constant $C$ such that, for any $r\in(0,\fz)$
and $x\in X$,
$$
\sum_{\dz^k\ge r}\frac 1{V_{\dz^k}(x)}\exp\lf\{-c\lf[\frac{d(x,\CY^k)}{\dz^k}\r]^a\r\}
\le \frac C{V_r(x)}.
$$
\end{lemma}
\begin{lemma}\label{lem:sum2}
For any fixed $a,\ c\in(0,\fz)$, there exists a positive constant $C$ such that, for any $x,\ y\in X$
with $x\neq y$,
$$
\sum_{k=-\fz}^\fz\frac 1{V_{\dz^k}(x)}\exp\lf\{-c\lf[\frac{d(x,y)}{\dz^k}\r]\r\}
\exp\lf\{-c\lf[\frac{d(x,\CY^k)}{\dz^k}\r]^a\r\}
\le \frac C{V(x,y)}.
$$
\end{lemma}
\begin{proof}
Take $r:=d(x,y)$. Due to Lemma \ref{lem:expsum}, to show this lemma, it suffices to prove that
\begin{equation*}
\RY:=\sum_{\dz^k<d(x,y)}\frac 1{V_{\dz^k}(x)}\exp\lf\{-c\lf[\frac{d(x,y)}{\dz^k}\r]^a\r\}\ls\frac 1{V(x,y)}.
\end{equation*}
Take $\Gamma\in(\omega,\fz)$. Indeed, by \eqref{eq:doub}, we have
\begin{equation*}
\RY\ls\sum_{\dz^k<d(x,y)}\frac{1}{V(x,y)}\frac{V(x,y)}{V_{\dz^k}(x)}\lf[\frac{\dz^k}{d(x,y)}\r]^\Gamma
\ls\frac{1}{V(x,y)}\sum_{\dz^k<d(x,y)}\lf[\frac{\dz^k}{d(x,y)}\r]^{\Gamma-\omega}
\sim\frac{1}{V(x,y)}.
\end{equation*}
This finishes the proof of Lemma \ref{lem:sum2}.
\end{proof}

\begin{proof}[Proof of Proposition \ref{prop:sizeRN}]
For any $N\in\nn$, we first prove that $R_N$ satisfies \eqref{eq:Ksize}. By \eqref{eq:mixsize} and Lemma
\ref{lem:sum2}, we conclude that, for any $x,\ y\in X$ with $x\neq y$,
\begin{align*}
|R_N(x,y)|&\le\sum_{k=-\fz}^\fz\sum_{|l|>N}|Q_{k+l}Q_{k}(x,y)|\\
&\ls\sum_{l=N+1}^\fz\dz^{\eta l}\sum_{k=-\fz}^\fz
\frac 1{V_{\dz^k}(x)}\exp\lf\{-c\lf[\frac{d(x,y)}{\dz^k}\r]^a\r\}
\exp\lf\{-c\lf[\frac{d(x,\CY^k)}{\dz^k}\r]^a\r\}\\
&\quad+\sum_{l=-\fz}^{-N-1}\dz^{-\eta l}\sum_{k=-\fz}^\fz
\frac 1{V_{\dz^{k+l}}(x)}\exp\lf\{-c\lf[\frac{d(x,y)}{\dz^{k+l}}\r]^a\r\}
\exp\lf\{-c\lf[\frac{d(x,\CY^{k+l})}{\dz^{k+l}}\r]^a\r\}\ls\dz^{\eta N}\frac 1{V(x,y)}.
\end{align*}

Next we prove \eqref{eq:Kreg}. Suppose that $d(x,x')\le (2A_0)^{-1}d(x,y)$ with $x\neq y$. Then, from
\eqref{eq:mixregxb} and Lemma \ref{lem:sum2}, we deduce that
\begin{align*}
&|R_N(x,y)-R_N(x',y)|+|R_N(y,x)-R_N(y,x')|\\
&\quad\le\sum_{|l|>N}^\fz\sum_{k=-\fz}^\fz
[|Q_{k+l}Q_{k}(x,y)-Q_{k+l}Q_k(x',y)|+|Q_{k+l}Q_{k}(y,x)-Q_{k+l}Q_k(y,x')|]\\
&\quad\ls\lf[\frac{d(x,x')}{d(x,y)}\r]^{\eta'}\lf[\sum_{l=N+1}^\fz\dz^{(\eta-\eta') l}\sum_{k=-\fz}^\fz
\frac 1{V_{\dz^k}(x)}\exp\lf\{-c'\lf[\frac{d(x,y)}{\dz^k}\r]^a\r\}
\exp\lf\{-c'\lf[\frac{d(x,\CY^k)}{\dz^k}\r]^a\r\}\r.\\
&\qquad\lf.+\sum_{l=-\fz}^{-N-1}\dz^{-(\eta-\eta')l}\sum_{k=-\fz}^\fz
\frac 1{V_{\dz^{k+l}}(x)}\exp\lf\{-c'\lf[\frac{d(x,y)}{\dz^{k+l}}\r]^a\r\}
\exp\lf\{-c'\lf[\frac{d(x,\CY^{k+l})}{\dz^{k+l}}\r]^a\r\}\r]\\
&\quad\ls\dz^{(\eta-\eta') N}\lf[\frac{d(x,x')}{d(x,y)}\r]^{\eta'}\frac 1{V(x,y)}.
\end{align*}

Finally, we show \eqref{eq:Kdreg}. By \eqref{eq:mixdregb} and Lemma \ref{lem:sum2}, we have
\begin{align*}
&|[R_N(x,y)-R_N(x',y)]-[R_N(x,y')-R_N(x',y')]|\\
&\quad\le\sum_{|l|>N}^\fz\sum_{k=-\fz}^\fz
|[Q_{k+l}Q_{k}(x,y)-Q_{k+l}Q_k(x',y)]-[Q_{k+l}Q_{k}(x,y')-Q_{k+l}Q_k(x',y')]|\\
&\quad\ls\lf[\frac{d(x,x')}{d(x,y)}\r]^{\eta'}\lf[\frac{d(y,y')}{d(x,y)}\r]^{\eta'}\\
&\qquad\times\lf[\sum_{l=N+1}^\fz\dz^{(\eta-\eta') l}\sum_{k=-\fz}^\fz
\frac 1{V_{\dz^k}(x)}\exp\lf\{-c'\lf[\frac{d(x,y)}{\dz^k}\r]^a\r\}
\exp\lf\{-c'\lf[\frac{d(x,\CY^k)}{\dz^k}\r]^a\r\}\r.\\
&\qquad\lf.+\sum_{l=-\fz}^{-N-1}\dz^{-(\eta-\eta')l}\sum_{k=-\fz}^\fz
\frac 1{V_{\dz^{k+l}}(x)}\exp\lf\{-c'\lf[\frac{d(x,y)}{\dz^{k+l}}\r]^a\r\}
\exp\lf\{-c'\lf[\frac{d(x,\CY^{k+l})}{\dz^{k+l}}\r]^a\r\}\r]\\
&\quad\ls\dz^{(\eta-\eta') N}\lf[\frac{d(x,x')}{d(x,y)}\r]^{\eta'}\lf[\frac{d(y,y')}{d(x,y)}\r]^{\eta'}
\frac 1{V(x,y)}.
\end{align*}
This finishes the proof of Proposition \ref{prop:sizeRN}.
\end{proof}

Let  $x_1\in X$, $r\in(0,\fz)$ and $\bz,\ \gz\in(0,\eta)$.
To prove the boundedness of $R_N$ on $\mathring{\CG}(x_1,r,\bz,\gz)$, we cannot use Theorem \ref{thm:Kbdd}
directly, since it is not clear whether or not $R_N$ satisfies conditions (b) and (c) of Theorem
\ref{thm:Kbdd}. To overcome this difficulty, for any $M\in\nn$, define
\begin{equation}\label{eq:defRNM}
R_{N,M}:=\sum_{|k|\le M}\sum_{N<|l|\le M}Q_{k+l}Q_k.
\end{equation}
Clearly, it is easy to see that, for any $f\in C^\bz(X)$ with $\bz\in(0,\eta]$ and $x\in X$,
$$
R_{N,M}f(x)=\int_X R_{N,M}(x,y)f(y)\,d\mu(y).
$$
Moreover, for any $x\in X$,
$$
\int_X R_{N,M}(x,y)\,d\mu(y)=0=\int_X R_{N,M}(y,x)\,d\mu(y).
$$
Notice that Lemma \ref{lem:ccrf2} and Proposition \ref{prop:sizeRN} hold true with $R_N$ replaced by $R_{N,M}$, with
all the constants involved independent of $M$ and $N$. Besides,
since $\int_{X} R_{N,M}(x,y)\,d\mu(x)=0$ for any $y\in X$, from the Fubini theorem, it
follows that $\int_X R_{N,M}f(x)\,d\mu(x)=0$. Applying these and
Theorem \ref{thm:Kbdd}, we know that, for any $f\in\mathring{\CG}(x_1,r,\bz,\gz)$,
\begin{equation}\label{eq:RNMG}
\|R_{N,M}f\|_{\mathring\CG(x_1,r,\bz,\gz)}\ls\dz^{(\eta-\eta')N}\|f\|_{\mathring\CG(x_1,r,\bz,\gz)}
\end{equation}
with $\eta'\in(\max\{\bz,\gz\},\eta)$, where the implicit positive constant is independent of $M, N$, $x_1$ and
$r$. To pass form $R_{N,M}$ in \eqref{eq:RNMG} to $R_N$, we consider the relationship between $R_Nf$ and
$R_{N,M}f$.

\begin{lemma}\label{lem:ccrf3} Let $f\in\CG(x_1,r,\bz,\gz)$ with $x_1\in X$, $r\in(0,\fz)$ and $\bz,\ \gz\in(0,\eta)$.
Fix $N\in\nn$. Then the following assertions hold true:
\begin{enumerate}
\item $\lim_{M\to\fz} R_{N,M}f=R_Nf$ in $L^2(X)$;
\item for any $x\in X$, the sequence $\{R_{N,M}f(x)\}_{M=N+1}^\infty$ converges locally uniformly to some
element, denoted by $\widetilde{R_N} f(x)$, where $\widetilde{R_N }f$ differs from $R_N f$ at most on a set of
$\mu$-measure $0$;
\item the operator $\widetilde {R_N}$ can be uniquely extended from $\CG(x_1,r,\bz,\gz)$ to $L^2(X)$, with the
extension operator coincides with $R_N$. In this sense, for any $f\in\CG(x_1,r,\bz,\gz)$ and almost every
$x\in X$,
\begin{equation*}
\lim_{M\to\fz} R_{N,M}f(x)=\widetilde{R_N}f(x)=R_Nf(x).
\end{equation*}
\end{enumerate}
\end{lemma}
\begin{proof}
We first prove (i). We only need to prove that, for any $f\in L^2(X)$,
\begin{equation}\label{eq:limL2}
\lim_{M\to\fz} \lf\|R_{N,M}f-R_Nf\r\|_{L^2(X)}=0.
\end{equation}
Indeed, when $M>N$, write
$$
R_N-R_{N,M}=\sum_{|k|\le M}\sum_{|l|>M}Q_{k+l}Q_k+\sum_{|k|>M}\sum_{|l|>N}Q_{k+l}Q_k
$$
An argument similar to that used in the proof of \eqref{eq:**} implies that, for any $k_1,\ k_2\in\zz$ and
$\thz\in(0,1)$,
$$
\lf\|\lf(\sum_{|l_1|>M}Q_{k_1+l_1}Q_{k_1}\r)^*\lf(\sum_{|l_2|>M}Q_{k_2+l_2}Q_{k_2}\r)\r\|_{L^2(X)\to L^2(X)}
\ls\dz^{2M\eta\thz}\dz^{|k_1-k_2|\eta(1-\thz)}
$$
and
$$
\lf\|\lf(\sum_{|l_1|>M}Q_{k_1+l_1}Q_{k_1}\r)\lf(\sum_{|l_2|>M}Q_{k_2+l_2}Q_{k_2}\r)^*\r\|_{L^2(X)\to L^2(X)}
\ls\dz^{2M\eta\thz}\dz^{|k_1-k_2|\eta(1-\thz)}.
$$
Combining this with Lemma \ref{lem:CSlem}, we obtain
$$
\lf\|\sum_{|k|\le M}\sum_{|l|>M}Q_{k+l}Q_k\r\|_{L^2(X)\to L^2(X)}\ls \dz^{M\eta\thz}.
$$
Therefore,  $\sum_{|k|\le M}\sum_{|l|>M}Q_{k+l}Q_kf\to 0$ in $L^2(X)$ when $M\to\fz$.

By Lemmas \ref{lem:ccrf2} and \ref{lem:CSlem}, we know that
$$
R_Nf=\sum_{k=-\fz}^\fz\sum_{|l|>N}Q_{k+l}Q_kf=\lim_{M\to\fz}\sum_{|k|\le M}\sum_{|l|>N}Q_{k+l}Q_kf
\quad\text{in $L^2(X)$},
$$
which implies that
$$
\lim_{M\to\fz}\lf\|\sum_{|k|>M}\sum_{|l|>N}Q_{k+l}Q_kf\r\|_{L^2(X)}=0.
$$
Thus, we obtain \eqref{eq:limL2}. This finishes the proof of (i).

Now we prove (ii). Let $f\in\CG(x_1,r,\bz,\gz)$.
To prove that $\{R_{N,M}f\}_{M=1}^\fz$ is a locally uniformly convergent sequence, fixing an arbitrary point $x\in X$,
we only need to find a positive sequence $\{c_{k,l}\}_{k,\ l=-\fz}^\fz$ such that
\begin{equation*}
\sup_{y\in B(x,r)}|Q_{k+l}Q_kf(y)|\le c_{k,l}\qquad \textup{and}\qquad \sum_{k=-\fz}^\fz\sum_{l=-\fz}^\fz c_{k,l}<\fz.
\end{equation*}
We claim that
\begin{align}\label{claim-add}
\sup_{y\in B(x,r)}|Q_{k+l}Q_kf(y)|\ls
\begin{cases}
\displaystyle \dz^{\eta |l|}\frac 1{V_{\dz^{[k\wedge (k+l)]}}(x)}
\exp\lf\{-c\lf[\frac{d(x,\CY^{[k\wedge (k+l)]})}{A_0\dz^{[k\wedge (k+l)]}}\r]^a\r\}
&\textup{if}\;\dz^{[k\wedge (k+l)]}\ge r,\\
\displaystyle\dz^{\eta |l|}\frac 1{V_r(x_1)}\lf(\frac{\dz^{[k\wedge (k+l)]}}r\r)^\bz
&\textup{if}\;\dz^{[k\wedge (k+l)]}<r.
\end{cases}
\end{align}
With $c_{k,l}$ defined as in the right-hand side of \eqref{claim-add}, we apply Lemma \ref{lem:expsum} to
obtain
$$
\sum_{k=-\fz}^\fz\sum_{l=-\fz}^\fz c_{k,l}\ls 1.
$$

To prove \eqref{claim-add}, due to the symmetry, we only consider the case $l\in\nn$.
When $\dz^k\ge r$, by \eqref{eq:mixsize}, Lemma \ref{lem-add}(ii) and
\eqref{eq-add8}, we conclude that, for any $y\in B(x,r)$,
\begin{align*}
|Q_{k+l}Q_kf(y)|&\le\int_X\lf|Q_{k+l}Q_k(y,z)\r||f(z)|\,d\mu(z)\\
&\ls\dz^{\eta l}\frac 1{V_{\dz^k}(y)}\exp\lf\{-c\lf[\frac{d(y,\CY^k)}{\dz^k}\r]^a\r\}
\int_X|f(x)|\,d\mu(z)\ls\dz^{\eta l}\frac 1{V_{\dz^k}(x)}\exp\lf\{-c\lf[\frac{d(x,\CY^k)}{A_0\dz^k}\r]^a\r\},
\end{align*}
as desired.
When $\dz^k<r$,  by \eqref{eq:mixcan}, we know that, for any $y\in B(x,r)$,
\begin{align*}
|Q_{k+l}Q_kf(y)|&=\lf|\int_X Q_{k+l}Q_k(y,z)[f(z)-f(y)]\,d\mu(y)\r|
\le\int_X|Q_{k+l}Q_k(y,z)||f(z)-f(y)|\,d\mu(z).
\end{align*}
For any $y,\ z\in X$,  we use the regularity of $f$ when $d(y,z)\le (2A_0)^{-1}[r+d(x_1,y)]$, or the size condition of
$f$ when $d(y,z)>(2A_0)^{-1}[r+d(x_1,y)]$ to conclude that
\begin{align*}
|f(z)-f(y)|\ls \lf[\frac{d(y,z)}{r+d(x_1,y)}\r]^\bz\frac 1{V_r(x_1)} \ls \lf[\frac{d(y,z)}{r}\r]^\bz\frac 1{V_r(x_1)}.
\end{align*}
From this, \eqref{eq:mixsize} and Lemma \ref{lem-add}(ii), it follows that
\begin{align*}
|Q_{k+l}Q_kf(y)|&\ls\dz^{\eta l}\frac 1{V_r(x_1)} \lf(\frac{\dz^k}r\r)^\bz
\int_X\frac 1{V_{\dz^k}(y)}
\exp\lf\{-c\lf[\frac{d(y,z)}{\dz^k}\r]^a\r\}\lf[\frac{d(y,z)}{\dz^k}\r]^\bz\,d\mu(z)\\
&\ls\dz^{\eta l}\frac 1{V_r(x_1)}\lf(\frac{\dz^k}r\r)^\bz.
\end{align*}
This finishes the proof of \eqref{claim-add}. Thus, for any $x\in X$, the sequence $\{R_{N,M}f(x)\}_{M=N+1}^\infty$ converges locally uniformly to some element, which is denoted by $\widetilde{R_N} f(x)$.

By \eqref{eq:limL2} and the Riesz theorem, we know that
there exists an increasing sequence $\{M_j\}_{j\in\nn}\subset\nn$ which tends to $\infty$ such that
\begin{equation*}
\lim_{j\to\fz} R_{N,M_j}f(x)=R_Nf(x)\quad\text{$\mu$-almost every $x\in X$}.
\end{equation*}
Consequently,
$\wz{R_N}f(x)=\lim_{M\to\fz} R_{N,M}f(x)=\lim_{j\to\fz} R_{N,M_j}f(x)=R_Nf(x)$ for $\mu$-almost every $x\in X$.
This finishes the proof of (ii).

For any $f\in \CG(x_1,r,\bz,\gz)$, since $\wz{R_N} f$ is well defined and
$\wz{R_N}f(x)=R_Nf(x)$ for $\mu$-almost every $x\in X$, it follows, from the boundedness of $R_N$ on $L^2(X)$
and the density of $\CG(x_1,r,\bz,\gz)$ in $L^2(X)$, that $\wz{R_N}$ can be uniquely extended to a boundedness
operator on $L^2(X)$. The extension operator, still denoted by $\wz{R_N}$,
satisfies that $R_Ng=\wz{R_N}g$ both in $L^2(X)$ and almost everywhere for all $g\in L^2(X)$. This finishes the
proof of (iii) and hence of Lemma \ref{lem:ccrf3}.
\end{proof}

Due to Lemma \ref{lem:ccrf3}, it is not necessary to distinguish $\wz R_N$ and $R_N$.
As a consequence of \eqref{eq:RNMG} and the
dominated convergence theorem, we easily deduce the following boundedness of $R_N$, the details being omitted.
\begin{proposition}\label{prop-add}
Let  $x_1\in X$, $r\in(0,\fz)$ and $\bz,\ \gz\in(0,\eta)$.
Then, for any $N\in\nn$ and $f\in\mathring{\CG}(x_1,r,\bz,\gz)$,
\begin{equation*}
\|R_Nf\|_{\mathring{\CG}(x_1,r,\bz,\gz)}\le C\dz^{(\eta-\eta')N}\|f\|_{\mathring{\CG}(x_1,r,\bz,\gz)},
\end{equation*}
where $\eta'\in(\max\{\bz,\gz\},\eta)$ and $C$ is a  positive constant independent of $x_1$, $r$, $N$ and $f$.
\end{proposition}


\subsection{Proofs of homogeneous continuous Calder\'{o}n reproducing formulae}\label{pr}

This section is devoted to the proofs of homogeneous Calder\'{o}n continuous reproducing formulae.
We start with following several lemmas.
\begin{lemma}\label{lem:propQkN}
Let $\{Q_k\}_{k\in\zz}$ be an $\exp$-{\rm ATI} and $Q_k^N:=\sum_{|l|\le N}Q_{k+l}$ for any $k\in\zz$ and $N\in\nn$.
Then there exist positive constants $C_{(N)}$ and $c_{(N)}$, depending on $N$, but being independent of $k$,
such that
\begin{enumerate}
\item for any $x,\ y\in X$,
\begin{equation}\label{eq:QkNsize}
\lf|Q_k^N(x,y)\r|\le C_{(N)}\frac 1{V_{\dz^k}(x)}\exp\lf\{-c_{(N)}\lf[\frac{d(x,y)}{\dz^k}\r]^a\r\};
\end{equation}
\item for any $x,\ x',\ y\in X$ with $d(x,x')\le\dz^k$,
\begin{align}\label{eq:QkNregx}
&\lf|Q_k^N(x,y)-Q_k^N(x',y)\r|+\lf|Q_k^N(y,x')-Q_k^N(y,x)\r|\\
&\quad\le C_{(N)}\lf[\frac{d(x,x')}{\dz^k}\r]^\eta
\frac 1{V_{\dz^k}(x)}\exp\lf\{-c_{(N)}\lf[\frac{d(x,y)}{\dz^k}\r]^a\r\};\noz
\end{align}
\item for any $x,\ x',\ y,\ y'\in X$ with $d(x,x')\le\dz^k$ and $d(y,y')\le\dz^k$,
\begin{align*}
&\lf|\lf[Q_k^N(x,y)-Q_k^N(x',y)\r]-\lf[Q_k^N(x,y')-Q_k^N(x',y')\r]\r|\\
&\quad\le C_{(N)}\lf[\frac{d(x,x')}{\dz^k}\r]^\eta\lf[\frac{d(y,y')}{\dz^k}\r]^\eta
\frac 1{V_{\dz^k}(x)}\exp\lf\{-c_{(N)}\lf[\frac{d(x,y)}{\dz^k}\r]^a\r\};\noz
\end{align*}
\item for any $x,\ y\in X$, $\int_X Q_k^N(x,y')\,d\mu(y')=0=\int_X Q_k^N(x',y)\,d\mu(x')$.
\end{enumerate}
\end{lemma}

\begin{proof}
From the cancellation of $Q_k$, it is easy to see that (iv) holds true. Noticing that the constants $C_{(N)}$
and $c_{(N)}$ are allowed to depend on $N$, we obtain (i) directly by the size condition of $Q_k$
and Remark \ref{rem:andef}(i).

To see (ii) and (iii), we make the following observation. Fix $N\in\nn$ and $\tau\in(0,\fz)$. Then, for any
$k\in\zz$ and $l\in\{-N,-N+1,\ldots,N-1,N\}$,
when $d(x,x')\le\tau\dz^{k+l}$ and $d(y,y')\le\tau\dz^{k+l}$, the regularity condition (resp., the second
difference regularity condition) of $Q_{k+l}$ in \eqref{eq:etiregx} [resp., \eqref{eq:etidreg}] remains true by
using the size conditon of $Q_{k+l}$ (resp., the regularity of $Q_{k+l}$),
with all constants involved depending on $\tau$ but independent of $k$, $l$, $x$, $y$, $x'$ and $y'$.

Using the above observation, we easily obtain (ii) and (iii), via taking
$\tau:=1$ when $l\in\{0,\ldots,N\}$, and $\tau:=\dz^{-N}$ when $l\in\{-N,-N+1,\ldots,-1\}$. This finishes the
proof of Lemma \ref{lem:propQkN}.
\end{proof}

\begin{lemma}\label{lem-add2} Let $\{Q_k\}_{k=-\fz}^\fz$ be an $\exp$-{\rm ATI}.
For any $k\in\zz$, let $E_k:=Q_k^NQ_k=\sum_{|l|\le N} Q_{k+l}Q_k$. Then there exist positive constants
$C_{(N)}$ and $c_{(N)}$, depending on $N$, but being independent of $k$, such that the integral kernel of
$E_k$, still denoted by $E_k$, satisfies the following:
\begin{enumerate}
\item for any $x,\ y\in X$,
\begin{equation*}
|E_k(x,y)|\le C_{(N)}\frac{1}{V_{\dz^k}(x)}\exp\lf\{-c_{(N)}\lf[\frac{d(x,y)}{\dz^{k}}\r]^a\r\}
\exp\lf\{-c_{(N)}\lf[\frac{d(x,\CY^k)}{\dz^{k}}\r]^a\r\};
\end{equation*}
\item for any $x,\ y,\ y'\in X$ with $d(y,y')\le\dz^k$ or $d(y,y')\le (2A_0)^{-1}[\dz^k+d(x, y)]$,
\begin{align*}
&|E_k(x,y)-E_k(x,y')|\\
&\quad\le C_{(N)}\lf[\frac{d(y,y')}{\dz^k}\r]^\eta
\frac{1}{V_{\dz^k}(x)}\exp\lf\{-c_{(N)}\lf[\frac{d(x,y)}{\dz^{k}}\r]^a\r\}
\exp\lf\{-c_{(N)}\lf[\frac{d(x,\CY^k)}{\dz^{k}}\r]^a\r\}.
\end{align*}
\end{enumerate}
\end{lemma}

\begin{proof}
Applying Lemma \ref{lem:propQkN} and Remark \ref{rem:andef}(i)
and following the proofs of (i) and (ii) of Lemma \ref{lem:ccrf1}, we
directly obtain (i) and (ii) for $d(y,y')\le\dz^k$.
Further,  applying (i) and proceeding as in the proof of Proposition \ref{prop:etoa} [see Remark
\ref{rem:andef}(iii)], we find that (ii) remains true when  $d(y,y')\le(2A_0)^{-1}[\dz^k+d(x,y)]$. This
finishes the proof of Lemma \ref{lem-add2}.
\end{proof}

\begin{lemma}\label{lem-add3}
Let $\{Q_k\}_{k=-\fz}^\fz$ be an $\exp$-{\rm ATI}.
Then, for any $k\in\zz$ and $f\in\CG(\eta,\eta)$, $Q_kf\in\CG(\eta,\eta)$.
\end{lemma}

\begin{proof}
Notice that  $\CG(x_0,\dz^k,\eta,\eta)$ and $\CG(\eta,\eta)$ coincide in the sense of equivalent norms, with
the equivalent positive constants depending on $k$, but this is harmless for the proof of this lemma. Thus,
without loss of generality, we may as well assume that $\|f\|_{ \CG(x_0,\dz^k,\eta,\eta)}= 1$ and, to prove
this lemma, it suffices to show $Q_k f\in \CG(x_0,\dz^k,\eta,\eta)$.

For any $x\in X$, by the size conditions of $Q_k$ and $f$, we write
\begin{align*}
|Q_kf(x)|&=\lf|\int_X Q_k(x,y)f(y)\,d\mu(y)\r|\\
&\ls\int_X \frac{1}{V_{\dz^k}(x)}\exp\lf\{-\nu'\lf[\frac{d(x,y)}{\dz^k}\r]^a\r\} \frac 1{V_{\dz^k}(x_0)+V(x_0,y)}\lf[\frac{\dz^k}{\dz^k+d(x_0,y)}\r]^\eta\,d\mu(y).
\end{align*}
Observe that, for any $y\in X$, the quasi-triangle inequality of $d$ implies that either $d(x, y)\ge d(x_0, x)/(2A_0)$
or $d(x_0, y)\ge d(x_0, x)/(2A_0)$. Also, notice that
\begin{align}\label{eq-xxx}
\frac{1}{V_{\dz^k}(x)}\ls \frac{1}{\mu(B(x, \dz^k+d(x,y)))}\lf[\frac{\dz^k+d(x,y)}{\dz^k}\r]^\omega.
\end{align}
From these and Lemma \ref{lem-add}(ii), it follows that, for any $x\in X$,
\begin{align}\label{eq-x4}
|Q_kf(x)|\ls \frac1{V_{\dz^k}(x_0)+V(x_0,x)}
\lf[\frac{\dz^k}{\dz^k+d(x_0,x)}\r]^\eta.
\end{align}

Now we consider the regularity of $Q_k f$.
For any $x,\ x'\in X$ satisfying $d(x,x')\le(2A_0)^{-2}[\dz^k+d(x_0,x)]$,
by the fact that $\int_X [Q_k(x,y)-Q_k(x',y)]\,d\mu(y)=0$, we write
\begin{align*}
|Q_kf(x)-Q_kf(x')|&=\lf|\int_X [Q_k(x,y)-Q_k(x',y)][f(y)-f(x)]\,d\mu(y)\r|\\
&\le\int_{d(x,y)\le(2A_0)^{-1}[\dz^k+d(x_0,x)]} |Q_k(x,y)-Q_k(x',y)||f(y)-f(x)|\,d\mu(y)\\
&\quad+\int_{d(x,y)>(2A_0)^{-1}[\dz^k+d(x_0,x)]}|Q_k(x,y)-Q_k(x',y)||f(y)|\,d\mu(y)\\
&\quad+|f(x)|\int_{d(x,y)>(2A_0)^{-1}[\dz^k+d(x_0,x)]}|Q_k(x,y)-Q_k(x',y)|\,d\mu(y)
=:\RZ_1+\RZ_2+\RZ_3.
\end{align*}
We first deal with $\RZ_1$. By the size condition of $Q_k$ and  Remark \ref{rem:andef}(i), we conclude that
$$
|Q_k(x,y)-Q_k(x',y)|\ls\frac{1}{V_{\dz^k}(y)}\lf(\exp\lf\{-\nu'\lf[\frac{d(x,y)}{\dz^k}\r]^a\r\}
+\exp\lf\{-\nu'\lf[\frac{d(x',y)}{\dz^k}\r]^a\r\}\r).
$$
If, in addition, $d(x,x')\le(2A_0)^{-1}[\dz^k+d(x,y)]$, then, by Remark \ref{rem:andef}(iii), the right-hand
side of the above formula can be multiplied by another
term $[\frac{d(x,x')}{\dz^k+d(x,y)}]^\eta$ by the regularity of $Q_k$. By this, the regularity
of $f$ and Lemma \ref{lem-add}(ii), we have
\begin{align*}
\RZ_1&\ls\frac 1{V_{\dz^k}(x_0)+V(x_0,x)}\lf[\frac{\dz^k}{\dz^k+d(x_0,x)}\r]^\eta
\int_{d(x,y)\le(2A_0)^{-1}[\dz^k+d(x_0,x)]}\lf[\frac{d(x,y)}{\dz^k+d(x_0,x)}\r]^\eta\\
&\quad\times\min\lf\{1,\lf[\frac{d(x,x')}{\dz^k+d(x,y)}\r]^\eta\r\}\frac 1{V_{\dz^k}(y)}
\lf(\exp\lf\{-\nu'\lf[\frac{d(x,y)}{\dz^k}\r]^a\r\}
+\exp\lf\{-\nu'\lf[\frac{d(x',y)}{\dz^k}\r]^a\r\}\r)\,d\mu(y)\\
&\ls\lf[\frac{d(x,x')}{\dz^k+d(x_0,x)}\r]^\eta\frac1{V_{\dz^k}(x_0)+V(x_0,x)}
\lf[\frac{\dz^k}{\dz^k+d(x_0,x)}\r]^\eta.
\end{align*}
Notice that, when $d(x,y)>(2A_0)^{-1}[\dz^k+d(x_0,x)]$, we have
$$d(x,x')\le(2A_0)^{-2}[\dz^k+d(x_0,x)]
<(2A_0)^{-1}d(x,y)\le(2A_0)^{-1}[\dz^k+d(x,y)].$$ Thus, from the regularity of $Q_k$, Remark \ref{rem:andef}(i)
and Lemma \ref{lem-add}(ii), we deduce that
\begin{align*}
\RZ_2&\ls\int_{d(x,y)>(2A_0)^{-1}[\dz^k+d(x_0,x)]}\lf[\frac{d(x,x')}{\dz^k+d(x,y)}\r]^\eta
\frac{1}{V_{\dz^k}(x)+V(x,y)}\exp\lf\{-\nu'\lf[\frac{d(x,y)}{\dz^k}\r]^a\r\}|f(y)|\,d\mu(y)\\
&\ls\lf[\frac{d(x,x')}{\dz^k+d(x_0,x)}\r]^\eta
\frac{1}{V_{\dz^k}(x)+V(x_0,x)}\lf[\frac{\dz^k}{\dz^k+d(x_0,x)}\r]^\eta.
\end{align*}
Similarly, by Remark \ref{rem:andef}(i) and Lemma \ref{lem-add}(ii), we conclude that
\begin{align*}
\RZ_3&\ls\frac 1{V_{\dz^k}(x_0)+V(x_0,x)}\lf[\frac{\dz^k}{\dz^k+d(x_0,x)}\r]^\eta\\
&\quad\times\int_{d(x,y)>(2A_0)^{-1}[\dz^k+d(x_0,x)]}\lf[\frac{d(x,x')}{\dz^k+d(x,y)}\r]^\eta
\frac{1}{V_{\dz^k}(x)+V(x,y)}\exp\lf\{-\nu'\lf[\frac{d(x,y)}{\dz^k}\r]^a\r\}\,d\mu(y)\\
&\ls\lf[\frac{d(x,x')}{\dz^k+d(x_0,x)}\r]^\eta
\frac 1{V_{\dz^k}(x_0)+V(x_0,x)}\lf[\frac{\dz^k}{\dz^k+d(x_0,x)}\r]^\eta.
\end{align*}
Combining this with $\RZ_1$ through $\RZ_3$, we find that, when $d(x,x')\le(2A_0)^{-2}[\dz^k+d(x_0,x)]$,
\begin{align}\label{eq-x5}
|Q_kf(x)-Q_kf(x')|\ls\lf[\frac{d(x,x')}{\dz^k+d(x_0,x)}\r]^\eta
\frac 1{V_{\dz^k}(x_0)+V(x_0,x)}\lf[\frac{\dz^k}{\dz^k+d(x_0,x)}\r]^\eta,
\end{align}

When $(2A_0)^{-2}[\dz^k+d(x_0,x)]<d(x,x')\le(2A_0)^{-1}[\dz^k+d(x_0,x)]$, we have
$$
\dz^k+d(x_0, x)\sim \dz^k+d(x_0, x').
$$
From this and \eqref{eq-x4}, we deduce that \eqref{eq-x5} also holds true when
$d(x,x')\le(2A_0)^{-1}[\dz^k+d(x_0,x)]$, which completes the proof of Lemma \ref{lem-add3}.
\end{proof}

\begin{theorem}\label{thm:hcrf}
Suppose that $\bz,\ \gz\in(0,\eta)$ and $\{Q_k\}_{k\in\zz}$ is an $\exp$-{\rm ATI}.
Then there exists a sequence $\{\wz{Q}_k\}_{k\in\zz}$ of bounded linear operators on $L^2(X)$ such that, for
any $f$ in $\GOO{\bz,\gz}$ [resp., $L^p(X)$ with any given $p\in(1,\fz)$],
\begin{equation}\label{eq:hcrf}
f=\sum_{k=-\fz}^\fz \wz{Q}_kQ_kf,
\end{equation}
where the series converges in $\mathring{\CG}^\eta_0(\bz,\gz)$ [resp., $L^p(X)$ with any given $p\in(1,\fz)$].
Moreover, for any $k\in\zz$, the kernel of $\wz{Q}_k$ satisfies
the size condition \eqref{eq:atisize}, the regularity condition \eqref{eq:atisregx}
only for the first variable, and also the following cancellation condition:  for any $x,\ y\in X$,
\begin{equation}\label{eq:wzQcan}
\int_X \wz{Q}_k(x',y)\,d\mu(x')=0=\int_X\wz{Q}_k(x,y')\,d\mu(y').
\end{equation}
\end{theorem}

\begin{proof}
Fix $\eta'\in(\max\{\bz,\gz\},\eta)$. According to \eqref{eq:RNL2} and Proposition \ref{prop-add}, we have
$$
\|R_N\|_{L^2(X)\to L^2(X)}\le \wz C\dz^{\eta'N}\qquad
\textup{and}\qquad
\|R_N\|_{\mathring{\CG}(x_1,r,\bz,\gz)\to\mathring{\CG}(x_1,r,\bz,\gz)}\le \wz C\dz^{(\eta-\eta')N},
$$
where $\wz C$ is a positive constant independent of $x_1$, $r$ and $N$.
Choose $N\in\nn$ sufficiently large such that $\max\{\wz C\dz^{\eta'N},\wz C\dz^{(\eta-\eta')N}\}\le 1/2$.
Then $T_N^{-1}=(I-R_N)^{-1}$ exists as a bounded operator on $L^2(X)$ and also on $\GO{x_1,r,\bz,\gz}$,
with operator norms at most $2$.
For any $k\in\zz$ and $x,\ y\in X$, we define
$$
\wz{Q}_k(x,y):=T_N^{-1}\lf(Q_k^N(\cdot,y)\r)(x) \quad \textup{with} \quad Q_k^N=\sum_{|l|\le N}Q_{k+l}.
$$
Notice that, for any $k\in\zz$ and $y\in X$,
based on Proposition \ref{prop:etoa}, every $Q_k(x,y)$ viewed as a function of $x$ belongs to the
space of test functions, ${\mathring{\CG}(y,\dz^k,\bz,\gz)}$, with
$\|\cdot\|_{{\mathring{\CG}(y,\dz^k,\bz,\gz)}}$-norm independent of $y\in X$ and $k\in\zz$.
This implies that $\wz{Q}_k$ satisfies \eqref{eq:atisize} and \eqref{eq:atisregx} for the first variable.

Next, we prove \eqref{eq:wzQcan}. For any $y\in X$, by $\wz{Q}_k(\cdot,y)\in \mathring{\CG}(y,\dz^k,\bz,\gz)$,
we have
$\int_X \wz{Q}_k(x,y)\,d\mu(x)=0$. Now we show the second equality in \eqref{eq:wzQcan}.
To this end, for any $x, \ y\in X$, we write
$$
\wz{Q}_k(x,y)=\sum_{j=0}^\fz(R_N)^j\lf(Q_k^N(\cdot,y)\r)(x),
$$
which converges in $\GO{y,\dz^k,\bz,\gz}$  as a function of variable $x$. %
By the dominated convergence theorem, we can obtain the second equality of \eqref{eq:wzQcan}, provided  we can show that, for any $j\in\zz_+$ and $x\in X$,
\begin{equation}\label{eq:RNjcan}
\int_X(R_N)^j\lf(Q_k^N(\cdot,y)\r)(x)\,d\mu(y)=0.
\end{equation}
We show \eqref{eq:RNjcan} via a method of induction.
Indeed, when $j=0$, \eqref{eq:RNjcan} follows directly from the cancellation of $Q_k^N$.
Assuming that \eqref{eq:RNjcan} holds true for some $j\in\zz_+$, we prove it for $j+1$.
To this end, for any $M\in\nn$, let $R_{N,M}$ be defined as in
\eqref{eq:defRNM}. By \eqref{eq:RNMG}, we conclude that
$$
\lf\|R_{N,M}R_N^j\lf(Q_k^N(\cdot,y)\r)\r\|_{\CG(y,\dz^k,\bz,\gz)}
\ls 2^{-j}\lf\|Q_k^N(\cdot,y)\r\|_{\CG(y,\dz^k,\bz,\gz)}\ls 2^{-j}.
$$
Combining this with the dominated convergence theorem and the Fubini theorem, together with Lemma \ref{lem:ccrf3}, we conclude that, for any $x\in X$,
\begin{align*}
\int_X (R_N)^{j+1}\lf(Q_k^N(\cdot,y)\r)(x)\,d\mu(y)
&=\lim_{M\to\fz} \int_X R_{N,M}(R_N)^{j}\lf(Q_k^N(\cdot,y)\r)(x)\,d\mu(y)\\
&= \lim_{M\to\fz} \int_X\int_X  R_{N,M}(x, z)(R_N)^{j}\lf(Q_k^N(\cdot,y)\r)(z)\,d\mu(x)\,d\mu(z)
=0,
\end{align*}
where, in the last step, we used the indiction hypothesis. This proves \eqref{eq:RNjcan} and hence finishes the
proof of \eqref{eq:wzQcan}.

For any $k\in\zz$ and $x,\ y\in X$, since
$T_N\wz{Q}_k(x,y)=T_N(\wz{Q}_k(\cdot,y))(x)=Q_k^N(x,y)$, it follows that
$T_N^{-1}Q_k^N=\wz{Q}_k$. Invoking the expression of $T_N$ in \eqref{eq:defRT}, we have
$$
f=T_N^{-1}T_Nf=T_N^{-1}\lf(\sum_{k=-\fz}^\fz Q_k^NQ_k\r)(f)=\sum_{k=-\fz}^\fz T_N^{-1}\lf(Q_k^N\r)Q_kf
=\sum_{k=-\fz}^\fz \wz{Q}_kQ_kf\qquad \textup{in}\;\; L^2(X).
$$
Further, for any $L\in\nn$ and $f\in L^2(X)$,
\begin{align}\label{eq:finsum}
\sum_{|k|\le L}\wz{Q}_kQ_kf
&=T_N^{-1}\lf(\sum_{|k|\le L}Q_k^NQ_k\r)f\\
&=T_N^{-1}\lf(T_N-\sum_{|k|\ge L+1}Q_k^NQ_k\r)f=f-T_N^{-1}\lf(\sum_{|k|\ge L+1}Q_k^NQ_kf\r)\qquad \textup{in}\;\; L^2(X).\noz
\end{align}
The remaining arguments are divided into the following three steps.

{\it Step 1) Proof of the convergence of \eqref{eq:hcrf} in $\mathring{\CG}(\bz,\gz)$
when $f\in \mathring{\CG}(\bz',\gz')$ with $\bz'\in(\bz,\eta)$ and $\gz'\in(\gz,\eta)$.}

Without loss of generality, we may assume that  $f\in \mathring{\CG}(\bz',\gz')$ with
$\|f\|_{\CG(\bz',\gz')}=1$. Due to \eqref{eq:finsum}  and the boundedness of $T_N^{-1}$ on
$\mathring{\CG}(\bz,\gz)$, we have
\begin{align*}
\lim_{L\to\fz}\lf\|f-\sum_{|k|\le L}\wz{Q}_kQ_kf\r\|_{\mathring{\CG}(\bz,\gz)}
&=\lim_{L\to\fz}\lf\|T_N^{-1}\lf(\sum_{|k|\ge L+1}Q_k^NQ_kf\r)\r\|_{\mathring{\CG}(\bz,\gz)}\ls \lim_{L\to\fz}\lf\|\sum_{|k|\ge L+1}Q_k^NQ_kf\r\|_{\mathring{\CG}(\bz,\gz)}.
\end{align*}
Assume for the moment that there exists $\sigma\in(0,\fz)$, independent of $k$, $L$ and $f$, such that
\begin{equation}\label{eq:sumsize}
\lf|Q_k^NQ_kf(x)\r|\ls \dz^{\sigma |k|}\frac 1{V_1(x_0)+V(x_0,x)}\lf[\frac 1{1+d(x_0,x)}\r]^\gz,
\qquad \forall\,x\in X
\end{equation}
and that, for any $x,\ x'\in X$ satisfying  $d(x,x')\le (2A_0)^{-1}[1+d(x_0,x)]$,
\begin{equation}\label{eq:sumreg}
\lf|Q_k^NQ_kf(x)-Q_k^NQ_kf(x')\r|
\ls\dz^{\sigma |k|}\lf[\frac{d(x,x')}{1+d(x_0,x)}\r]^\bz\frac 1{V_1(x_0)+V(x_0,x)}\lf[\frac 1{1+d(x_0,x)}\r]^\gz.
\end{equation}
Indeed, once we have proved  \eqref{eq:sumsize} and \eqref{eq:sumreg}, then
\begin{equation*}
\lim_{L\to\fz}\lf\|\sum_{|k|\ge L+1}Q_k^NQ_kf\r\|_{{\CG}(\bz,\gz)}\le\lim_{L\to\fz}
\sum_{|k|\ge L+1}\lf\|Q_k^NQ_kf\r\|_{{\CG}(\bz,\gz)}\ls \lim_{L\to\fz}\sum_{|k|\ge L+1} \dz^{\sigma|k|}=0.
\end{equation*}
Moreover, by \eqref{eq:sumsize}, the Fubini theorem and the cancelation property of $Q_k$, we obtain
$$
\int_X\sum_{|k|\ge L+1}Q_k^NQ_kf(x)\,d\mu(x)=0.
$$
Therefore, we have $\lim_{L\to\fz}\lf\|\sum_{|k|\ge L+1}Q_k^NQ_kf\r\|_{\mathring{\CG}(\bz,\gz)}=0$ and hence
$f=\sum_{k\in\zz}\wz{Q}_kQ_kf$ in $\mathring{\CG}(\bz,\gz)$, which is the desired conclusion.

Once we have proved \eqref{eq:sumsize}, then we can use it to show \eqref{eq:sumreg} in the following way.
Indeed, when $d(x,x')\le(2A_0)^{-1}[1+d(x_0,x)]$, we have $(2A_0)^{-1} d(x_0, x')\le d(x_0, x)\le 2A_0 d(x_0, x')$ and $1+d(x_0,x)\sim 1+d(x_0,x')$
so that \eqref{eq:sumsize} implies that
\begin{align}\label{eq-x3}
&\lf|Q_k^NQ_kf(x)-Q_k^NQ_kf(x')\r|\\
&\quad\ls\dz^{\sigma|k|}\lf\{\frac 1{V_1(x_0)+V(x_0,x)}\lf[\frac 1{1+d(x_0,x)}\r]^{\gz}
+\frac 1{V_1(x_0)+V(x_0,x')}\lf[\frac 1{1+d(x_0,x')}\r]^{\gz}\r\}\noz\\
&\quad\sim\dz^{\sigma|k|}\frac 1{V_1(x_0)+V(x_0,x)}\lf[\frac 1{1+d(x_0,x)}\r]^{\gz}.\noz
\end{align}
Notice that \eqref{eq:sumsize} and Proposition  \ref{prop:basic}(iii) imply that  $Q_k^NQ_k$ is a bounded
operator on $L^2(X)$, with operator norm independent of $k$. Moreover, since $Q_k^N Q_k=\sum_{|l|\le N}
Q_{k+l}Q_k$, by Lemma \ref{lem:ccrf1} and Corollary \ref{cor:mixb}, we find that every $Q_{k+l}Q_k$
with $l\in\{-N,-N+1,\ldots,N-1,N\}$ satisfies \eqref{eq:Ksize}, \eqref{eq:Kreg} and
\eqref{eq:Kdreg} with $C_T$ therein being a positive constant independent of $k$ (but $C_T$ may depend on
$N$), so does $Q_k^NQ_k$. Thus, applying Theorem \ref{thm:Kbdd}, we know that $Q_k^N Q_k$ is a bounded operator
on $\mathring\CG(\bz',\gz')$. In particular, we have $Q_k^N Q_kf\in \mathring\CG(\bz',\gz')$, which implies
that, when $d(x,x')\le (2A_0)^{-1}[1+d(x_0,x)]$,
\begin{align*}
\lf|Q_k^NQ_kf(x)-Q_k^NQ_kf(x')\r|
\ls\lf[\frac{d(x,x')}{1+d(x_0,x)}\r]^{\bz'}\frac 1{V_1(x_0)+V(x_0,x)}\lf[\frac 1{1+d(x_0,x)}\r]^{\gz}.
\end{align*}
Taking the geometric means between the last inequality as above and \eqref{eq-x3}, we obtain \eqref{eq:sumreg}.

It remains to show \eqref{eq:sumsize}. Let us first prove \eqref{eq:sumsize} for the case $k\in\zz_+$. For any
$x\in X$, we write
\begin{align*}
\lf|Q_{k}^NQ_kf(x)\r|&=\sum_{k=L+1}^\fz\lf|\int_X E_k(x,y)[f(y)-f(x)]\,d\mu(y)\r|\\
&\le\int_{d(x,y)\le(2A_0)^{-1}[1+d(x_0,x)]}|E_k(x,y)||f(y)-f(x)|\,d\mu(y)\\
&\quad+\int_{d(x,y)>(2A_0)^{-1}[1+d(x_0,x)]}|E_k(x,y)||f(y)|\,d\mu(y)\\
&\quad+|f(x)|\int_{d(x,y)>(2A_0)^{-1}[1+d(x_0,x)]}|E_k(x,y)|\,d\mu(y)
=:\RZ_{1,1}+\RZ_{1,2}+\RZ_{1,3}.
\end{align*}
From Lemma \ref{lem-add2}(i), the regularity condition of $f$, Lemma \ref{lem-add}(ii)  and $\gz'>\gz$, we
deduce that
\begin{align*}
\RZ_{1,1}&\ls\frac 1{V_1(x_0)+V(x_0,x)}\lf[\frac 1{1+d(x_0,x)}\r]^{\gz'}\int_X
\frac 1{V_{\dz^k}(x)}\exp\lf\{-c\lf[\frac{d(x,y)}{\dz^{k}}\r]^a\r\}
\lf[\frac{d(x,y)}{1+d(x_0,y)}\r]^{\bz'}\,d\mu(y)\\
&\ls\frac 1{V_1(x_0)+V(x_0,x)}\lf[\frac 1{1+d(x_0,x)}\r]^{\gz'}\int_X\frac {\dz^{k\bz'}}{V_{\dz^k}(x)}\exp\lf\{-c\lf[\frac{d(x,y)}{\dz^{k}}\r]^a\r\}
\lf[\frac{d(x,y)}{\dz^k}\r]^{\bz'}\,d\mu(y)\\
&\ls\dz^{|k|\bz'}\frac 1{V_1(x_0)+V(x_0,x)}\lf[\frac 1{1+d(x_0,x)}\r]^{\gz}.
\end{align*}
For the term $\RZ_{1,2}$, applying Lemma \ref{lem-add2}(i) and the regularity condition of $f$, we conclude that
\begin{align*}
\RZ_{1,2}&\ls\int_{d(x,y)>(2A_0)^{-1}[1+d(x_0,x)]}\frac 1{V_{\dz^k}(x)}
\exp\lf\{-c\lf[\frac{d(x,y)}{\dz^{k}}\r]^a\r\}\frac 1{V_1(x_0)}\,d\mu(y).
\end{align*}
Observe that the doubling condition \eqref{eq:doub} implies that $V_1(x_0)+V(x_0,x)\ls [1+d(x_0, x)]^\omega V_1(x_0)$.
Meanwhile, if $d(x,y)>(2A_0)^{-1}[1+d(x_0,x)]$ and $k\in\zz_+$, then
$$\exp\lf\{-\frac c2\lf[\frac{d(x,y)}{\dz^{k}}\r]^a\r\}\le \exp\lf\{-\frac c 2\lf[\frac{(2A_0)^{-1}[1+d(x_0,x)]}{\dz^{k}}\r]^a\r\}
\ls \lf[\frac{\dz^{k}}{1+d(x_0, x)}\r]^{\omega+\gz+\bz'}.$$
By these and Lemma \ref{lem-add}(ii), we further obtain
\begin{align*}
\RZ_{1,2}&\ls\dz^{|k|(\omega+\gz+\bz')}\frac 1{V_1(x_0)+V(x_0,x)}\lf[\frac 1{1+d(x_0,x)}\r]^{\gz}.
\end{align*}
Now we estimate $\RZ_{1,3}$. Again, using the fact that the conditions $d(x,y)>(2A_0)^{-1}[1+d(x_0,x)]$ and
$k\in\zz_+$, we obtain $\exp\{-\frac c2[\frac{d(x,y)}{\dz^{k}}]^a\}\ls \dz^{k\bz'}$.
From this, Lemma \ref{lem-add2}(i), the regularity condition of $f$ and Lemma \ref{lem-add}(ii), it follows that
\begin{align*}
\RZ_{1,3}&\ls|f(x)|
\int_{d(x,y)>(2A_0)^{-1}[1+d(x_0,x)]}\frac 1{V_{\dz^k}(x)}\exp\lf\{-c\lf[\frac{d(x,y)}{\dz^{k}}\r]^a\r\}\,d\mu(y)\\
&\ls\dz^{k\bz'}\frac 1{V_1(x_0)+V(x_0,x)}\lf[\frac 1{1+d(x_0,x)}\r]^{\gz}.
\end{align*}
Combining the estimates of $\RZ_{1,1}$ through $\RZ_{1,3}$, we obtain \eqref{eq:sumsize} when $k\in\zz_+$.

Next we  prove \eqref{eq:sumsize} for the case $k\in\zz\setminus\zz_+$. Notice that, for any $x\in X$,
\begin{align*}
\lf|Q_{k}^NQ_kf(x)\r|
&=\lf|\int_X [E_k(x,y)-E_k(x,x_0)]|f(y)|\,d\mu(y)\r|\\
&\le\int_{d(x_0,y)\le(2A_0)^{-1}[\dz^k+d(x_0, x)]}|E_k(x,y)-E_k(x,x_0)||f(y)|\,d\mu(y)\\
&\quad+\int_{d(x_0,y)>(2A_0)^{-1}[\dz^k+d(x_0, x)]}|E_k(x,y)||f(y)|\,d\mu(y)\\
&\quad+|E_k(x,x_0)|\int_{d(x_0,y)>(2A_0)^{-1}[\dz^k+d(x_0, x)]}|f(y)|\,d\mu(y)
=:\RZ_{1,4}+\RZ_{1,5}+\RZ_{1,6}.
\end{align*}
To estimate $\RZ_{1,4}$, we choose $\wz\gz\in(\gz,\gz')$. By  Lemma \ref{lem-add2}(ii), $\wz\gz<\eta$ and the
size condition of $f$, we have
\begin{align*}
\RZ_{1,4}&\ls\int_X
\lf[\frac{d(x_0,y)}{\dz^k}\r]^{\wz\gz}\frac 1{V_{\dz^k}(x_0)}\exp\lf\{-c\lf[\frac{d(x_0,x)}{\dz^k}\r]^a\r\}
\frac 1{V_1(x_0)+V(x_0,y)}\lf[\frac 1{1+d(x_0,y)}\r]^{\gz'}\,d\mu(y).
\end{align*}
Since $\dz^k\ge 1$, it follows that
\begin{align}\label{eq-x2}
&\frac 1{V_{\dz^k}(x_0)}\exp\lf\{-c\lf[\frac{d(x_0,x)}{\dz^k}\r]^a\r\}\\
&\quad\ls \frac 1{\mu(B(x_0, \dz^k+d(x_0, x)))} \lf[\frac{\dz^k+d(x_0, x)}{\dz^k}\r]^\omega \lf[1+\frac{d(x_0,x)}{\dz^k}\r]^{-\omega-\gz}\noz\\
&\quad\ls \frac 1{\mu(B(x_0, 1+d(x_0, x)))} \lf[\frac{\dz^k}{\dz^k+d(x_0,x)}\r]^{\gz}
\ls \dz^{k\gz}\frac 1{V_1(x_0)+V(x_0,x)}\lf[\frac 1{1+d(x_0,x)}\r]^{\gz}.\noz
\end{align}
Also, notice that
$[\frac{d(x_0,y)}{\dz^k}]^{\wz\gz} [\frac 1{1+d(x_0,y)}]^{\gz'} \le \dz^{-k\wz\gz}[\frac 1{1+d(x_0,y)}]^{\gz'-\wz\gz}.$
Combining these with Lemma \ref{lem-add}(ii), we find that
\begin{align*}
\RZ_{1,4}&\ls\dz^{k(\gz-\wz\gz)}\frac 1{V_1(x_0)+V(x_0,x)}\lf[\frac 1{1+d(x_0,x)}\r]^{\gz}\\
&\quad\times\int_{d(x_0,y)\le\dz^k}
\frac 1{V_1(x_0)+V(x_0,y)}\lf[\frac 1{1+d(x_0,y)}\r]^{\gz'-\wz\gz}\,d\mu(y)\\
&\ls\dz^{(\wz\gz-\gz)|k|}\frac 1{V_1(x_0)+V(x_0,x)}\lf[\frac 1{1+d(x_0,x)}\r]^{\gz}.
\end{align*}
For the term $\RZ_{1,5}$, from Lemma \ref{lem-add2}(i), the size condition of $f$ and Lemma \ref{lem-add}(iii),
we deduce that
\begin{align*}
\RZ_{1,5}&\ls\int_{d(x_0,y)>(2A_0)^{-1}[\dz^k+d(x_0, x)]}\lf[\frac{d(x_0,y)}{\dz^k}\r]^{\gz'-\gz}\frac 1{V_{\dz^k}(x)}
\exp\lf\{-c\lf[\frac{d(x,y)}{\dz^{k}}\r]^a\r\}\\
&\quad\times\frac 1{V_1(x_0)+V(x_0,y)}\lf[\frac 1{1+d(x_0,y)}\r]^{\gz'}\,d\mu(y)\\
&\ls\frac 1{V_1(x_0)+V(x_0,x)}\lf[\frac 1{1+d(x_0,x)}\r]^{\gz}\dz^{(\gz-\gz')k}\int_X\frac 1{V_{\dz^k}(x)}
\exp\lf\{-c\lf[\frac{d(x,y)}{\dz^{k}}\r]^a\r\}\,d\mu(y)\\
&\ls \dz^{(\gz'-\gz)|k|}\frac 1{V_1(x_0)+V(x_0,x)}\lf[\frac 1{1+d(x_0,x)}\r]^{\gz}.
\end{align*}
To estimate $\RZ_{1,6}$, we again choose $\wz\gz\in(\gz,\gz')$.
Applying  Lemma \ref{lem-add2}(i), \eqref{eq-x2}, the size condition of $f$ and Lemma \ref{lem-add}(ii), we
proceed as in the estimate of $\RZ_{1,4}$ to derive that
\begin{align*}
\RZ_{1,6}&\ls\frac 1{V_{\dz^k}(x_0)}\exp\lf\{-c\lf[\frac{d(x,x_0)}{\dz^{k}}\r]^a\r\}
\int_{d(x_0,y)>(2A_0)^{-1}[\dz^k+d(x_0,x)]}\lf[\frac{d(x_0,y)}{\dz^k}\r]^{\wz\gz}\\
&\qquad\times\frac 1{V_1(x_0)+V(x_0,y)}\lf[\frac{1}{1+d(x_0,y)}\r]^{\gz'}\,d\mu(y)\\
&\ls\dz^{(\wz\gz-\gz)|k|}\frac 1{V_1(x_0)+V(x_0,x)}\lf[\frac 1{1+d(x_0,x)}\r]^{\gz}.
\end{align*}
Combining the estimates of $\RZ_{1,4}$ through $\RZ_{1,6}$, we obtain \eqref{eq:sumsize} when
$k\in\zz\setminus\zz_+$. This finishes the proof of \eqref{eq:sumsize} and hence of Step 1).

{\it Step 2) Proof of the convergence of \eqref{eq:hcrf} in $\mathring{\CG}^\eta_0(\bz,\gz)$ when $f\in \mathring{\CG}^\eta_0(\bz,\gz)$.}

If $f\in \mathring{\CG}^\eta_0(\bz,\gz)$, then  there exists a sequence
$\{g_n\}_{m=1}^\infty\subset\GO{\eta,\eta}$ such that $\lim_{n\to\infty}\|f-g_n\|_{\GO{\bz,\gz}}=0$.
By the already proved result in Step 1), we know that every $g_n$ satisfies
$$
\lim_{L\to\infty}\lf\|g_n-\sum_{|k|\le L}\wz{Q}_kQ_kg_n\r\|_{\mathring{\CG}(\bz,\gz)}=0.
$$
For any $N,\ L\in\nn$, define
\begin{align}\label{TNL}
\wz T_{N,L}:=\sum_{|k|\le L}{Q}_k^NQ_k=\sum_{|k|\le L}\sum_{|l|\le N}Q_{k+l}Q_k.
\end{align}
Repeating the proof of Lemma \ref{lem:ccrf2}, we find that, for any fixed $\thz\in(0,1)$ and any
$k_1,\ k_2\in\zz$,
\begin{align*}
&\lf\|\lf(\sum_{|l_1|\le N}Q_{k_1+l_1}Q_{k_1}\r)\lf(\sum_{|l_2|\le N}Q_{k_2+l_2}Q_{k_2}\r)^*\r\|
_{L^2(X)\to L^2(X)}\\
&\quad\le\sum_{|l_1|\le N}\sum_{|l_2|\le N}
\lf\|Q_{k_1+l_1}Q_{k_1}\lf(Q_{k_2+l_2}Q_{k_2}\r)^*\r\|_{L^2(X)\to L^2(X)}\noz\\
&\quad\ls\sum_{l_1=-\fz}^\fz\sum_{l_2=-\fz}^\fz\dz^{(|l_1|+|l_2|)\eta\thz}\dz^{|k_1-k_2|\eta(1-\thz)}
\sim \dz^{|k_1-k_2|\eta(1-\thz)}\noz.
\end{align*}
Similar estimate also holds true for
$$
\lf\|\lf(\sum_{|l_1|\le N}Q_{k_1+l_1}Q_{k_1}\r)^*\lf(\sum_{|l_2|\le N}Q_{k_2+l_2}Q_{k_2}\r)\r\|
_{L^2(X)\to L^2(X)}
$$
due to the symmetry. Then, by Lemma \ref{lem:CSlem}, we conclude that $\wz{T}_{N,L}$ is bounded on $L^2(X)$
with its operator norm independent of $N$ and $L$.
Moreover, repeating the proof of Proposition \ref{prop:sizeRN} with the summation $\sum_{|l|>N}$ therein
replaced by $\sum_{|l|\le N}$, we know that the kernel $\wz T_{N,L}$ satisfies all the conditions of Theorem
\ref{thm:Kbdd}, with $c_0:=0$ and $C_T$ therein being a positive constant independent of $L$ and $N$.
Thus, from Theorem \ref{thm:Kbdd}, it follows that $\wz T_{N,L}$ is a bounded operator on
$\mathring\CG(\bz,\gz)$.

Further, recall that $T_N^{-1}Q_k^N=\wz{Q}_k$ and $\|T_N^{-1}\|_{\mathring{\CG}(\bz,\gz)\to
\mathring{\CG}(\bz,\gz)}\le 2$. Therefore,
\begin{align*}
&\lf\|f-\sum_{|k|\le L}\wz{Q}_kQ_kf\r\|_{\mathring{\CG}(\bz,\gz)}\\
&\quad\le \lf\|f-g_n\r\|_{\mathring{\CG}(\bz,\gz)}+\lf\|g_n-\sum_{|k|\le L}\wz{Q}_kQ_kg_n\r\|_{\mathring{\CG}(\bz,\gz)}
+\lf\|T_N^{-1}\lf(\sum_{|k|\le L}Q_k^NQ_k(g_n-f)\r)\r\|_{\mathring{\CG}(\bz,\gz)}\\
&\quad\ls \|f-g_n\|_{\mathring{\CG}(\bz,\gz)}+\lf\|g_n-\sum_{|k|\le L}\wz{Q}_kQ_kg_n\r\|_{\mathring{\CG}(\bz,\gz)},
\end{align*}
which tends to $0$ as $n,\ L\to \infty$.

We still need to prove that $\sum_{|k|\le L}\wz{Q}_kQ_kf$ can be approximated by a sequence of functions in
$\mathring\CG(\eta,\eta)$ in the norm of $\mathring{\CG}(\bz,\gz)$.
Notice that the boundedness of $T_N^{-1}$ and $\wz{T}_{N,L}$ on $\GO{\bz,\gz}$ implies that
$$
\lim_{n\to\infty}\lf\|\sum_{|k|\le L}\wz{Q}_kQ_kf
-\sum_{|k|\le L}\wz{Q}_kQ_kg_n\r\|_{\mathring{\CG}(\bz,\gz)}
=\lim_{n\to\infty}\lf\|T_N^{-1}\wz{T}_{N,L}(f-g_n)\r\|_{\mathring{\CG}(\bz,\gz)}=0.
$$
If we know that $\sum_{|k|\le L}\wz{Q}_kQ_kg_n\in \mathring{\CG}_0^\eta(\bz,\gz)$, then
$\sum_{|k|\le L}\wz{Q}_kQ_kg_n$ can be approximated by functions in $\mathring\CG(\eta,\eta)$, so does $\sum_{|k|\le L}\wz{Q}_kQ_kf$.

Suppose that $h\in\mathring{\CG}^\eta_0(\bz,\gz)$. Then there exists
$\{h_j\}_{j=1}^\infty\subset\mathring{\CG}(\eta,\eta)$ such that $\|h-h_j\|_{\mathring{\CG}(\bz,\gz)}\to 0$ as $j\to\infty$.
By Lemma \ref{lem-add3}, we find that $Q_k^NQ_k h_j=\sum_{|l|\le N}Q_{k+l}Q_k h_j\in\mathring{\CG}(\eta,\eta)$ for
any $k\in\zz$ and $j\in\nn$. By the definition of $T_N$ and the arguments in the proof of Step 1), we conclude that
$\wz T_{N,L}h_j=\sum_{|k|\le L}Q_k^NQ_k h_j$ converges to $T_N h_j$ in $\mathring{\CG}(\bz,\gz)$ as $L\to\infty$. Therefore, every $T_N h_j\in\mathring{\CG}^\eta_0(\bz,\gz)$.
Recall that $R_N$ is bounded on $\mathring{\CG}(\bz,\gz)$
with operator norm at most $1/2$. Thus, $T_N$ is also bounded on $\mathring{\CG}(\bz,\gz)$, which implies that
 $\|T_Nh-T_Nh_j\|_{\mathring{\CG}(\bz,\gz)}\to 0$ as $j\to\infty$. We therefore obtain $T_N h\in \mathring{\CG}^\eta_0(\bz,\gz)$, so does
 $R_N h$ for any $h\in \mathring{\CG}^\eta_0(\bz,\gz)$. Further, applying
 \begin{align*}
 T_N^{-1} h =(I-R_N)^{-1}= \sum_{j=0}^\infty R_N^j h,
 \end{align*}
we know that $T_N^{-1} h \in \mathring{\CG}^\eta_0(\bz,\gz)$.

Again applying Lemma \ref{lem-add3}, we also know that $h:=Q_k^NQ_k g_n\in \mathring\CG(\eta,\eta)\subset
\mathring\CG_0^\eta(\bz,\gz)$. Thus, $\wz{Q}_kQ_kg_n=T_N^{-1}h$ belongs to $ \mathring\CG_0^\eta(\bz,\gz)$, so
does $\sum_{|k|\le L}\wz{Q}_kQ_kg_n$. This finishes the proof of Step 2).

{\it Step 3) Proof of the convergence of \eqref{eq:hcrf} in $L^p(X)$ when $f\in L^p(X)$ with any given
$p\in(1,\infty)$.}

For any $L\in\nn$, define $\wz{T}_L:=\sum_{|k|\le L}\wz{Q}_kQ_k$. Notice that  $\wz{T}_L$ is associated to an
integral kernel
$$
\wz T_L(x,y)=\sum_{|k|\le L}\wz{Q}_kQ_k(x,y) =\sum_{|k|\le L}\int_X\wz{Q}_k(x,z)Q_k(z,y)\,d\mu(y),
\quad\forall\;x,\ y\in X.
$$
With $\wz T_{N,L}$ as defined in \eqref{TNL}, we have $\wz T_L=T_N^{-1}T_{N,L}$.
Recall that both $T_N^{-1}$ and $T_{N,L}$ are  bounded  on $L^2(X)$ with the operator norm independent of $N$
and $L$, so does $\wz T_L$.

Next, we show that $\wz T_L$ is a standard Calder\'on-Zygmund kernel.
We first prove the size condition. Indeed, for any $x,\ y\in X$ with $x\neq y$, by Remark \ref{rem:andef}(i)
and the size condition of $\wz Q_k$, we have
\begin{align*}
&\lf|\wz{Q}_kQ_k(x,y)\r|\\
&\quad=\lf|\int_X\wz{Q}_k(x,z)Q_k(z,y)\,d\mu(y)\r|\\
&\quad\ls\exp\lf\{-\nu'\lf[\frac{d(y,\CY^k)}{\dz^k}\r]^a\r\} \int_X \frac 1{V_{\dz^k}(x)+V(x,z)}
\lf[\frac{\dz^k}{\dz^k+d(x,z)}\r]^\gz\frac 1{V_{\dz^k}(y)}
\exp\lf\{-\nu'\lf[\frac{d(z,y)}{\dz^k}\r]^a\r\}\,d\mu(z).
\end{align*}
For the last integral, we separate $X$ into two domains $\{z\in X:\ d(z,x)\ge(2A_0)^{-1}d(x,y)\}$ and
$\{z\in X:\ d(z,y)\ge(2A_0)^{-1}d(x,y)\}$. Then, from Lemma \ref{lem-add}(ii), we deduce that
\begin{align*}
&\int_X \frac 1{V_{\dz^k}(x)+V(x,z)}\lf[\frac{\dz^k}{\dz^k+d(x,z)}\r]^\gz\frac 1{V_{\dz^k}(y)}
\exp\lf\{-\nu'\lf[\frac{d(z,y)}{\dz^k}\r]^a\r\}\,d\mu(z)\\
&\quad\ls\frac 1{V_{\dz^k}(x)+V(x,y)}\lf[\frac{\dz^k}{\dz^k+d(x,z)}\r]^\gz\int_{d(z,x)\ge (2A_0)^{-1}d(x,y)}
\frac 1{V_{\dz^k}(z)}\exp\lf\{-\nu'\lf[\frac{d(z,y)}{\dz^k}\r]^a\r\}\,d\mu(z)\\
&\qquad+\frac{1}{V_{\dz^k}(y)}\exp\lf\{-\frac{\nu'} 2\lf[\frac{d(x,y)}{\dz^k}\r]^a\r\}
\int_{d(z,y)\ge (2A_0)^{-1}d(x,y)}\frac 1{V_{\dz^k}(x)+V(x,z)}\lf[\frac{\dz^k}{\dz^k+d(x,z)}\r]^\gz\,d\mu(z)\\
&\quad\ls\frac 1{V_{\dz^k}(x)+V(x,y)}\lf[\frac{\dz^k}{\dz^k+d(x,z)}\r]^\gz.
\end{align*}
Therefore, by the two formulae above, we find that
$$
\lf|\wz{Q}_kQ_k(x,y)\r|\ls\frac 1{V_{\dz^k}(x)+V(x,y)}\lf[\frac{\dz^k}{\dz^k+d(x,z)}\r]^\gz
\exp\lf\{-\nu'\lf[\frac{d(y,\CY^k)}{\dz^k}\r]^a\r\},
$$
which, consequently, implies that
\begin{align}\label{eq:wTLsize}
\lf|\wz{T}_L(x,y)\r|
\ls\sum_{\dz^k\ge d(x,y)}\frac 1{V_{\dz^k}(y)}\exp\lf\{-\nu'\lf[\frac{d(y,\CY^k)}{\dz^k}\r]^a\r\}
+\sum_{\dz^k<d(x,y)}\frac 1{V(x,y)}\lf[\frac{\dz^k}{d(x,y)}\r]^\gz\ls\frac 1{V(x,y)}.
\end{align}
Next we consider the regularity condition for the $x$ variable. To this end, we choose
$\bz_1\in(0,\bz\wedge\gz)$. When $d(x,x')\le(2A_0)^{-1}\min\{d(x,y),\dz^k\}$ with $x\neq y$,
applying the regularity of $\wz Q_k$ and the size condition of $Q_k$, and then,
following the same argument as that used in the proof of the size condition of $\wz Q_kQ_k$, we deduce that
\begin{align*}
&\lf|\wz{Q}_kQ_k(x,y)-\wz{Q}_kQ_k(x',y)\r|\\
&\quad=\lf|\int_X\lf[\wz{Q}_k(x,z)-\wz{Q}_k(x',z)\r]Q_k(z,y)\,d\mu(z)\r|\\
&\quad\ls\int_X \lf[\frac{d(x,x')}{\dz^k+d(x,z)}\r]^\bz\frac 1{V_{\dz^k}(x)+V(x,z)}
\lf[\frac{\dz^k}{\dz^k+d(x,z)}\r]^\gz\frac 1{V_{\dz^k}(z)}\exp\lf\{-\nu'\lf[\frac{d(z,y)}{\dz^k}\r]^a\r\}\\
&\qquad\times\exp\lf\{-\nu'\lf[\frac{d(y,\CY^k)}{\dz^k}\r]^a\r\}\,d\mu(z)\\
&\quad\ls\lf[\frac{d(x,x')}{\dz^k}\r]^{\bz_1}\frac1{V_{\dz^k}(x)+V(x,y)}\lf[\frac{\dz^k}{\dz^k+d(x,y)}\r]^\gz
\exp\lf\{-\nu'\lf[\frac{d(y,\CY^k)}{\dz^k}\r]^a\r\}.
\end{align*}
When $(2A_0)^{-1}\dz^k<d(x,x')\le (2A_0)^{-1}d(x,y)$ with $x\neq y$, the size conditions of $\wz Q_k$ and
$Q_k$ imply that
\begin{align*}
\lf|\wz{Q}_kQ_k(x,y)-\wz{Q}_kQ_k(x',y)\r|&\ls\frac1{V_{\dz^k}(x)+V(x,y)}\lf[\frac{\dz^k}{\dz^k+d(x,y)}\r]^\gz
+\frac1{V_{\dz^k}(x')+V(x',y)}\lf[\frac{\dz^k}{\dz^k+d(x',y)}\r]^\gz\\
&\ls\lf[\frac{d(x,x')}{\dz^k}\r]^{\bz_1}\frac1{V(x,y)}\lf[\frac{\dz^k}{d(x,y)}\r]^{\gz}
\sim\lf[\frac{d(x,x')}{d(x,y)}\r]^{\bz_1}\frac1{V(x,y)}\lf[\frac{\dz^k}{d(x,y)}\r]^{\gz-\bz_1}.
\end{align*}
Thus, by the above two formulae and Lemma \ref{lem:expsum}, we conclude that, when
$d(x,x')\le(2A_0)^{-1}d(x,y)$ with $x\neq y$,
\begin{align}\label{eq:wTLregx}
&\lf|\wz{T}_L(x,y)-\wz{T}_L(x',y)\r|\\
&\quad\le\sum_{k=-\fz}^\fz\lf|\wz{Q}_kQ_k(x,y)-\wz{Q}_kQ_k(x',y)\r|\noz\\
&\quad\ls\lf[\frac{d(x,x')}{d(x,y)}\r]^{\bz_1}\lf(\sum_{\dz^k\ge d(x,y)}\frac 1{V_{\dz^k}(y)}
\exp\lf\{-\nu'\lf[\frac{d(y,\CY^k)}{\dz^k}\r]^a\r\}
+\sum_{\dz^k<d(x,y)}\frac 1{V(x,y)}\lf[\frac{\dz^k}{d(x,y)}\r]^\gz\r)\noz\\
&\quad\ls\lf[\frac{d(x,x')}{d(x,y)}\r]^{\bz_1}\frac 1{V(x,y)}.\noz
\end{align}
Similarly, by symmetry, when $d(y,y')\le(2A_0)^{-1}d(x,y)$ with $x\neq y$, we have
\begin{equation}\label{eq:wTLregy}
\lf|\wz{T}_L(x,y)-\wz{T}_L(x,y')\r|\ls\lf[\frac{d(y,y')}{d(x,y)}\r]^{\bz_1}\frac 1{V(x,y)}.
\end{equation}
Combining \eqref{eq:wTLsize}, \eqref{eq:wTLregx} and \eqref{eq:wTLregy}, we find that $\wz T_L$ is a
$\bz_1$-Calder\'on-Zygmund kernel.
The well-known Calder\'on-Zygmund theory on spaces of homogeneous type developed in \cite{CW71} then implies
that the operator $\wz T_L$ is bounded on $L^p(X)$ for any given $p\in(1,\infty)$ with the operator norm
independent of $L$.

Now we suppose $f\in L^p(X)$ with any given $p\in(1,\fz)$. It follows, from \cite[Corollary 10.4]{AH13},
that $\GO{\eta,\eta}$ is dense in $L^p(X)$, so is $\GOO{\bz,\gz}$. Therefore, for any given $\ez\in(0,\fz)$,
there exists $g\in\GOO{\bz,\gz}$ such that $\|f-g\|_{L^p(X)}<\ez$. By the already proved conclusion in
Step 2), there exists $L_0\in\nn$ such that, for any $L\ge L_0$,
$\|g-\sum_{|k|\le L}\wz{Q}_kQ_kg\|_{\GOO{\bz,\gz}}<\ez$.
Observe that Lemma \ref{lem-add}(ii) shows that the space $\GOO{\bz,\gz}$ is continuously embedded into
$L^p(X)$ with the embedding constant depending on $x_0$ and $p$.
Combining these and the boundedness of $\wz T_L$ on $L^p(X)$, we obtain
\begin{align*}
\lf\|f-\sum_{|k|\le L}\wz{Q}_kQ_kf\r\|_{L^p(X)}&\le\|f-g\|_{L^p(X)}+
\lf\|g-\sum_{|k|\le L}\wz{Q}_kQ_kg\r\|_{L^p(X)}+\lf\|\sum_{|k|\le L}\wz{Q}_kQ_k(g-f)\r\|_{L^p(X)}
\ls\ez.
\end{align*}
Meanwhile, notice that Proposition \ref{prop:basic}(iii) implies that $\sum_{|k|\le L}\wz{Q}_kQ_kf\in L^p(X)$.
This proves the convergence of \eqref{eq:hcrf} in $L^p(X)$ when $f\in L^p(X)$, which completes the proof of Theorem \ref{thm:hcrf}.
\end{proof}

Similarly to the proof of Theorem \ref{thm:hcrf}, we obtain another homogeneous continuous Calder\'{o}n
reproducing formula, the details being omitted.
\begin{theorem}\label{thm:hcrf2}
Suppose that $\bz,\ \gz\in(0,\eta)$ and $\{Q_k\}_{k\in\zz}$ is an $\exp$-{\rm ATI}.
Then there exists a sequence $\{\overline{Q}_k\}_{k\in\zz}$ of bounded linear operators on $L^2(X)$ such that,
for any $f$ in  $\GOO{\bz,\gz}$ [resp., $L^p(X)$ with any given $p\in(1,\fz)$],
\begin{equation}\label{eq:hcrf2}
f=\sum_{k=-\fz}^\fz Q_k\overline{Q}_kf,
\end{equation}
where the series converges in $\mathring{\CG}^\eta_0(\bz,\gz)$ [resp., $L^p(X)$ with any given $p\in(1,\fz)$].
Moreover,  for any $k\in\zz$, the kernel of $\overline{Q}_k$ satisfies the size condition \eqref{eq:atisize},
the regularity condition \eqref{eq:atisregx} only for the second variable, and the cancellation condition
\eqref{eq:wzQcan}.
\end{theorem}

\begin{remark}\label{rem:r1}
Based on \eqref{eq:defRT}, we have
$
f=T_Nf+R_Nf=\sum_{k=-\fz}^\fz Q_kQ_k^Nf+R_Nf.
$
Taking the duality in both side of this equality, we obtain
$$
f=T_N^*f+R_N^*f=\sum_{k=-\fz}^\fz\lf(Q_k^N\r)^*Q_k^*f+R_N^*f.
$$
Based on the proof of Theorem \ref{thm:hcrf}, we find that  $\overline{Q}_k$ in Theorem
\ref{thm:hcrf2} is defined by setting, for any $x,\ y\in X$,
$$
\lf(\overline{Q}_k\r)^*(x,y)=\lf(T_N^*\r)^{-1}\lf(\lf(Q_k^N\r)^*(\cdot,y)\r)(x)
=\lf(T_N^*\r)^{-1}\lf(Q_k^N(y,\cdot)\r)(x),
$$
which further implies that $\overline{Q}_k(x,y)=(T_N^*)^{-1}(Q_k^N(x,\cdot))(y)$ for any
$x,\ y\in X$.
\end{remark}

Using Theorems \ref{thm:hcrf} and \ref{thm:hcrf2}, via a duality argument, we have the following conclusion, the
details being omitted.

\begin{theorem}\label{thm:hcrf3}
Let all the notation be as in Theorems \ref{thm:hcrf} and \ref{thm:hcrf2}. Then, for any $f\in(\GOO{\bz,\gz})'$
with $\bz,\ \gz\in(0,\eta)$, both \eqref{eq:hcrf} and \eqref{eq:hcrf2} hold true in $(\GOO{\bz,\gz})'$.
\end{theorem}

\begin{remark}\label{rem:r2}
Observe that, if $K$ is a compact subset of $(0,\eta)^2$, then the estimates related to
$\{\wz{Q}_k\}_{k\in\zz}$ in Theorem \ref{thm:hcrf} are independent of $\bz$ and $\gz$ whenever
$(\bz,\gz)\in K$, but depend on $K$. Similar observations also hold true for Theorems \ref{thm:hcrf2} and
\ref{thm:hcrf3}.
\end{remark}

\section{Homogeneous discrete Calder\'{o}n reproducing formulae}\label{hdrf}

This section concerns the homogeneous discrete reproducing formulae. Let $j_0\in\nn$ be sufficiently large
such that
\begin{align}\label{j_0}
\dz^{j_0}\le(2A_0)^{-4}C^\natural,
\end{align}
where $C^\natural$ is as in Theorem \ref{thm:dys}. Based on Theorem \ref{thm:dys},  for any $k\in\zz$ and
$\az\in\CA_k$, we define
$$
\CN(k,\az):=\{\tau\in\CA_{k+j_0}:\ Q_\tau^{k+j_0}\subset Q_\az^k\}
$$
and $N(k,\az)$ to be the \emph{cardinality} of the set $\CN(k,\az)$. From Theorem \ref{thm:dys}, it follows that
$N(k,\az)\ls\dz^{-j_0\omega}$ and $\bigcup_{\tau\in\CN(k,\az)}Q_\tau^{k+j_0}=Q_\az^k$. We rearrange the set
$\{Q_\tau^{k+j_0}:\ \tau\in\CN(k,\az)\}$ as
$\{Q_\az^{k,m}\}_{m=1}^{N(k,\az)}$. Also, denote by $y_\az^{k,m}$  an arbitrary point in $Q_\az^{k,m}$ and
$z_\az^{k,m}$ the  ``center" of $Q_\az^{k,m}$.

Fix a large integer $N\in\nn$. For any $f\in L^2(X)$ and $x\in X$, define the \emph{discrete Riemannian sum}
\begin{align}\label{eq:S}
\mathcal S_Nf(x):=\sum_{k=-\fz}^\fz\sum_{\az\in\CA_k}
\sum_{m=1}^{N(k,\az)}\int_{Q_\az^{k,m}}Q_k^N(x,y)\,d\mu(y)Q_kf\lf(y_\az^{k,m}\r),
\end{align}
 where $y_\az^{k,m}$ is an arbitrary point in $Q_\az^{k,m}$.
Notice that, for any $f\in L^2(X)$ and $x\in X$, we have
\begin{align}\label{eq:defGR}
\mathcal R_Nf(x):={}&(I-\mathcal  S_N)f(x)\\
={}&\sum_{k=-\fz}^\fz\sum_{\az\in\CA_k}\sum_{m=1}^{N(k,\az)}\int_{Q_\az^{k,m}}Q_k^N(x,y)
\lf[Q_kf(y)-Q_kf\lf(y_\az^{k,m}\r)\r]\,d\mu(y)\noz\\
&+\sum_{|l|>N}\sum_{k=-\fz}^\fz Q_{k+l}Q_kf(x)\noz\\
=:{}&\sum_{k=-\fz}^\fz G_{k, N}f(x)+R_Nf(x)=:G_Nf(x)+R_Nf(x),\noz
\end{align}
where $R_N$ is as in \eqref{eq:defRT}.
We have shown, in the previous section, that $R_N$ is bounded on both $L^2(X)$ and
$\mathring{\CG}(x_1,r,\bz,\gz)$, where $x_1\in X$, $r\in(0,\fz)$ and
$\bz,\ \gz\in(0,\eta)$. We will show, in Section \ref{sec5.1} below, that the operator $G_N$ is also bounded
on both $L^2(X)$ and $\mathring{\CG}(x_1,r,\bz,\gz)$. Thus, if $f$ belongs to $L^2(X)$ [resp.,
$\mathring{\CG}(x_1,r,\bz,\gz)$], then $\CS_Nf$ in \eqref{eq:S} is a well defined  function in $L^2(X)$ [resp.,
in $\mathring{\CG}(x_1,r,\bz,\gz)$]. The proofs for the homogeneous discrete reproducing formulae are presented
in Section \ref{pr2}.

\subsection{Boundedness of the remainder $\CR_N$}\label{sec5.1}

Based on the discussion after \eqref{eq:defGR}, the boundedness of $\CR_N$ can be reduced to the study
of the boundedness of $G_N$ on  $L^2(X)$ and spaces of test functions. We need the following lemma.

\begin{lemma}\label{lem:GkN}
For any $k\in\zz$ and $N\in\nn$, let $G_{k,N}$ be defined as in \eqref{eq:defGR}, that is, for any $x,\ y\in X$,
$$
G_{k,N}(x,y)=\sum_{\az\in\CA_k}\sum_{m=1}^{N(k,\az)}\int_{Q_\az^{k,m}}Q_k^N(x,z)
\lf[Q_k(z,y)-Q_k\lf(y_\az^{k,m},y\r)\r]\,d\mu(z).
$$
Then there exist positive constants $C_{(N)}$ and $c_{(N)}$, depending on $N$, but being independent of $k$,
$j_0$ and $y_\az^{k,m}$, such that $G_{k,N}$ satisfies:
\begin{enumerate}
\item for any $x,\ y\in X$,
\begin{equation}\label{eq:Gksize}
|G_{k,N}(x,y)|\le C_{(N)}\dz^{j_0\eta}\frac 1{V_{\dz^k}(x)}\exp\lf\{-c'_{(N)}\lf[\frac{d(x,y)}{\dz^k}\r]^a\r\}
\exp\lf\{-c'_{(N)}\lf[\frac{d(x,\CY^k)}{\dz^k}\r]^a\r\};
\end{equation}
\item for any $x,\ x',\ y\in X$ with $d(x,x')\le\dz^k$ or $d(x,x')\le(2A_0)^{-1}[\dz^k+d(x,y)]$,
\begin{align}\label{eq:Gkregx}
&|G_{k,N}(x,y)-G_{k,N}(x',y)|+|G_k(y,x)-G_k(y,x')|\\
&\quad\le C_{(N)}\dz^{j_0\eta}\lf[\frac{d(x,x')}{\dz^k}\r]^\eta
\frac 1{V_{\dz^k}(x)}\exp\lf\{-c'_{(N)}\lf[\frac{d(x,y)}{\dz^k}\r]^a\r\}
\exp\lf\{-c'_{(N)}\lf[\frac{d(x,\CY^k)}{\dz^k}\r]^a\r\}\noz;
\end{align}
\item for any $x,\ x',\ y,\ y'\in X$ with $d(x,x')\le\dz^k$ and $d(y,y')\le\dz^k$, or
$d(x,x')\le(2A_0)^{-2}[\dz^k+d(x,y)]$ and $d(y,y')\le(2A_0)^{-2}[\dz^k+d(x,y)]$,
\begin{align}\label{eq:Gkdreg}
&|[G_{k,N}(x,y)-G_{k,N}(x',y)]-[G_{k,N}(x,y')-G_{k,N}(x',y')]|\\
&\quad\le C_{(N)}\dz^{j_0\eta}\lf[\frac{d(x,x')}{\dz^k}\r]^\eta\lf[\frac{d(y,y')}{\dz^k}\r]^\eta
\frac 1{V_{\dz^k}(x)}\exp\lf\{-c'_{(N)}\lf[\frac{d(x,y)}{\dz^k}\r]^a\r\}\noz\\
&\qquad\times\exp\lf\{-c'_{(N)}\lf[\frac{d(x,\CY^k)}{\dz^k}\r]^a\r\};\noz
\end{align}
\item for any $x\in X$, $\int_X G_k(x,y)\,d\mu(y)=0=\int_X G_k(y,x)\,d\mu(y)$.
\end{enumerate}
\end{lemma}

\begin{proof}
We first prove (i). Indeed, from \eqref{j_0}, it follows that, for any $z\in Q_\az^{k,m}$,
\begin{align}\label{eq-xx}
d\lf(z,y_\az^{k,m}\r)\le(2A_0)^2C^\natural\dz^{k+j_0}\le(2A_0)^{-2}\dz^k\le\dz^k.
\end{align}
With this, applying \eqref{eq:QkNsize}, Remark \ref{rem:andef}(i), \eqref{eq-add7}, Lemma \ref{lem-add}(ii)
and \eqref{eq-add8}, we conclude that, for any $x,\ y\in X$,
\begin{align*}
|G_{k,N}(x,y)|&\ls\frac 1{V_{\dz^k}(x)}\sum_{\az\in\CA_k}\sum_{m=1}^{N(k,\az)}\int_{Q_\az^{k,m}}
\exp\lf\{-c\lf[\frac{d(x,z)}{\dz^k}\r]^a\r\}\lf[\frac{d(z,y_\az^{k,m})}{\dz^k}\r]^\eta\frac 1{V_{\dz^k}(z)}
\\
&\quad\times\exp\lf\{-\nu'\lf[\frac{d(z,y)}{\dz^k}\r]^a\r\}\exp\lf\{-\nu'\lf[\frac{d(y,\CY^k)}{\dz^k}\r]^a\r\}\,d\mu(z)\\
&\ls\dz^{j_0\eta}\frac 1{V_{\dz^k}(x)}\exp\lf\{-\frac{c\wedge\nu'}2\lf[\frac{d(x,y)}{A_0\dz^k}\r]^a\r\}
\exp\lf\{-\nu'\lf[\frac{d(y,\CY^k)}{\dz^k}\r]^a\r\}\\
&\quad\times\int_X\frac{1}{V_{\dz^k}(z)}\exp\lf\{-\frac{\nu'}{2}\lf[\frac{d(z,y)}{\dz^k}\r]^a\r\}
\,d\mu(z)\\
&\ls\dz^{j_0\eta}\frac 1{V_{\dz^k}(x)}\exp\lf\{-\frac{c\wedge\nu'}4\lf[\frac{d(x,y)}{A_0\dz^k}\r]^a\r\}
\exp\lf\{-\frac{c\wedge\nu'}4\lf[\frac{d(x,\CY^k)}{A_0^2\dz^k}\r]^a\r\},
\end{align*}
as desired. Moreover, by the dominated convergence theorem and the cancellation of $Q_k$, we also obtain (iv).

Suppose now that $y\in X$ and $x,\ x'\in X$ satisfying $d(x,x')\le\dz^k$. Then, from \eqref{eq:QkNregx}, the
regularity condition of $Q_k$, Remark \ref{rem:andef}(i), \eqref{eq-add7}, Lemma \ref{lem-add}(ii) and
\eqref{eq-add8}, we deduce that
\begin{align*}
&|G_{k,N}(x,y)-G_{k,N}(x',y)|\\
&\quad\le\sum_{\az\in\CA_k}\sum_{m=1}^{N(k,\az)}\int_{Q_\az^{k,m}}
\lf|Q_k^N(x,z)-Q_k^N(x',z)\r|\lf|Q_k(z,y)-Q_k\lf(y_\az^{k,m},y\r)\r|\,d\mu(z)\\
&\quad\ls\lf[\frac{d(x,x')}{\dz^k}\r]^\eta
\frac 1{V_{\dz^k}(x)}\sum_{\az\in\CA_k}\sum_{m=1}^{N(k,\az)}\int_{Q_\az^{k,m}}
\exp\lf\{-c\lf[\frac{d(x,z)}{\dz^k}\r]^a\r\}\lf[\frac{d(z,y_\az^{k,m})}{\dz^k}\r]^\eta\frac 1{V_{\dz^k}(z)}\\
&\qquad\times\exp\lf\{-\nu'\lf[\frac{d(z,y)}{\dz^k}\r]^a\r\}
\exp\lf\{-\nu'\lf[\frac{d(y,\CY^k)}{\dz^k}\r]^a\r\}\,d\mu(z)\\
&\quad\ls\dz^{j_0\eta}\lf[\frac{d(x,x')}{\dz^k}\r]^\eta\frac 1{V_{\dz^k}(x)}
\exp\lf\{-\frac{c\wedge\nu'}2\lf[\frac{d(x,y)}{A_0\dz^k}\r]^a\r\}
\exp\lf\{-\nu'\lf[\frac{d(y,\CY^k)}{\dz^k}\r]^a\r\}\\
&\qquad\times\int_X\frac{1}{V_{\dz^k}(z)}\exp\lf\{-\frac{\nu'}{2}
\lf[\frac{d(z,y)}{\dz^k}\r]^a\r\}\,d\mu(z)\\
&\quad\ls\dz^{j_0\eta}\lf[\frac{d(x,x')}{\dz^k}\r]^\eta
\frac 1{V_{\dz^k}(x)}\exp\lf\{-\frac{c\wedge\nu'}4\lf[\frac{d(x,y)}{A_0\dz^k}\r]^a\r\}
\exp\lf\{-\frac{c\wedge\nu'}4\lf[\frac{d(x,\CY^k)}{ A_0^2\dz^k}\r]^a\r\}.
\end{align*}
Similarly, when $d(x,x')\le\dz^k$, by \eqref{eq:QkNsize}, Remark \ref{rem:andef}(i),
\eqref{eq-add7} and Lemma \ref{lem-add}(ii), we conclude that
\begin{align*}
&|G_{k,N}(y,x)-G_{k,N}(y,x')|\\
&\quad\le\sum_{\az\in\CA_k}\sum_{m=1}^{N(k,\az)}\int_{Q_\az^{k,m}}
\lf|Q_k^N(y,z)\r|\lf|Q_k(z,x)-Q_k\lf(y_\az^{k,m},x\r)-Q_k(z,x')+Q_k\lf(y_\az^{k,m},x'\r)\r|\,d\mu(z)\\
&\quad\ls\lf[\frac{d(x,x')}{\dz^k}\r]^\eta
\frac 1{V_{\dz^k}(x)}\sum_{\az\in\CA_k}\sum_{m=1}^{N(k,\az)}\int_{Q_\az^{k,m}}\frac 1{V_{\dz^k}(z)}
\exp\lf\{-c\lf[\frac{d(y,z)}{\dz^k}\r]^a\r\}\lf[\frac{d(z,y_\az^{k,m})}{\dz^k}\r]^\eta\\
&\qquad\times\exp\lf\{-\nu'\lf[\frac{d(z,x)}{\dz^k}\r]^a\r\}
\exp\lf\{-\nu'\lf[\frac{d(x,\CY^k)}{\dz^k}\r]^a\r\}\,d\mu(z)\\
&\quad\ls\dz^{j_0\eta}\lf[\frac{d(x,x')}{\dz^k}\r]^\eta
\frac 1{V_{\dz^k}(x)}\exp\lf\{-\frac{c\wedge\nu'}2\lf[\frac{d(x,y)}{2A_0\dz^k}\r]^a\r\}
\exp\lf\{-\nu'\lf[\frac{d(x,\CY^k)}{\dz^k}\r]^a\r\}
\end{align*}
By the two formulae above, we obtain (ii) when $d(x,x')\le\dz^k$. Using (i) and arguing as Proposition
\ref{prop:etoa} [see also Remark \ref{rem:andef}(iii)], we conclude that \eqref{eq:Gkregx} remains true when
$d(x,x')\le(2A_0)^{-1}[\dz^k+d(x,y)]$.

Based on (ii) and the proof of Proposition
\ref{prop:etoa} [see also Remark \ref{rem:andef}(iii)], we only show that (iii) holds true when
$x,\ x',\ y,\ y'\in X$ satisfying $d(x,x')\le\dz^k$ and $d(y,y')\le\dz^k$. In this case, applying
$\eqref{eq:QkNregx}$, the second difference regularity of $Q_k$, Remark \ref{rem:andef}(i), \eqref{eq-add7},
Lemma \ref{lem-add}(ii) and \eqref{eq-add8}, we obtain
\begin{align*}
&|[G_{k,N}(x,y)-G_{k,N}(x',y)]-[G_{k,N}(x,y')-G_{k,N}(x',y')]|\\
&\quad\le\sum_{\az\in\CA_k}\sum_{m=1}^{N(k,\az)}
\int_{Q_\az^{k,m}}\lf|Q_k^N(x,z)-Q_k^N(x',z)\r|
\lf|Q_k(z,y)-Q_k\lf(y_\az^{k,m},y\r)-Q_k(z,y')+Q_k\lf(y_\az^{k,m},y'\r)\r|\,d\mu(z)\\
&\quad\ls\lf[\frac{d(x,x')}{\dz^k}\r]^\eta\lf[\frac{d(y,y')}{\dz^k}\r]^\eta
\frac 1{V_{\dz^k}(x)}\sum_{\az\in\CA_k}\sum_{m=1}^{N(k,\az)}\int_{Q_\az^{k,m}}
\exp\lf\{-c\lf[\frac{d(x,z)}{\dz^k}\r]^a\r\}\lf[\frac{d(z,y_\az^{k,m})}{\dz^k}\r]^\eta\frac 1{V_{\dz^k}(z)}\\
&\qquad\times\exp\lf\{-\nu'\lf[\frac{d(z,y)}{\dz^k}\r]^a\r\}
\exp\lf\{-\nu'\lf[\frac{d(y,\CY^k)}{\dz^k}\r]^a\r\}\,d\mu(z)\\
&\quad\ls\dz^{j_0\eta}\lf[\frac{d(x,x')}{\dz^k}\r]^\eta\lf[\frac{d(y,y')}{\dz^k}\r]^\eta
\frac 1{V_{\dz^k}(x)}\exp\lf\{-\frac{c\wedge\nu'}2\lf[\frac{d(x,y)}{A_0\dz^k}\r]^a\r\}
\exp\lf\{-\nu'\lf[\frac{d(y,\CY^k)}{\dz^k}\r]^a\r\}\\
&\qquad\times\int_X\frac{1}{V_{\dz^k}(z)}\exp\lf\{-\frac{\nu'}{2}\lf[\frac{d(z,y)}{\dz^k}\r]^a\r\}\,d\mu(z)\\
&\quad\ls\dz^{j_0\eta}\lf[\frac{d(x,x')}{\dz^k}\r]^\eta\lf[\frac{d(y,y')}{\dz^k}\r]^\eta
\frac 1{V_{\dz^k}(x)}\exp\lf\{-\frac{c\wedge\nu'}4\lf[\frac{d(x,y)}{A_0\dz^k}\r]^a\r\}
\exp\lf\{-\frac{c\wedge\nu'}4\lf[\frac{d(x,\CY^k)}{ A_0^2\dz^k}\r]^a\r\}.
\end{align*}
This proves (iii) and hence finishes the proof of Lemma \ref{lem:GkN}.
\end{proof}

To consider the boundedness of $G_N$ on the spaces of test functions, we use the method
similar to the proof of $R_N$ in Section \ref{RN}. For any $N,\ M\in\nn$ and $x,\ y\in X$, define
\begin{equation}\label{eq:defGM}
G_N^{(M)}(x,y):=\sum_{|k|\le M} G_{k,N}(x,y)
=\sum_{|k|\le M}\sum_{\az\in\CA_k}\sum_{m=1}^{N(k,\az)}\int_{Q_\az^{k,m}}Q_k^N(x,z)
\lf[Q_k(z,y)-Q_k\lf(y_\az^{k,m},y\r)\r]\,d\mu(z).
\end{equation}

\begin{proposition}\label{prop:GNM}
For any $N,\ M\in\nn$, let $G_N$ and $G_N^{(M)}$ be as in \eqref{eq:defGR} and \eqref{eq:defGM}. Let
$x_1\in X$, $r\in(0,\infty)$ and $\bz,\ \gz\in(0,\eta)$. Then there exists a positive constant $C_{(N)}$,
depending on $N$, but being independent of $M$, $j_0$, $x_1$, $r$ and $y_\az^{k,m}$, such that
\begin{equation}\label{eq:GL2}
\|G_N\|_{L^2(X)\to L^2(X)}+\lf\|G_N^{(M)}\r\|_{L^2(X)\to L^2(X)}\le C_{(N)}\dz^{j_0\eta}
\end{equation}
and
\begin{equation}\label{eq:GMG}
\lf\|G_{N}^{(M)}\r\|_{\mathring{\CG}(x_1,r,\bz,\gz)\to\mathring{\CG}(x_1,r,\bz,\gz)}\le C_{(N)}\dz^{j_0\eta}.
\end{equation}
\end{proposition}

\begin{proof}
Applying (i) and (ii) of Lemma \ref{lem:GkN} and Remark \ref{rem:andef}, we follow the proof of Lemma
\ref{lem:ccrf1}(i) with $\wz Q_j$ replaced by $G_{k,N}$ or $G_{k,N}^*$, and $Q_k$ replaced by $G_{l,N}$ or
$G_{l,N}^*$, to deduce that, for any $k,\ l\in\zz$ and $x,\ y\in X$,
$$
\lf|G_{k,N}^*G_{l,N}(x,y)\r|+\lf|G_{k,N}G_{l,N}^*(x,y)\r|\le \wz C\dz^{j_0\eta}\dz^{|k-l|\eta}\frac 1{V_{\dz^{k\wedge l}}(x)}
\exp\lf\{-\wz{c}\lf[\frac{d(x,y)}{\dz^{k\wedge l}}\r]^a\r\},
$$
where $\wz C$ and $\wz{c}$ are positive constants independent of $k$, $l$, $x$, $y$, $j_0$ and $y_\az^{k,m}$.
Combining this with Proposition \ref{prop:basic}(iii), we find that, for any $k,\ l\in\zz$,
$$
\lf\|G_{k,N}^*G_{l,N}\r\|_{L^2(X)\to L^2(X)}+\lf\|G_{k,N}G_{l,N}^*\r\|_{L^2(X)\to L^2(X)}
\ls_N\dz^{j_0\eta}\dz^{|k-l|\eta}.
$$
Then \eqref{eq:GL2} follows from  Lemma \ref{lem:CSlem}.

Now we show \eqref{eq:GMG}. Let $M\in\nn$. By the definition of $G_{N}^{(M)}$ and \eqref{eq:Gksize}, for
any $f\in C^\bz(X)$ and $x\in X$, we have
$$
G_N^{(M)}f(x)=\int_X G_N^{(M)}(x,y)f(y)\,d\mu(y).
$$
It remains to prove that $G_N^{(M)}$ satisfies \eqref{eq:Ksize}, \eqref{eq:Kreg} and \eqref{eq:Kdreg} with
$C_T:=C\dz^{j_0\eta}$, where $C$ is a positive constant depending on $N$, but being independent of $M$ and
$j_0$. For the size condition, by \eqref{eq:Gksize} and Lemma \ref{lem:sum2}, we conclude that, for any
$x,\ y\in X$ with $x\neq y$,
\begin{align*}
\lf|G_N^{(M)}(x,y)\r|&\le\sum_{|k|\le M}|G_{k,N}(x,y)|\\
&\ls\dz^{j_0\eta}\sum_{k=-\fz}^\fz\frac 1{V_{\dz^k}(x)}\exp\lf\{-c'\lf[\frac{d(x,y)}{\dz^k}\r]^a\r\}
\exp\lf\{-c'\lf[\frac{d(x,\CY^k)}{\dz^k}\r]^a\r\}
\ls\dz^{j_0\eta}\frac 1{V(x,y)}.
\end{align*}
Similarly, by \eqref{eq:Gkregx} [resp., \eqref{eq:Gkdreg}] and Lemma \ref{lem:sum2}, we obtain the regularity
condition (resp., the second difference regularity condition).
Notice that $G_N^{(M)}$ is bounded on $L^2(X)$ with its operator norm independent of $j_0$ and $M$ [see
\eqref{eq:GL2}] and the kernel of $G_N^{(M)}$ has
cancellation condition. Then, applying Theorem \ref{thm:Kbdd}, we obtain \eqref{eq:GMG}. This finishes the
proof of Proposition \ref{prop:GNM}.
\end{proof}

\begin{lemma}\label{lem:GMlim}
Let $f\in\CG(x_1,r,\bz,\gz)$ with $x_1\in X$, $r\in(0,\fz)$ and $\bz,\ \gz\in(0,\eta)$. Fix $N\in\nn$.
Then the following assertions are true:
\begin{enumerate}
\item $\lim_{M\to\fz} G_N^{(M)}=G_Nf$ in $L^2(X)$;
\item for any $x\in X$, the sequence $\{G_{N}^{(M)}f(x)\}_{M=1}^\infty$ converges locally uniformly to some
element, denoted by $\widetilde{G_N} f(x)$, where $\widetilde{G_N}f$ differs from $R_N f$ at most a set of
$\mu$-measure $0$;
\item the operator $\widetilde {G_N}$ can be uniquely extended from $\CG(x_1,r,\bz,\gz)$ to $L^2(X)$, with the
extension operator coinciding with $G_N$. In this sense, for any $f\in\CG(x_1,r,\bz,\gz)$ and $x\in X$,
\begin{equation*}
\lim_{M\to\fz}G_N^{(M)}f(x)=\widetilde{G_N}f(x)=G_Nf(x).
\end{equation*}
\end{enumerate}
\end{lemma}

\begin{proof}
To obtain (i), applying Lemma \ref{lem:CSlem} and \eqref{eq:defGM}, we find that, for any $f\in L^2(X)$,
$$
G_Nf=\sum_{k=-\fz}^\fz G_{k,N}f=\lim_{M\to\fz}G_N^{(M)}f\quad\text{in $L^2(X)$}.
$$

Next we prove (ii). Let $f\in\CG(x_1,r,\bz,\gz)$. To prove that $\{G_{N}^{(M)}f\}_{M=1}^\fz$ is a locally
uniformly convergent sequence, for any fixed point $x\in X$, we only need to find a positive sequence
$\{c_{k}\}_{k=-\fz}^\fz$ such that
\begin{equation}\label{eq:sfin2}
\sup_{y\in B(x,r)}|G_{k,N}f(y)|\le c_{k}\qquad \textup{and}\qquad
\sum_{k=-\fz}^\fz c_{k}<\fz.
\end{equation}
We claim that
\begin{align}\label{claim-add2}
\sup_{y\in B(x,r)}|G_{k,N}f(y)|\ls
\begin{cases}
\displaystyle \frac 1{V_{\dz^{k}}(x)}\exp\lf\{-c\lf[\frac{d(x,\CY^{k})}{A_0\dz^{k}}\r]^a\r\}
&\qquad \textup{if}\;\dz^{k}\ge r,\\
\displaystyle \frac 1{V_r(x_1)}\lf(\frac{\dz^k}r\r)^\bz&\qquad \textup{if}\;\dz^k<r.
\end{cases}
\end{align}
With each $c_k$ defined as in the right-hand side of \eqref{claim-add2}, from Lemma
\ref{lem:expsum} we deduce $\sum_{k=-\fz}^\fz c_k<\fz$, so that \eqref{eq:sfin2} holds true.

Now we  show \eqref{claim-add2}.
When $\dz^k\ge r$,  by \eqref{eq:Gksize}, Lemma \ref{lem-add}(ii) and \eqref{eq-add8}, we conclude that,
for any $y\in B(x,r)$,
\begin{align*}
|G_{k,N}f(y)|&\le\int_X\lf|G_{k,N}(y,z)\r||f(z)|\,d\mu(z)\\
&\ls\frac 1{V_{\dz^k}(y)}\exp\lf\{-c'\lf[\frac{d(y,\CY^k)}{\dz^k}\r]^a\r\}
\int_X|f(z)|\,d\mu(z)\ls\frac 1{V_{\dz^k}(x)}\exp\lf\{-c'\lf[\frac{d(x,\CY^k)}{A_0\dz^k}\r]^a\r\},
\end{align*}
as desired.
When $\dz^k<r$, by Lemma \ref{lem:GkN}(iv), for any $y\in X$, we have
$$
|G_{k,N}f(y)|=\lf|\int_X G_{k,N}(y,z)[f(z)-f(y)]\,d\mu(y)\r|\le\int_X|G_{k,N}(y,z)||f(z)-f(y)|\,d\mu(z).
$$
As was shown in the proof of Lemma \ref{lem:ccrf3}, for any $f\in\CG(x_1, r,\bz,\gz)$ and $y,\ z\in X$,
\begin{align*}
|f(z)-f(y)|\ls \lf[\frac{d(y,z)}{r}\r]^\bz\frac 1{V_r(x_1)}.
\end{align*}
This, together with \eqref{eq:Gksize} and Lemma \ref{lem-add}(ii), implies that
\begin{align*}
|G_{k,N}f(y)|&\ls\dz^{\eta l}\frac 1{V_r(x_1)} \lf(\frac{\dz^k}r\r)^\bz
\int_X\frac 1{V_{\dz^k}(y)}
\exp\lf\{-c\lf[\frac{d(y,z)}{\dz^k}\r]^a\r\}\lf[\frac{d(y,z)}{\dz^k}\r]^\bz\,d\mu(z)
\ls\dz^{\eta l}\frac 1{V_r(x_1)}\lf(\frac{\dz^k}r\r)^\bz.
\end{align*}
This proves \eqref{claim-add2}. Consequently, for any $x\in X$, the sequence $\{G_N^{(M)}f(x)\}_{M=N+1}^\infty$ converges locally uniformly to some element, which is denoted by $\widetilde{G_N} f(x)$.

Moreover, by (i) and the Riesz theorem, there exists an increasing sequence $\{M_j\}_{j\in\nn}\subset\nn$,
which tends to $\infty$, such that
\begin{equation*}
\lim_{j\to\fz} G_{N}^{(M_j)}f(x)=G_Nf(x)\quad\text{$\mu$-almost every $x\in X$}.
\end{equation*}
Thus,
$\wz{G_N}f(x)=\lim_{M\to\fz} G_{N}^{(M)}f(x)=\lim_{j\to\fz} G_{N}^{(M_j)}f(x)=G_Nf(x)$ for $\mu$-almost every
$x\in X$. This finishes the proof of (ii).

Notice that $G_{N}$ is bounded on $L^2(X)$ and $\CG(x_1,r,\bz,\gz)$ is dense on $L^2(X)$. By (ii) and an
argument similar to that used in the proof of Lemma \ref{lem:ccrf3} with $R_{N,M}$ therein replaced by
$G_N^{(M)}$ and $R_N$ replaced by $G_N$, we easily obtain (iii). This finishes the proof of Lemma
\ref{lem:GMlim}.
\end{proof}

By \eqref{eq:GL2}, Lemma \ref{lem:GMlim}, \eqref{eq:GMG}, Definition \ref{def:test} and the dominated
convergence theorem, we immediately obtain the following result, the details being omitted.

\begin{proposition}\label{prop:GG}
Let  $x_1\in X$, $r\in(0,\fz)$ and $\bz,\ \gz\in(0,\eta)$.
Then, for any $N\in\nn$, there exists a positive constant $C_{(N)}$, depending on $N$, but being independent of
$x_1$, $r$, $j_0$ and $y_\az^{k,m}\in Q_\az^{k,m}$, such that
$$
\|G_N\|_{L^2(X)\to L^2(X)}+\|G_N\|_{\mathring{\CG}(x_1,r,\bz,\gz)\to\mathring{\CG}(x_1,r,\bz,\gz)}
\le C_{(N)}\dz^{j_0\eta}.
$$
Consequently, for any $\eta'\in(\max\{\bz,\gz\},\eta)$, there exists another positive constant $C$, which is independent of $N$, $j_0$ and
$y_\az^{k,m}$, such that
\begin{equation*}
\|\CR_N\|_{L^2(X)\to L^2(X)}\le C\dz^{\eta'N}+C_{(N)}\dz^{j_0\eta}\quad\text{and}\quad
\|\CR_N\|_{\GO{x_1,r,\bz,\gz}\to\GO{x_1,r,\bz,\gz}}\le C\dz^{(\eta-\eta')N}+C_{(N)}\dz^{j_0\eta}.
\end{equation*}
\end{proposition}

\begin{remark}\label{rem:j0N}
Let all the notation be as in Proposition \ref{prop:GG}. To proceed  the proof  in Section
\ref{pr2} below, we first fix $N\in\nn$ sufficiently large such that
$\max\{C\dz^{\eta'N},C\dz^{(\eta-\eta')N}\}<1/4$,
and then fix $j_0\in\nn$ sufficiently large such that both \eqref{j_0}
and $C_{(N)}\dz^{j_0\eta}<1/4$ hold true. From
Proposition \ref{prop:GG}, it then follows that
\begin{equation*}
\max\lf\{\|\CR_N\|_{L^2(X)\to L^2(X)},\|\CR_N\|_{\GO{x_1,r,\bz,\gz}\to\GO{x_1,r,\bz,\gz}}\r\}\le \frac 12.
\end{equation*}
\end{remark}

\subsection{Proofs of homogeneous discrete Calder\'{o}n reproducing formulae}\label{pr2}

The main aim of this subsection is to establish the following homogeneous discrete Calder\'{o}n reproducing
formulae. We need the following estimates regarding an analogous discrete version of the kernel $E_k$ in Lemma
\ref{lem-add2}.

\begin{lemma}\label{lem-add4}
Let $N\in\nn$ and $\{Q_k\}_{k=-\fz}^\fz$ be an $\exp$-{\rm ATI}.
For any $k\in\zz$ and $x,\ y\in X$, define
$$
F_k(x,y):=\sum_{\az\in\CA_k}\sum_{m=1}^{N(k,\az)}\int_{Q_\az^{k,m}}Q_k^N(x,z)\,d\mu(z)Q_k\lf(y_\az^{k,m},y\r),
$$
where $y_\az^{k,m}$ is an arbitrary point in $Q_\az^{k,m}$. Then there exist positive constants $C_{(N)}$ and
$c_{(N)}$, depending on $N$, but being independent of $k$, $j_0$ and $y_\az^{k,m}$, such that
the following assertions are true:
\begin{enumerate}
\item for any $x,\,y\in X$,
\begin{equation*}
|F_k(x,y)|\le C_{(N)}\frac{1}{V_{\dz^k}(x)}\exp\lf\{-c_{(N)}\lf[\frac{d(x,y)}{\dz^{k}}\r]^a\r\}
\exp\lf\{-c\lf[\frac{d(x,\CY^k)}{\dz^{k}}\r]^a\r\};
\end{equation*}
\item for any $x,\ y,\ y'\in X$ with $d(y,y')\le\dz^k$ or $d(y,y')\le (2A_0)^{-1}[\dz^k+d(x, y)]$,
\begin{align*}
&|F_k(x,y)-F_k(x,y')|+|F_k(y,x)-F_k(y',x)|\\
&\quad\le C_{(N)}\lf[\frac{d(y,y')}{\dz^k}\r]^\eta
\frac{1}{V_{\dz^k}(x)}\exp\lf\{-c_{(N)}\lf[\frac{d(x,y)}{\dz^{k}}\r]^a\r\}
\exp\lf\{-c_{(N)}\lf[\frac{d(x,\CY^k)}{\dz^{k}}\r]^a\r\};
\end{align*}
\item  for any $x,\ x',\ y,\ y'\in X$ with $d(x,x')\le\dz^k$ and $d(y,y')\le\dz^k$, or
$d(x,x')\le (2A_0)^{-2}[\dz^k+d(x, y)]$ and $d(y,y')\le(2A_0)^{-2}[\dz^k+d(x,y)]$,
\begin{align*}
&|[F_k(x,y)-F_k(x',y)]-[F_k(x,y')-F_k(x',y')]|\\
&\le C_{(N)}
\lf[\frac{d(x,x')}{\dz^k}\r]^\eta\frac{1}{V_{\dz^k}(x)}\exp\lf\{-c_{(N)}\lf[\frac{d(x,y)}{\dz^{k}}\r]^a\r\}
\exp\lf\{-c_{(N)}\lf[\frac{d(x,\CY^k)}{\dz^{k}}\r]^a\r\};
\end{align*}
\item for any $x\in X$, $\int_X F_k(x,y)\,d\mu(y)=0=\int_X F_k(y,x)\,d\mu(y)$.
\end{enumerate}
\end{lemma}

\begin{proof}
We first prove (i). By the size conditions of $Q_k^N$ [see \eqref{eq:QkNsize}] and $Q_k$, and Remark
\ref{rem:andef}(i), we find that, for any $x,\ y\in X$,
\begin{align*}
|F_k(x,y)|&\ls\frac 1{V_{\dz^k}(x)}\exp\lf\{-\nu'\lf[\frac{d(y,\CY^k)}{\dz^k}\r]^a\r\}
\sum_{\az\in\CA_k}\sum_{m=1}^{N(k,\az)}\int_{Q_\az^{k,m}}
\exp\lf\{-c\lf[\frac{d(x,z)}{\dz^k}\r]^a\r\}\frac 1{V_{\dz^k}(y)}\\
&\quad\times\exp\lf\{-\nu'\lf[\frac{d(y_\az^{k,m},y)}{\dz^k}\r]^a\r\}\,d\mu(z).
\end{align*}
For any $z\in Q_\az^{k,m}$ and $y\in X$, we have
$d(z,y_\az^{k,m})\le(2A_0)^2C^\natural\dz^{k+j_0}\le(2A_0)^{-2}\dz^k\le\dz^k$ and hence
$$
[d(z, y)]^a\le A_0^a\dz^{ka}+ A_0^a[d(y_\az^{k,m},y)]^a,
$$
so that $\exp\{-\nu'[\frac{d(y_\az^{k,m},y)}{\dz^k}]^a\}$ in the above formula can be replaced by $\exp\{-\nu'[\frac{d(z,y)}{A_0\dz^k}]^a\}$.
Define $\nu'':=c\wedge v'$. By this,  \eqref{eq-add7}, Lemma
\ref{lem-add}(ii) and \eqref{eq-add8}, we
conclude that
\begin{align*}
|F_k(x,y)|&\ls\frac 1{V_{\dz^k}(x)}\exp\lf\{-\nu'\lf[\frac{d(y,\CY^k)}{\dz^k}\r]^a\r\}
\int_X
\exp\lf\{-c\lf[\frac{d(x,z)}{\dz^k}\r]^a\r\}\frac 1{V_{\dz^k}(y)}\\
&\quad\times\exp\lf\{-\nu'\lf[\frac{d(z,y)}{A_0\dz^k}\r]^a\r\}\,d\mu(z)\\
&\ls\frac 1{V_{\dz^k}(x)}\exp\lf\{-\frac{\nu''}2\lf[\frac{d(x,y)}{A_0^2\dz^k}\r]^a\r\}
\exp\lf\{-\nu'\lf[\frac{d(y,\CY^k)}{\dz^k}\r]^a\r\}\\
&\quad\times\int_X\frac{1}{V_{\dz^k}(y)}\exp\lf\{-\frac{\nu''}{2}\lf[\frac{d(y,z)}{A_0\dz^k}\r]^a\r\}
\,d\mu(z)\\
&\ls\frac 1{V_{\dz^k}(x)}\exp\lf\{-\frac{\nu''}4\lf[\frac{d(x,y)}{A_0^2\dz^k}\r]^a\r\}
\exp\lf\{-\frac{\nu''}4\lf[\frac{d(x,\CY^k)}{A_0^3\dz^k}\r]^a\r\},
\end{align*}
as desired. By this, the dominated convergence theorem and the cancellation of $Q_k$, we obtain (iv).

By (i) and the proof of Proposition \ref{prop-add} [see also Remark \ref{rem:andef}(iii)], to prove (ii), it
suffices to show (ii) when $x\in X$ and $y,\ y'\in X$ satisfying $d(y,y')\le\dz^k$. Due to the symmetry, we
only consider $|F_k(x,y)-F_k(x,y')|$. From \eqref{eq:QkNregx}, the regularity condition of $Q_k$ with Remark
\ref{rem:andef}(i), Lemma \ref{lem-add}(i) and \eqref{eq-add8}, we deduce that
\begin{align*}
&|F_k(x,y)-F_k(x,y')|\\
&\quad\le\sum_{\az\in\CA_k}\sum_{m=1}^{N(k,\az)}\int_{Q_\az^{k,m}}\lf|Q_k^N(x,z)\r|\,d\mu(z)
\lf|Q_k\lf(y_\az^{k,m},y\r)-Q_k\lf(y_\az^{k,m},y'\r)\r|\\
&\quad\ls\lf[\frac{d(y,y')}{\dz^k}\r]^\eta
\frac 1{V_{\dz^k}(x)}\exp\lf\{-\nu'\lf[\frac{d(y,\CY^k)}{\dz^k}\r]^a\r\}\\
&\qquad\times\sum_{\az\in\CA_k}\sum_{m=1}^{N(k,\az)}\int_{Q_\az^{k,m}}
\exp\lf\{-c\lf[\frac{d(x,z)}{\dz^k}\r]^a\r\}\frac 1{V_{\dz^k}(y)}
\exp\lf\{-\nu'\lf[\frac{d(y,y_\az^{k,m})}{\dz^k}\r]^a\r\}
\,d\mu(z).
\end{align*}
Similarly to the above discussion in (i), by \eqref{eq-add7}, Lemma \ref{lem-add}(ii) and \eqref{eq-add8}, we
have
\begin{align*}
|F_k(x,y)-F_k(x,y')|
&\ls\lf[\frac{d(y,y')}{\dz^k}\r]^\eta\frac 1{V_{\dz^k}(x)}
\exp\lf\{-\frac{\nu''}2\lf[\frac{d(x,y)}{A_0^2\dz^k}\r]^a\r\}
\exp\lf\{-\nu'\lf[\frac{d(y,\CY^k)}{\dz^k}\r]^a\r\}\\
&\quad\times\int_X\frac{1}{V_{\dz^k}(y)}\exp\lf\{-\frac{\nu'}{2}
\lf[\frac{d(z,y)}{A_0\dz^k}\r]^a\r\}\,d\mu(z)\\
&\ls\lf[\frac{d(y,y')}{\dz^k}\r]^\eta\frac 1{V_{\dz^k}(x)}
\exp\lf\{-\frac{\nu''}4\lf[\frac{d(x,y)}{A_0^2\dz^k}\r]^a\r\}
\exp\lf\{-\frac{\nu''}4\lf[\frac{d(x,\CY^k)}{A_0^3\dz^k}\r]^a\r\}.
\end{align*}
This  finishes the proof of (ii).

It remains to prove (iii). Again, by Remark \ref{rem:andef}(iii), we only consider the case when
$x,x',y,y'\in X$ satisfying $d(x,x')\le \dz^k$ and $d(y,y')\le \dz^k$ (see Proposition \ref{prop:etoa}).
By $\eqref{eq:QkNregx}$, the regularity condition of $Q_k$, Remark \ref{rem:andef}(i), \eqref{eq-add7}, Lemma
\ref{lem-add}(ii) and \eqref{eq-add8}, we conclude that
\begin{align*}
&|[F_k(x,y)-F_k(x',y)]-[F_k(x,y')-F_k(x',y')]|\\
&\quad\le\sum_{\az\in\CA_k}\sum_{m=1}^{N(k,\az)}
\int_{Q_\az^{k,m}}\lf|Q_k^N(x,z)-Q_k^N(x',z)\r|\,d\mu(z)
\lf|Q_k\lf(y_\az^{k,m},y\r)-Q_k\lf(y_\az^{k,m},y'\r)\r|\\
&\quad\ls\lf[\frac{d(x,x')}{\dz^k}\r]^\eta\lf[\frac{d(y,y')}{\dz^k}\r]^\eta
\frac 1{V_{\dz^k}(x)}\sum_{\az\in\CA_k}\sum_{m=1}^{N(k,\az)}\int_{Q_\az^{k,m}}
\exp\lf\{-c\lf[\frac{d(x,z)}{\dz^k}\r]^a\r\}\frac 1{V_{\dz^k}(y)}\\
&\qquad\times\exp\lf\{-\nu'\lf[\frac{d(y_\az^{k,m},y)}{\dz^k}\r]^a\r\}
\exp\lf\{-\nu'\lf[\frac{d(y,\CY^k)}{\dz^k}\r]^a\r\}\,d\mu(z)\\
&\quad\ls\lf[\frac{d(x,x')}{\dz^k}\r]^\eta\lf[\frac{d(y,y')}{\dz^k}\r]^\eta
\frac 1{V_{\dz^k}(x)}\exp\lf\{-\frac{\nu''}2\lf[\frac{d(x,y)}{A_0^2\dz^k}\r]^a\r\}
\exp\lf\{-\nu'\lf[\frac{d(y,\CY^k)}{\dz^k}\r]^a\r\}\\
&\qquad\times\int_X\frac{1}{V_{\dz^k}(y)}\exp\lf\{-\frac{\nu'}{2}\lf[\frac{d(z,y)}{A_0\dz^k}\r]^a\r\}
\,d\mu(z)\\
&\quad\ls\lf[\frac{d(x,x')}{\dz^k}\r]^\eta\lf[\frac{d(y,y')}{\dz^k}\r]^\eta
\frac 1{V_{\dz^k}(x)}\exp\lf\{-\frac{\nu''}4\lf[\frac{d(x,y)}{A_0^2\dz^k}\r]^a\r\}
\exp\lf\{-\frac{\nu''}4\lf[\frac{d(x,\CY^k)}{A_0^3\dz^k}\r]^a\r\}.
\end{align*}
This finishes the proof of (iii) and hence of Lemma \ref{lem-add4}.
\end{proof}

\begin{lemma}\label{lem-add5}
Let $\{Q_k\}_{k=-\fz}^\fz$ be an $\exp$-{\rm ATI}. Assume that $f\in\CG(x_0,\dz^k,\bz,\gz)$ for some $k\in\zz$ and
$\bz,\ \gz\in(0,\eta]$.
Then there exists a positive constant $C$, independent of $k$, $x_0$ and $r$, such that, for any $y,\ u\in X$,
\begin{align}\label{eq-xxx0}
|Q_kf(y)|&\le C\|f\|_{\CG(x_0,\dz^k,\bz,\gz)} \frac 1{V_{\dz^k}(x_0)+V(x_0,u)}
\lf[\frac {\dz^k}{\dz^k+d(x_0,u)}\r]^{\gz}
\end{align}
and
\begin{align}\label{eq-xxx1}
|Q_kf(y)|&\le C\|f\|_{\CG(x_0,\dz^k,\bz,\gz)} \lf[1+\frac{d(u,y)}{\dz^k}\r]^{\bz+\omega+\gz}
\lf[\frac{\dz^k}{\dz^k+d(u, x_0)}\r]^{\bz}\\
&\quad\times\frac 1{V_{\dz^k}(x_0)+V(x_0,u)}\lf[\frac {\dz^k}{\dz^k+d(x_0,u)}\r]^{\gz}.\noz
\end{align}
\end{lemma}

\begin{proof}
Without loss of generality, we may assume that $\|f\|_{\CG(x_0,\dz^k,\bz,\gz)}=1$.
One may observe that the proof of \eqref{eq-x4} in Lemma \ref{lem-add3} also implies \eqref{eq-xxx0} with
the decaying index $\eta$ therein replaced by $\gz$.

Now we show \eqref{eq-xxx1}. For any $y,\ u\in X$, due to the cancellation condition of $Q_k$, we write
\begin{align*}
\lf|Q_kf\lf(y\r)\r|&
=\lf|\int_XQ_k\lf(y,z\r)[f(z)-f(u)]\,d\mu(y)\r|\\
&\le\int_{d(u, z)\le(2A_0)^{-1}[\dz^k+d(x_0,u)]}|Q_k(y,z)||f(z)-f(u)|\,d\mu(z)\noz\\
&\quad+\int_{d(u, z)>(2A_0)^{-1}[\dz^k+d(x_0,u)]}|Q_k(y,z)||f(z)|\,d\mu(z)\noz\\
&\quad+|f(u)|\int_{d(u, z)>(2A_0)^{-1}[\dz^k+d(x_0,u)]}|Q_k(y,z)|\,d\mu(z)\noz
=:\mathrm{J}_1+\mathrm{J}_2+\mathrm{J}_3.\noz
\end{align*}
By the size condition of $Q_k$, Remark \ref{rem:andef}(i), the regularity condition of $f$, \eqref{eq-xx} and
Lemma \ref{lem-add}(ii), we conclude that
\begin{align*}
\mathrm{J}_1
&\ls \int_{d(u, z)\le(2A_0)^{-1}[\dz^k+d(x_0,u)]}\frac1{V_{\dz^k}(z)}
\exp\lf\{-\nu'\lf[\frac{d(y,z)}{\dz^k}\r]^a\r\}\lf[\frac{d(u, z)}{\dz^k+d(x_0,u)}\r]^{\bz}\\
&\quad\times\frac 1{V_{\dz^k}(x_0)+V(x_0,u)}\lf[\frac {\dz^k}{\dz^k+d(x_0,u)}\r]^{\gz}\,d\mu(z)\\
&\ls\lf[1+\frac{d(u,  y)}{\dz^k}\r]^{\bz}\lf[\frac{\dz^k}{\dz^k+d(u, x_0)}\r]^{\bz}\frac 1{V_{\dz^k}(x_0)+V(x_0,u)}\lf[\frac {\dz^k}{\dz^k+d(x_0,u)}\r]^{\gz},
\end{align*}
where, in the last step, we used the inequality
\begin{align}\label{eq-xx5}
\frac{d(u,  z)}{\dz^k}\le \frac{A_0^2[d(u,  y)+d(y,z)]}{\dz^k}
\ls\lf[1+\frac{d(u,  y)}{\dz^k}\r]\lf[1+\frac{d(y,z)}{\dz^k}\r]
\end{align}
and the fact that the last term $ 1+\frac{d(y,z)}{\dz^k} $ is absorbed by the factor
$\exp\{-\nu'[\frac{d(y,z)}{\dz^k}]^a\}$.

By the size condition of $Q_k$, Remark \ref{rem:andef}(i), the size condition of $f$, the fact that
$[\frac{d(u, z)}{\dz^k+d(x_0,u)}]^{\bz}\ge (2A_0)^{-\bz}$  and \eqref{eq-xx5}, we have
\begin{align*}
\mathrm{J}_2&\ls\int_{d(u, z)>(2A_0)^{-1}[\dz^k+d(x_0,u)]}\frac1{V_{\dz^k}(z)}
\exp\lf\{-\nu'\lf[\frac{d(y,z)}{\dz^k}\r]^a\r\}\lf[\frac{d(u, z)}{\dz^k+d(x_0,u)}\r]^{\bz}\\
&\quad\times\frac 1{V_{\dz^k}(x_0)+V(x_0,z)}\lf[\frac {\dz^k}{\dz^k+d(x_0,z)}\r]^{\gz}\,d\mu(z)\\
&\ls \lf[1+\frac{d(u,y)}{\dz^k}\r]^{\bz}\lf[\frac{\dz^k}{\dz^k+d(u, x_0)}\r]^{\bz}
\int_X\lf[1+\frac{d(y,z)}{\dz^k}\r]^\bz\frac1{V_{\dz^k}(z)}\\
&\quad\times\exp\lf\{-\nu'\lf[\frac{d(y,z)}{\dz^k}\r]^a\r\}\frac 1{V_{\dz^k}(x_0)+V(x_0,z)}
\lf[\frac {\dz^k}{\dz^k+d(x_0,z)}\r]^{\gz}\,d\mu(z)\\
&\ls \lf[1+\frac{d(u,y)}{\dz^k}\r]^{\bz}\lf[\frac{\dz^k}{\dz^k+d(u, x_0)}\r]^{\bz}
\int_X\frac1{V_{\dz^k}(z)}\exp\lf\{-\frac{\nu'}2\lf[\frac{d(y,z)}{\dz^k}\r]^a\r\}\\
&\quad\times\frac 1{V_{\dz^k}(x_0)+V(x_0,z)}
\lf[\frac {\dz^k}{\dz^k+d(x_0,z)}\r]^{\gz}\,d\mu(z).
\end{align*}
For the last integral displayed above, we separate $X$ into
$\{z\in X:\ d(y,z)\ge d(x_0,y)/(2A_0)\}$ and $\{z\in X:\ d(x_0,z)\ge d(x_0,y)/(2A_0)\}$. Then, by Lemma
\ref{lem-add}(ii), we conclude that
\begin{align*}
&\int_X\frac1{V_{\dz^k}(z)}\exp\lf\{-\frac{\nu'}2\lf[\frac{d(y,z)}{\dz^k}\r]^a\r\}
\frac 1{V_{\dz^k}(x_0)+V(x_0,z)}\lf[\frac {\dz^k}{\dz^k+d(x_0,z)}\r]^{\gz}\,d\mu(z)\\
&\quad\ls\frac1{V_{\dz^k}(x_0)}\exp\lf\{-\frac{\nu'}4\lf[\frac{d(x_0,y)}{2A_0\dz^k}\r]^a\r\}
\int_{d(y,z)\ge (2A_0)^{-1}d(x_0,y)}\frac 1{V_{\dz^k}(x_0)+V(x_0,z)}
\lf[\frac {\dz^k}{\dz^k+d(x_0,z)}\r]^{\gz}\,d\mu(z)\\
&\qquad+\frac 1{V_{\dz^k}(x_0)+V(x_0,z)}\lf[\frac {\dz^k}{\dz^k+d(x_0,z)}\r]^{\gz}
\int_{d(x_0,z)\ge (2A_0)^{-1}d(x_0,y)}
\frac1{V_{\dz^k}(z)}\exp\lf\{-\frac{\nu'}2\lf[\frac{d(y,z)}{\dz^k}\r]^a\r\}\,d\mu(z)\\
&\quad\ls\frac 1{V_{\dz^k}(x_0)+V(x_0,z)}\lf[\frac {\dz^k}{\dz^k+d(x_0,z)}\r]^{\gz}.
\end{align*}
From this, the doubling condition \eqref{eq:doub} and \eqref{eq-xx5}, it follows that
\begin{align*}
\mathrm{J}_2&\ls \lf[1+\frac{d(u,  y)}{\dz^k}\r]^{\bz}\lf[\frac{\dz^k}{\dz^k+d(u, x_0)}\r]^{\bz}
\frac1{V_{\dz^k}(x_0)+V(y, x_0)} \lf[\frac{\dz^k}{\dz^k+d(y, x_0)}\r]^{\gz}\\
&\ls \lf[1+\frac{d(u,  y)}{\dz^k}\r]^{\bz+\omega+\gz}\lf[\frac{\dz^k}{\dz^k+d(u, x_0)}\r]^{\bz}
\frac 1{V_{\dz^k}(x_0)+V(x_0,u)}\lf[\frac {\dz^k}{\dz^k+d(x_0,u)}\r]^{\gz}.
\end{align*}
Finally, for the term $\RJ_3$, by the size conditions of $f$ and $Q_k$, Remark \ref{rem:andef}(i),
\eqref{eq-xx5} and Lemma \ref{lem-add}(ii), we obtain
\begin{align*}
\RJ_3&\ls\frac 1{V_{\dz^k}(x_0)+V(x_0,u)}\lf[\frac {\dz^k}{\dz^k+d(x_0,u)}\r]^{\gz}
\int_{d(u,z)>(2A_0)^{-1}[\dz^k+d(x_0,u)]}\lf[1+\frac{d(u,z)}{\dz^k}\r]^\bz
\frac 1{V_{\dz^k}(z)}\\
&\quad\times\lf[1+\frac{d(u,z)}{\dz^k}\r]^\bz\exp\lf\{-\nu'\lf[\frac{d(y,z)}{\dz^k}\r]^a\r\}\,d\mu(z)\\
&\ls\lf[1+\frac{d(u,y)}{\dz^k}\r]^\bz\frac 1{V_{\dz^k}(x_0)+V(x_0,u)}\lf[\frac {\dz^k}{\dz^k+d(x_0,u)}\r]^{\gz}
\int_{d(u,z)>(2A_0)^{-1}[\dz^k+d(x_0,u)]}\frac 1{V_{\dz^k}(z)}\\
&\quad\times \exp\lf\{-\frac{\nu'}2\lf[\frac{d(y,z)}{\dz^k}\r]^a\r\}\,d\mu(z)\\
&\ls\lf[1+\frac{d(u,y)}{\dz^k}\r]^\bz\frac 1{V_{\dz^k}(x_0)+V(x_0,u)}
\lf[\frac {\dz^k}{\dz^k+d(x_0,u)}\r]^{\gz}.
\end{align*}

Combining the estimates of $\mathrm{J}_1$ through $\mathrm{J}_3$, we obtain \eqref{eq-xxx1},
which completes the proof of Lemma \ref{lem-add5}.
\end{proof}

\begin{theorem}\label{thm:hdrf}
Let $\{Q_k\}_{k=-\fz}^\fz$ be an $\exp$-{\rm ATI} and $\bz,\ \gz\in(0,\eta)$.
For any $k\in\zz$, $\az\in\CA_k$ and  $m\in\{1,\ldots,N(k,\az)\}$, suppose that $y_\az^{k,m}$ is an arbitrary point in $Q_\az^{k,m}$. Then, for any $i\in\{0,1,2\}$, there
exists a  sequence $\{\wz{Q}_k^{(i)}\}_{k=-\fz}^\fz$ of bounded linear operators on $L^2(X)$ such that, for any
$f$ in $\GOO{\bz,\gz}$ [resp., $L^p(X)$ with $p\in(1,\fz)$],
\begin{align}\label{eq:hdrf}
f(\cdot)&=\sum_{k=-\fz}^\fz\sum_{\az\in\CA_k}\sum_{m=1}^{N(k,\az)}\int_{Q_\az^{k,m}}\wz{Q}_k^{(0)}
(\cdot,y)\,d\mu(y)Q_kf\lf(y_\az^{k,m}\r)\\
&=\sum_{k=-\fz}^\fz\sum_{\az\in\CA_k}\sum_{m=1}^{N(k,\az)}\wz{Q}^{(1)}_k\lf(\cdot,y_\az^{k,m}\r)
\int_{Q_\az^{k,m}}Q_kf(y)\,d\mu(y)\noz\\
&=\sum_{k=-\fz}^\fz\sum_{\az\in\CA_k}\sum_{m=1}^{N(k,\az)}\mu\lf(Q_\az^{k,m}\r)
\wz{Q}^{(2)}_k\lf(\cdot,y_\az^{k,m}\r)Q_kf\lf(y_\az^{k,m}\r),\noz
\end{align}
where all the summations converge in the sense of $\GOO{\bz,\gz}$ [resp., $L^p(X)$ with $p\in(1,\fz)$].
Moreover, the kernels of $\wz{Q}_k^{(0)}$, $\wz{Q}^{(1)}_k$ and $\wz{Q}^{(2)}_k$ satisfy
the size condition \eqref{eq:atisize}, the regularity condition \eqref{eq:atisregx}
only for the first variable, and also the following cancellation condition: for any $x\in X$,
\begin{align}\label{eq:x00}
\int_X \wz{Q}_k^{(i)}(x,y)\,d\mu(y)=0=\int_X\wz{Q}_k^{(i)}(y,x)\,d\mu(y), \qquad \forall\,i\in\{0,1,2\}.
\end{align}
\end{theorem}

\begin{proof}
We only prove the first equality in \eqref{eq:hdrf}.
Indeed, to obtain the second and the third equalities in \eqref{eq:hdrf}, instead of $\CS_N$ in
\eqref{eq:S}, we only need to consider
$$
\CS_N^{(1)}f(x):=\sum_{k=-\fz}^\fz\sum_{\az\in\CA_k}\sum_{m=1}^{N(k,\az)}Q_k^N\lf(x,y_\az^{k,m}\r)
\int_{Q_\az^{k,m}}Q_kf(y)\,d\mu(y),\qquad \forall\,x\in X,
$$
respectively,
$$
\CS_N^{(2)}f(x):=\sum_{k=-\fz}^\fz\sum_{\az\in\CA_k}\sum_{m=1}^{N(k,\az)}\mu\lf(Q_\az^{k,m}\r)
Q_k^N\lf(x,y_\az^{k,m}\r)Q_kf\lf(y_\az^{k,m}\r),\qquad \forall\,x\in X.
$$
One finds that the corresponding remainders $\CR_N^{(1)}:= I-\CS_N^{(1)}$ and $\CR_N^{(2)}:=I-\CS_N^{(2)}$
satisfy the same estimate as $\CR_N$ in Proposition \ref{prop:GG}. The remaining arguments are similar,
the details being omitted.

Now we prove the first equality in \eqref{eq:hdrf}. Due to Remark \ref{rem:j0N}, the operator $\CS_N^{-1}=(I-\CR_N)^{-1}$ satisfies
$$
\lf\|\CS_N^{-1}\r\|_{L^2(X)\to L^2(X)}\le 2 \quad\text{and}\quad
\lf\|\CS_N^{-1}\r\|_{\GOX{x_1,r,\bz,\gz}\to\GOX{x_1,r,\bz,\gz}}\le 2.
$$
for any $x_1\in X$ and $r\in(0,\fz)$. For any $k\in\zz$ and $x,\ y\in X$, define
$$
\wz{Q}_k^{(0)}(x,y):=\wz{Q}_k(x,y):=\CS_N^{-1}\lf(Q_k^N(\cdot,y)\r)(x).
$$
By Lemma \ref{lem:propQkN} and the proof of Proposition \ref{prop:etoa}, we find that $Q_k^N(\cdot, y)\in
\mathring\CG(y, \dz^k, \bz,\gz)$ with $\|\cdot\|_{\CG(y, \dz^k, \bz,\gz)}$-norm independent of $k$ and $y$,
which implies that
$\{\wz{Q}_k\}_{k=-\fz}^\fz$ satisfies the size condition \eqref{eq:atisize} and the regularity condition \eqref{eq:atisregx}
only for the first variable. The proof of \eqref{eq:x00} is similar to that of \eqref{eq:wzQcan}, with $R_N$
therein replaced by $\CR_N$, the details being omitted. Moreover, for any $f\in L^2(X)$, \eqref{eq:hdrf}
converges in $L^2(X)$. We divide the remaining arguments into three steps.

{\it Step 1) Proof of the convergence of \eqref{eq:hdrf} in $\GO{\bz,\gz}$ when $f\in\GO{\bz',\gz'}$
with $\bz'\in(\bz,\eta)$ and $\gz'\in(\gz,\eta)$.}

Without loss of generality, we may assume that $\|f\|_{\GO{\bz',\gz'}}=1$.
For any $k\in\zz$ and $M\in\nn$, define
$$\CA_{k,M}:=\{\az\in\CA_k:\ d(x_0,z_\az^k)\le M\}\quad \textup{and}\quad \CA_{k,M}^\complement:=\CA_k\setminus\CA_{k,M}=\{\az\in\CA_k:\ d(x_0,z_\az^k)>M\}.$$
To obtain the convergence of \eqref{eq:hdrf} in $\GO{\bz,\gz}$, it suffices to show that
\begin{equation}\label{eq:limb}
\lim_{L\to\fz}\lim_{M\to\fz}
\lf\|f-\sum_{|k|\le L}\sum_{\az\in\CA_{k,M}}\sum_{m=1}^{N(k,\az)}\int_{Q_\az^{k,m}}\wz{Q}_k(\cdot,y)\,d\mu(y)
Q_kf\lf(y_\az^{k,m}\r)\r\|_{\GO{\bz,\gz}}=0.
\end{equation}
Writing $f=\CS_N^{-1}\CS_N f$ and noticing that $\CS_N^{-1}$ is bounded on $\GO{\bz,\gz}$, we only need to prove
\begin{equation}\label{eq:limb1}
\lim_{L\to\fz} \lf\|\sum_{|k|\ge L+1}\sum_{\az\in\CA_{k}}
\sum_{m=1}^{N(k,\az)}\int_{Q_\az^{k,m}}Q_k^N(\cdot,y)\,d\mu(y)Q_kf\lf(y_\az^{k,m}\r)\r\|_{\GO{\bz,\gz}}=0
\end{equation}
and
\begin{equation}\label{eq:limb2}
\lim_{L\to\fz}\lim_{M\to\fz}
\lf\|\sum_{|k|\le L}\sum_{\az\in\CA_{k,M}^\complement}
\sum_{m=1}^{N(k,\az)}\int_{Q_\az^{k,m}}Q_k^N(\cdot,y)\,d\mu(y)Q_kf\lf(y_\az^{k,m}\r)\r\|_{\GO{\bz,\gz}}=0.
\end{equation}

Invoking the definition of the kernel $F_k$ in Lemma \ref{lem-add4}, we find that, for any $f\in\GO{\bz,\gz}$,
$$
F_kf(x)=\int_X F_k(x,y)f(y)\,d\mu(y),\qquad \forall\,x\in X
$$
and hence that \eqref{eq:limb1} is equivalent to that
\begin{equation}\label{eq:limb3}
\lim_{L\to\fz} \lf\|\sum_{|k|\ge L+1} F_kf\r\|_{\GO{\bz,\gz}}=0.
\end{equation}
Comparing Lemma \ref{lem-add4} with Lemma \ref{lem-add3}, we find that $F_k$ satisfies the same estimate as
$E_k=Q_k^NQ_k$ defined in Lemma \ref{lem-add3}. Therefore, repeating the estimations of \eqref{eq:sumsize} and
\eqref{eq:sumreg}, with $Q_k^NQ_k$ replaced by $F_k$, and using the Fubini theorem,
we find that there exists a positive constant
$\sigma\in(0,\fz)$ such that, for any $k\in\zz$ and $f\in\GOO{\bz,\gz}$,
$\|F_kf\|_{\GO{\bz,\gz}}\ls\dz^{|k|\sigma}\|f\|_{\GO{\bz',\gz'}}$, which implies \eqref{eq:limb1}.

Since the summation in $k$ in \eqref{eq:limb2} has only finite terms,
the proof of  \eqref{eq:limb2} can be reduced  to proving that,
 for any fixed $k\in\zz$,
\begin{equation}\label{eq:limb2aim}
\lim_{M\to\fz}\lf\|\sum_{\az\in\CA_{k,M}^\complement}
\sum_{m=1}^{N(k,\az)}\int_{Q_\az^{k,m}}Q_k^N(\cdot,y)\,d\mu(y)Q_kf\lf(y_\az^{k,m}\r)\r\|_{\GO{\bz,\gz}}=0.
\end{equation}
Noticing that $\GO{\bz,\gz}=\mathring\CG(x_0,\dz^k,\bz,\gz)$, we may as well consider the $\|\cdot\|_{\mathring\CG(x_0,\dz^k,\bz,\gz)}$-norm in \eqref{eq:limb2aim}. To simplify the notation, we let
$$
H_M(x):= \sum_{\az\in\CA_{k,M}^\complement}
\sum_{m=1}^{N(k,\az)}\int_{Q_\az^{k,m}}Q_k^N(x,y)\,d\mu(y)Q_kf\lf(y_\az^{k,m}\r),\quad\forall\, x\in X.
$$

Choose $M\in\nn$ large enough such that $M\ge 2A_0C^\natural\dz^{-L}$. Then, when $|k|\le L$
and $\az\in \CA_{k,M}^\complement$, we know that, for any $y\in Q_\az^{k,m}$
\begin{equation}\label{star2}
d(y,x_0)\ge A_0^{-1}d\lf(z_\az^{k},x_0\r)-d\lf(z_\az^{k},y\r)\ge A_0^{-1}M-\delta^k\ge M/(2A_0).
\end{equation}
Based on $f\in \GO{\bz',\gz'}=\mathring\CG(x_0,\dz^k,\bz',\gz')$ and \eqref{eq-xxx0}, we know that, for any $y\in Q_\az^{k,m}$,
\begin{align*}
\lf|Q_kf\lf(y_\az^{k,m}\r)\r|
&\ls \frac1{V_{\dz^k}(x_0)+V(y_\az^{k,m}, x_0)} \lf[\frac{\dz^k}{\dz^k+d(y_\az^{k,m}, x_0)}\r]^{\gz'}
\sim \frac1{V_{\dz^k}(x_0)+V(y, x_0)} \lf[\frac{\dz^k}{\dz^k+d(y, x_0)}\r]^{\gz'}.
\end{align*}
From this, \eqref{eq:QkNsize}, and \eqref{star2}, it follows that, for any $x\in X$,
\begin{align*}
|H_M(x)| \ls \int_{d(y, x_0)>M/(2A_0)} \frac1{V_{\dz^k}(x)}
\exp\lf\{-c\lf[\frac{d(x,y)}{\dz^k}\r]^a\r\}\frac1{V_{\dz^k}(x_0)+V(y, x_0)} \lf[\frac{\dz^k}{\dz^k+d(y, x_0)}\r]^{\gz'}\,d\mu(y).
\end{align*}
For any $y\in X$, the quasi-triangle inequality of $d$ implies that either
$d(y, x)\ge d(x,x_0)/(2A_0)$ or  $d(y, x_0)\ge d(x,x_0)/(2A_0)$.
With this and \eqref{eq-xxx}, the last integral can be further controlled by
\begin{align*}
&\frac1{V_{\dz^k}(x_0)+V(x, x_0)} \lf[\frac{\dz^k}{\dz^k+d(x, x_0)}\r]^{\gz}
\lf\{\int_{\gfz{d(y, x_0)>M/(2A_0)}{d(y, x)\ge d(x,x_0)/(2A_0)}} \frac1{V_{\dz^k}(x_0)+V(y, x_0)} \lf[\frac{\dz^k}{\dz^k+d(y, x_0)}\r]^{\gz'}\,d\mu(y)\r.\\
&\quad+\lf.\int_{\gfz{d(y, x_0)>M/(2A_0)}{d(y, x_0)\ge d(x,x_0)/(2A_0)}} \frac1{V_{\dz^k}(x)}
\exp\lf\{-c\lf[\frac{d(x,y)}{\dz^k}\r]^a\r\}\lf[\frac{\dz^k}{\dz^k+M}\r]^{\gz'-\gz}\,d\mu(y)\r\}.
\end{align*}
This, together with Lemma \ref{lem-add}(ii), implies that, for any $x\in X$,
\begin{align}\label{eq-xx1}
|H_M(x)|
\ls M^{\gz-\gz'}\frac1{V_{\dz^k}(x_0)+V(x, x_0)} \lf[\frac{\dz^k}{\dz^k+(x, x_0)}\r]^{\gz}.
\end{align}

Assume for the moment that, when $d(x,x')\le (2A_0)^{-1}[\dz^k+d(x,x_0)]$,
\begin{align}\label{eq-xx2}
|H_M(x)-H_M(x')|\ls \lf[\frac{d(x,x')}{\dz^k+d(x,x_0)}\r]^{\bz'}\frac1{V_{\dz^k}(x_0)+V(x, x_0)}
 \lf[\frac{\dz^k}{\dz^k+(x, x_0)}\r]^{\gz}.
\end{align}
Meanwhile, when $d(x,x')\le (2A_0)^{-1}[\dz^k+d(x,x_0)]$,  we have $V_{\dz^k}(x_0)+V(x, x_0)\sim V_{\dz^k}(x_0)+V(x', x_0)$
and $\dz^k+(x, x_0)\sim \dz^k+(x', x_0)$, which combined with \eqref{eq-xx1}, gives
\begin{align*}
|H_M(x)-H_M(x')|\ls \frac{M^{\gz-\gz'} }{V_{\dz^k}(x_0)+V(x, x_0)} \lf[\frac{\dz^k}{\dz^k+(x, x_0)}\r]^{\gz}.\noz
\end{align*}
Then, taking the geometry means between the above two formulae, we obtain
\begin{align}\label{eq-xx3}
|H_M(x)-H_M(x')|\ls M^{(\gz-\gz')(1-\bz/\bz')} \lf[\frac{d(x,x')}{\dz^k+d(x,x_0)}\r]^{\bz}\frac1{V_{\dz^k}(x_0)+V(x, x_0)} \lf[\frac{\dz^k}{\dz^k+(x, x_0)}\r]^{\gz}.
\end{align}
From \eqref{eq-xx1} and \eqref{eq-xx3}, it follows directly \eqref{eq:limb2aim}.

Now we show that \eqref{eq-xx2} holds true when $d(x,x')\le (2A_0)^{-1}[\dz^k+d(x,x_0)]$. From
\eqref{eq:QkNregx} when $d(x,x')\le \dz^k$ and \eqref{eq:QkNsize} when $d(x,x')>\dz^k$, we deduce that
\begin{align*}
\lf|Q_k^N(x,y)-Q_k^N(x',y)\r|
&\ls\min\lf\{1,\lf[\frac{d(x,x')}{\dz^k}\r]^\eta\r\} \lf[ \frac1{V_{\dz^k}(x)}
\exp\lf\{-c\lf[\frac{d(x,y)}{\dz^k}\r]^a\r\} \r.\\
&\quad\lf.+\frac1{V_{\dz^k}(x')}
\exp\lf\{-c\lf[\frac{d(x',y)}{\dz^k}\r]^a\r\}\r].
\end{align*}
Notice  that the condition $d(x,x')\le (2A_0)^{-1}[\dz^k+d(x,x_0)]$ implies that $V_{\dz^k}(x_0)+V(x_0,x)
\sim V_{\dz^k}(x_0)+V(x_0,x')$ and $\dz^k+d(x_0,x)\sim \dz^k+d(x_0,x')$.
With this and $f\in \GO{\bz',\gz'}=\mathring\CG(x_0,\dz^k,\bz',\gz')$ , we apply \eqref{eq-xxx1}
(with $u=x$ or $u=x'$ therein) to deduce that
\begin{align*}
\lf|Q_kf\lf(y_\az^{k,m}\r)\r|
&\ls  \lf[\frac{\dz^k}{\dz^k+d(x, x_0)}\r]^{\bz'}
\frac 1{V_{\dz^k}(x_0)+V(x_0,u)}\lf[\frac {\dz^k}{\dz^k+d(x_0,x)}\r]^{\gz'}\\
&\quad\times
\min\lf\{ \lf[1+\frac{d(x,  y_\az^{k,m})}{\dz^k}\r]^{\bz+\omega+\gz},
\lf[1+\frac{d(x',  y_\az^{k,m})}{\dz^k}\r]^{\bz+\omega+\gz}\r\}.
\end{align*}
Also, due to \eqref{eq-xx}, the variable $y_\az^{k,m}$ in the right-hand side of the above estimate of $|Q_kf(y_\az^{k,m})|$ can be replaced by any
point  $y\in Q_\az^{k,m}$. Therefore, by Lemma \ref{lem-add}(ii), we conclude that
\begin{align*}
|H_M(x)-H_M(x')|
&\ls \lf[\frac{d(x,x')}{\dz^k+d(x, x_0)}\r]^{\bz'}
\frac 1{V_{\dz^k}(x_0)+V(x_0,u)}\lf[\frac {\dz^k}{\dz^k+d(x_0,x)}\r]^{\gz'} \\
&\quad\times
\int_X
\lf(\frac1{V_{\dz^k}(x)}\exp\lf\{-c\lf[\frac{d(x,y)}{\dz^k}\r]^a\r\} +\frac1{V_{\dz^k}(x')}
\exp\lf\{-c\lf[\frac{d(x',y)}{\dz^k}\r]^a\r\}\r)\\
&\quad\times\min\lf\{ \lf[1+\frac{d(x,  y)}{\dz^k}\r]^{\bz+\omega+\gz},
\lf[1+\frac{d(x',y)}{\dz^k}\r]^{\bz+\omega+\gz}\r\}\,d\mu(y)\\
&\ls \lf[\frac{d(x,x')}{\dz^k+d(x, x_0)}\r]^{\bz'}
\frac 1{V_{\dz^k}(x_0)+V(x_0,u)}\lf[\frac {\dz^k}{\dz^k+d(x_0,x)}\r]^{\gz'}.
\end{align*}
This proves  \eqref{eq-xx2}, and hence finishes the proof of Step 1).

{\it Step 2) Proof of the convergence of \eqref{eq:hdrf} in $\GOO{\bz,\gz}$ when $f\in\GOO{\bz,\gz}$.}

We first claim that $\CS_N^{-1}=(I-\CR_N)^{-1}$ maps $\GOO{\bz,\gz}$ continuously into $\GOO{\bz,\gz}$.
Indeed, recalling that Remark \eqref{rem:j0N} says that $\|\CR_N\|_{\GO{x_1,r,\bz,\gz}\to\GO{x_1,r,\bz,\gz}}\le \frac 12$, it suffices to show that $\CR_N$ maps $\GOO{\bz,\gz}$  into $\GOO{\bz,\gz}$.
With $\CR_N=I-\CS_N$, we only need to show that $\CS_N h\in\GOO{\bz,\gz}$ whenever $h\in \GOO{\bz,\gz}$.

Indeed, for any $h\in\mathring{\CG}^\eta_0(\bz,\gz)$, there exists
$\{h_j\}_{j=1}^\infty\subset\mathring{\CG}(\eta,\eta)$ such that $\|h-h_j\|_{\mathring{\CG}(\bz,\gz)}\to 0$ as
$j\to\infty$. Notice that, for ant $j\in\nn$,
$$
\CS_Nh_j(\cdot)=\sum_{k=-\fz}^\fz\sum_{\az\in\CA_{k}}
\sum_{m=1}^{N(k,\az)}\int_{Q_\az^{k,m}}Q_k^N(\cdot,y)\,d\mu(y)Q_kh_j\lf(y_\az^{k,m}\r),
$$
where the series converges in $\mathring\CG(\bz,\gz)$ due to \eqref{eq:limb1} and \eqref{eq:limb2}.
Since every $\int_{Q_\az^{k,m}}Q_k^N(\cdot,y)\,d\mu(y)\in \mathring\CG(\eta,\eta)$, then, from
\eqref{eq:limb2aim}, it follows that, for any $k\in\zz$, $F_kh_j\in\GOO{\bz,\gz}$. This, together with the
definition of $F_k$ and \eqref{eq:limb1}, further implies that $\CS_Nh_j\in\GOO{\bz,\gz}$.
Moreover, by the fact that
$$
\|\CS_Nh-\CS_Nh_j\|_{\GO{\bz,\gz}}=\|(h-h_j)-\CR_N(h-h_j)\|_{\GO{\bz,\gz}}\le2\|h-h_j\|_{\GO{\bz,\gz}}\to 0
$$
as $j\to\infty$, we obtain $\CS_Nh\in\mathring{\CG}^\eta_0(\bz,\gz)$. This finishes the proof of the claim.
Moreover, repeating the proof of Step 2) in the proof of Theorem \ref{thm:hcrf} with $T_N$ and $R_N$
replaced, respectively, by $\CS_N$ and $\CR_N$, we find that both $\CS_N$ and $\CS_N^{-1}$ are bounded on
$\GOO{\bz,\gz}$.

Next, we use the above claim to conclude the proof of Step 2). For any $f\in\GOO{\bz,\gz}$, there exists
$\{h_j\}_{j=1}^\fz\subset\GO{\eta,\eta}$ such that $\|f-h_j\|_{\GO{\bz,\gz}}\to 0$ as $j\to\fz$. For any
$k\in\zz$, $L,\ M\in\nn$ and $x,\ y\in X$, define
$$
F_{k,M}(x,y):=\sum_{\az\in\CA_{k,M}}\sum_{m=1}^{N(k,\az)}\int_{Q_\az^{k,m}}Q_k^N(x,z)\,d\mu(z)
Q_k\lf(y_\az^{k,m},y\r)
$$
and $F_M^{(L)}:=\sum_{|k|\le L} F_{k,M}$. Repeating the proof of Lemma \ref{lem-add4} with $F_k$ replaced by
$F_{k,M}$ and the sum $\sum_{\az\in\CA_k}$ by $\sum_{\az\in\CA_{k,M}}$, we find that $F_{k,M}$
satisfies (i) through (iv) of Lemma \ref{lem-add4}, with the implicit positive constants independent of $M$.
This, combined with Lemma \ref{lem:GkN}, implies that $F_{k,M}$ satisfies the same estimates as $G_{k,N}$ with
the implicit positive constants independent of $M$ and the factor $\dz^{j_0\eta}$ removed. Therefore, following
the proof of Proposition \ref{prop:GNM}, with $G_N^{(M)}$ replaced by $F_M^{(L)}$, we further conclude that
$F_M^{(L)}$ is bounded on $\GO{\bz,\gz}$ with its operator norm independent of $M$ and $L$. By this,
\eqref{eq:limb} and the boundedness of $\CS_N$ on $\GOO{\bz,\gz}$, we conclude that
\begin{align*}
\lf\|\CS_Nf-F_M^{(L)}f\r\|_{\GOO{\bz,\gz}}&\le\lf\|\CS_N(f-h_j)\r\|_{\GO{\bz,\gz}}+
\lf\|h_j-F_M^{(L)}h_j\r\|_{\GO{\bz,\gz}}+\lf\|F_M^{(L)}(h_j-f)\r\|_{\GO{\bz,\gz}}\\
&\ls \lf\|f-h_j\r\|_{\GO{\bz,\gz}}+\lf\|h_j-F_M^{(L)}h_j\r\|_{\GO{\bz,\gz}}
\to \lf\|f-h_j\r\|_{\GO{\bz,\gz}}
\end{align*}
as $L,\ M\to\fz$. If we let $j\to\fz$ and use the boundedness of $\CS_N^{-1}$ on $\GOO{\bz,\fz}$, then we know
that \eqref{eq:hdrf} converges in $\GOO{\bz,\gz}$. This finishes the proof of Step 2).

{\it Step 3) Proof of the convergence of \eqref{eq:hdrf} in  $L^p(X)$ when $f\in L^p(X)$ with any given
$p\in(1,\fz)$.}

For any $k\in\zz$ and $M,\ L\in\nn$, recall that
$$
F_{k,M}(x,y):=\sum_{\az\in\CA_{k,M}}\sum_{m=1}^{N(k,\az)}\int_{Q_\az^{k,m}}Q_k^N(x,z)\,d\mu(z)
Q_k\lf(y_\az^{k,m},y\r),\qquad\forall\, x,\ y\in X
$$
and $F_M^{(L)}=\sum_{|k|\le L} F_{k,M}$.
Notice that $F_{k,M}$ satisfies all conditions in Lemma \ref{lem-add4} with the implicit positive constants
independent of $M$ [this has been proved in Step 2)]. Following the proof of \eqref{eq:GL2} with $G_{k,N}$
replaced by $F_{k,M}$ and $G_N^{(M)}$ by $F_M^{(L)}$,
we know that the operator $F_{M}^{(L)}$ is bounded on $L^2(X)$, so is the action of $\CS_N^{-1}$ on it.

Let us write
$$
\wz F_{k, M}f:= \CS_N^{-1}F_{k,M}f=\sum_{|k|\le L}\sum_{\az\in\CA_{k,M}}\sum_{m=1}^{N(k,\az)}\int_{Q_\az^{k,m}}\wz{Q}_k(\cdot,y)\,d\mu(y)
Q_kf\lf(y_\az^{k,m}\r).
$$
Consequently, $ \sum_{|k|\le L} \wz F_{k, M}$ is bounded on $L^2(X)$, with its operator norm independent
of $M$ and $L$. Notice that each $\wz F_{k, M}$ is associated to an integral kernel
$$
\wz F_{k, M}(x,y)=\sum_{\az\in\CA_{k,M}}\sum_{m=1}^{N(k,\az)}\int_{Q_\az^{k,m}}\wz{Q}_k(x,z)\,d\mu(z)
Q_k\lf(y_\az^{k,m},y\r),\qquad\forall\, x,\ y\in X.
$$
Applying the size and the regularity conditions of $\wz Q_k$, and following the estimations of
\eqref{eq:wTLsize}, \eqref{eq:wTLregx} and \eqref{eq:wTLregy} in Step 3) of the proof of Theorem \ref{thm:hcrf},
we easily conclude that
\begin{align*}
\sum_{k\le |L|}\lf|\wz{F}_{k,M}(x,y)\r|\ls \frac1{V(x,y)}
\end{align*}
and that, when $d(x,x')\le (2A_0)^{-1}d(x,y)$ with $x\neq y$,
\begin{align*}
\sum_{k\le |L|}\lf|\wz{F}_{k,M}(x,y)-\wz{F}_{k,M}(x',y)\r|
+\sum_{k\le |L|}\lf|\wz{F}_{k,M}(y,x)-\wz{F}_{k,M}(y,x')\r|
&\ls \lf[\frac{d(x,x')}{d(x,y)}\r]^{\bz_1}\frac1{V(x,y)},
\end{align*}
where $\bz_1\in(0,\bz\wedge\gz)$ and the implicit positive constants are independent of $M$,  $L$, $x$, $x'$
and $y$. Thus, $\sum_{|k|\le L} \wz F_{k, M}$ has a standard $\bz_1$-Calder\'on-Zygmund kernel.
From the well-known Calder\'on-Zygmund theory on spaces of homogeneous type in \cite{CW71}, it follows that
the operator $\sum_{|k|\le L} \wz F_{k, M}$ is bounded on $L^p(X)$ for any $p\in(1,\infty)$, with its
operator norm independent of $M$ and $L$.

With this, by a standard density argument as in the proof of Step 3) of the proof of Theorem \ref{thm:hcrf}, we
conclude that \eqref{eq:hdrf} converges in  $L^p(X)$ when $f\in L^p(X)$ with $p\in(1,\fz)$. This finishes the
proof Step 3) and hence of Theorem \ref{thm:hdrf} overall.
\end{proof}

We state some other homogeneous discrete Calder\'on reproducing formulae,
with the proof omitted due to the similarity.

\begin{theorem}\label{thm:hdrf3}
Let all the notation be as in Theorem \ref{thm:hdrf}. Then there exist sequences
$\{\overline{Q}_k^{(0)}\}_{k=-\fz}^\fz$, $\{\overline{Q}_k^{(1)}\}_{k=-\fz}^\fz$ and
$\{\overline{Q}_k^{(2)}\}_{k=-\fz}^\fz$ of bounded linear operators on $L^2(X)$ such that, for any $f$ in
$\GOO{\bz,\gz}$ [resp., $L^p(X)$ with any given $p\in(1,\fz)$],
\begin{align}\label{eq:hdrf3}
f(\cdot)&=\sum_{k=-\fz}^\fz\sum_{\az\in\CA_k}\sum_{m=1}^{N(k,\az)}\int_{Q_\az^{k,m}}Q_k(\cdot,y)\,d\mu(y)
\overline{Q}_k^{(0)}f\lf(y_\az^k\r)\\
&=\sum_{k=-\fz}^\fz\sum_{\az\in\CA_k}\sum_{m=1}^{N(k,\az)}Q_k\lf(\cdot,y_\az^{k,m}\r)
\int_{Q_\az^{k,m}}\overline{Q}^{(1)}_kf(y)\,d\mu(y)\noz\\
&=\sum_{k=-\fz}^\fz\sum_{\az\in\CA_k}\sum_{m=1}^{N(k,\az)}\mu\lf(Q_\az^{k,m}\r)
Q_k\lf(\cdot,y_\az^{k,m}\r)\overline{Q}^{(2)}_kf\lf(y_\az^{k,m}\r),\noz
\end{align}
where all the series converge in $\mathring{\CG}^\eta_0(\bz,\gz)$ [resp., $L^p(X)$ with any given
$p\in(1,\fz)$]. Moreover, for any $k\in\zz$, the kernel of $\overline{Q}_k^{(i)}$ with $i\in\{0,1,2\}$ satisfies
the size condition \eqref{eq:atisize}, the regularity condition \eqref{eq:atisregx}
only for the second variable, and also the  cancellation condition \eqref{eq:x00}.
\end{theorem}

\begin{remark}\label{rem:hdrf}
In Theorems \ref{thm:hdrf} and \ref{thm:hdrf3}, for any $i\in\{0,1,2\}$, all the estimates satisfied by
$\wz{Q}_k^{(i)}$ and
$\overline{Q}_k^{(i)}$ are independent of $\{y_\az^{k,m}:\ \az\in\CA_k,\ m\in\{1,\ldots,N(k,\az)\}\}$, and also
independent of $\bz$ and $\gz$ whenever $(\bz,\gz)$ belongs to a compact subset $K$ of $(0,\eta)^2$
(but, in this case, may depend on $K$).
\end{remark}

Since we already have  Theorems \ref{thm:hdrf} and \ref{thm:hdrf3}, then a duality argument implies the
following conclusion, the details being omitted.

\begin{theorem}\label{thm:hdrf4}
Let all the notation be as in Theorems \ref{thm:hdrf}. Then, for any $f\in(\GOO{\bz,\gz})'$, both
\eqref{eq:hdrf} and \eqref{eq:hdrf3} hold true in $(\GOO{\bz,\gz})'$.
\end{theorem}

\section{Inhomogeneous Calder\'{o}n reproducing formulae}\label{irf}

In this section, we establish Calder\'{o}n reproducing formulae by using the newly defined inhomogeneous
approximation of the identity.
\begin{definition}\label{def:ieti}
A sequence $\{Q_k\}_{k=0}^\fz$ of bounded linear operators on $L^2(X)$ is called an \emph{inhomogeneous
approximation of the identity with exponential decay} (for short, $\exp$-IATI) if $\{Q_k\}_{k=0}^\fz$
has the following properties:
\begin{enumerate}
\item $\sum_{k=0}^\fz Q_k=I$ in $L^2(X)$;
\item for any $k\in\nn$, $Q_k$ satisfies (ii) through (v) in Definition \ref{def:eti};
\item $Q_0$ satisfies (ii), (iii) and (iv) in Definition \ref{def:eti} with $k:=0$ but without the term
$$
\exp\lf\{-\nu\lf[\max\lf\{d\lf(x,\CY^0\r),d\lf(y,\CY^0\r)\r\}\r]^a\r\};
$$
moreover, $\int_X Q_0(x,y)\,d\mu(y)=1=\int_X Q_0(y,x)\,d\mu(y)$ for any $x\in X$.
\end{enumerate}
\end{definition}
Via the above $\exp$-IATIs,  we show inhomogeneous continuous and discrete Calder\'{o}n reproducing formulae in
Sections \ref{icrf} and \ref{idrf}, respectively.

\subsection{Inhomogeneous continuous Calder\'{o}n reproducing formulae}\label{icrf}

By Definition \ref{def:ieti}, we write
\begin{equation}\label{eq:defTR2}
I=\sum_{k=0}^\fz Q_k=\sum_{k=0}^\fz\lf(\sum_{l=0}^\fz Q_l\r)Q_k
=\sum_{k=0}^\fz Q_k^NQ_k+\sum_{k=0}^\fz\sum_{|l|>N}Q_{k+l}Q_k=:\FT_N+\FR_N,
\end{equation}
where
\begin{align}\label{eq-z1}
Q_k^N:=\begin{cases}
\displaystyle \sum_{l=0}^{k+N}Q_l & \text{if $k\in\{0,\ldots,N\}$,}\\
\displaystyle \sum_{l=k-N}^{k+N}Q_l & \text{if $k\in\{N+1,N+2,\ldots\}$}\\
\end{cases}
\qquad
\textup{and}\qquad Q_k:=0\;\textup{if}\; k\in\zz\setminus\zz_+.
\end{align}
Therefore, for any $x\in X$,
\begin{equation}\label{eq:QkNint}
\int_X Q_k^N(x,y)\,d\mu(y)=\int_X Q_k^N(y,x)\,d\mu(y)=\begin{cases}
1 & \text{if $k\in\{0,\ldots,N\}$,}\\
0 & \text{if $k\in\{N+1,N+2,\ldots\}$.}
\end{cases}
\end{equation}
Next, we consider the
boundedness of $\FR_N$ on  $L^2(X)$ and $\CG(\bz,\gz)$. To this end, we prove
the following two lemmas.

\begin{lemma}\label{lem:RNi}
Fix $N\in\nn$ and $\eta'\in(0,\eta)$. Then  $\FR_N$ in \eqref{eq:defTR2} is a standard
$\eta'$-Calder\'on-Zygmund operator with the kernel satisfying (a) of Theorem \ref{thm:Kbdd} and
(d) and (e) of Theorem \ref{thm:Kibdd} with $s:=\eta'$, $r_0:=1$,
$\sigma\in(0,\infty)$, $C_T:=C\dz^{(\eta-\eta')N}$ and
$\|\FR_{N}\|_{L^2(X)\to L^2(X)}\le C\dz^{\eta'N}$, where $C$ is a positive constant independent of $N$.
\end{lemma}

\begin{proof}
In the definition of $\FR_N$,  when $\min\{k+l,k\}=0$, we have
$Q_{k+l}Q_k=Q_{l}Q_0$ with $l>N$ or $Q_{k+l}Q_k=Q_0Q_k$ with $k>N$, where we recall that $Q_0$ has no cancellation.
Thus,
\begin{align}\label{eq-z0}
\FR_N=\sum_{\gfz{k>0,\,k+l> 0}{|l|>N}}Q_{k+l}Q_k+\sum_{l>N}Q_{l}Q_0+\sum_{k>N} Q_0Q_k.
\end{align}
Following the proofs of Lemma \ref{lem:ccrf2} and Proposition \ref{prop:sizeRN},
we deduce that the first term in the right-hand side of \eqref{eq-z0} is a standard Calder\'on-Zygmund operator
with the kernel satisfying (a) of Theorem \ref{thm:Kbdd}.

Let $\sigma\in(0,\infty)$ and $\eta'\in(0,\eta)$.
Notice that, for any $x,\ y\in X$ and $k,\ l\in\zz$ satisfying $\min\{k+l,k\}\ge 0$,
\begin{align}\label{eq-z}
\exp\lf\{-\frac c2\lf[\frac{d(x,y)}{\dz^{(k+l)\wedge k}}\r]^a\r\}\ls \lf[\frac{\dz^{(k+l)\wedge k}}{d(x,y)}\r]^\sigma\ls \lf[\frac{1}{d(x,y)}\r]^\sigma.
\end{align}
When $d(x,y)\ge 1$, by Lemma \ref{lem:ccrf1}(i), Remark \ref{rem:andef}(i), \eqref{eq-z} and Lemma
\ref{lem:sum2}, we have
\begin{align*}
\sum_{\gfz{k>0,\,k+l> 0}{|l|>N}}\lf|Q_{k+l}Q_k(x,y)\r|
&\ls \sum_{\gfz{k>0,\,k+l> 0}{|l|>N}}\dz^{|l|\eta}\frac 1{V_{\dz^{(k+l)\wedge k}}(x)}
\exp\lf\{-c\lf[\frac{d(x,y)}{\dz^{(k+l)\wedge k}}\r]^a\r\}\exp\lf\{-c\lf[\frac{d(x,\CY^{(k+l)\wedge
k})}{\dz^{(k+l)\wedge k}}\r]^a\r\}\\
&\ls \dz^{N\eta}\lf[\frac{1}{d(x,y)}\r]^\sigma \frac1{V(x,y)}.
\end{align*}
This shows that the first term in the right-hand side of \eqref{eq-z0} satisfies (d) of Theorem \ref{thm:Kibdd}.

When $d(x,y)\ge 1$ and $d(x,x')\le(2A_0)^{-1}d(x,y)$, then, from Corollary \ref{cor:mixb}(i), \eqref{eq-z} and
Lemma \ref{lem:sum2}, we deduce that
\begin{align*}
&\sum_{\gfz{k>0,\,k+l> 0}{|l|>N}}
|Q_{k+l}Q_{k}(x,y)-Q_{k+l}Q_k(x',y)|\\
&\quad\ls \sum_{\gfz{k>0,\,k+l> 0}{|l|>N}}\dz^{|l|(\eta-\eta')}\lf[\frac{d(x,x')}{d(x,y)}\r]^{\eta'}\frac 1{V_{\dz^{(k+l)\wedge k}}(x)}
\exp\lf\{-c\lf[\frac{d(x,y)}{\dz^{(k+l)\wedge k}}\r]^a\r\}\exp\lf\{-c\lf[\frac{d(x,\CY^{(k+l)\wedge k})}{\dz^{(k+l)\wedge k}}\r]^a\r\}\\
&\quad\ls \dz^{(\eta-\eta') N} \lf[\frac{1}{d(x,y)}\r]^\sigma \lf[\frac{d(x,x')}{d(x,y)}\r]^{\eta'}\frac1{V(x,y)}.
\end{align*}
Thus, the first term in the right-hand side of \eqref{eq-z0} satisfies (e) of Theorem \ref{thm:Kibdd} .

Due to Remark \ref{rem:mix}, following the proofs of Lemma \ref{lem:ccrf1} and Corollary \ref{cor:mixb}, with
$\wz Q_jQ_k$ therein replaced by $Q_lQ_0$ and $Q_0Q_k$ with $l,\ k>N$, we find that the cancellation of $Q_0$
is not needed (see Remark \ref{rem:mix}) and hence, all conclusions of Lemma \ref{lem:ccrf1} and Corollary
\ref{cor:mixb} hold true for $Q_lQ_0$ and $Q_0Q_k$, only with the factor $\exp\{-c[d(x,\CY^0)]^a\}$ therein
removed. Therefore, by Proposition \ref{prop:basic}(ii) and Lemma \ref{lem:HL}, we find that
$\|Q_lQ_0\|_{L^2(X)\to L^2(X)}\ls\dz^{l\eta}$ and $\|Q_0Q_k\|_{L^2(X)\to L^2(X)}\ls\dz^{k\eta}$, which further
imply that
$$
\sum_{l>N}\|Q_lQ_0\|_{L^2(X)\to L^2(X)}+\sum_{k>N}\|Q_0Q_k\|_{L^2(X)\to L^2(X)}\ls\dz^{N\eta}.
$$
Then we deduce that the second and the third terms in the right-hand side of \eqref{eq-z0} are bounded on
$L^2(X)$.

For any $x,\ y\in X$, by the proof Lemma \ref{lem:ccrf1}(i) and \eqref{eq:doub}, we obtain, when $x\neq y$.
\begin{align*}
\sum_{l>N}\lf|Q_{l}Q_0(x,y)\r|
&\ls \dz^{N\eta} \frac {\exp\lf\{-c[{d(x,y)}]^a\r\}}{V_{1}(x)}\ls \dz^{N\eta}
\min\lf\{1,\ \lf[\frac1{d(x,y)}\r]^\sigma\r\} \frac1{V(x,y)}.
\end{align*}
Moreover, when $d(x,x')\le(2A_0)^{-1}d(x,y)$ with $x\neq y$, by the proof of Corollary \ref{cor:mixb}(i) and
\eqref{eq:doub}, we find that
\begin{align*}
&\sum_{l>N}\lf|Q_{l}Q_0(x,y)-Q_{l}Q_0(x',y)\r|+\sum_{l>N}\lf|Q_{l}Q_0(y,x)-Q_{l}Q_0(y,x')\r|\\
&\quad \ls \dz^{N(\eta-\eta')} \lf[\frac{d(x,x')}{d(x,y)}\r]^{\eta'} \frac {\exp\lf\{-c[{d(x,y)}]^a\r\}}{V_{1}(x)}\ls \dz^{N(\eta-\eta')}\min\lf\{1,\ \lf[\frac1{d(x,y)}\r]^\sigma\r\} \lf[\frac{d(x,x')}{d(x,y)}\r]^{\eta'}\frac1{V(x,y)}.
\end{align*}
When $d(x,x')\le(2A_0)^{-2}d(x,y)$ and $d(y,y')\le (2A_0)^{-2}d(x,y)$ with $x\neq y$, by the proof of
Corollary \ref{cor:mixb}(ii) and \eqref{eq:doub}, we also obtain
\begin{align*}
&\sum_{l>N} |[Q_lQ_0(x,y)-Q_lQ_0(x',y)]-[Q_lQ_0(x,y')-Q_lQ_0(x',y')]|\\
&\quad\ls\dz^{N(\eta-\eta')} \lf[\frac{d(x,x')}{d(x,y)}\r]^{\eta'} \lf[\frac{d(y,y')}{d(x,y)}\r]^{\eta'} \frac {\exp\lf\{-c[{d(x,y)}]^a\r\}}{V_{1}(x)}\ls\dz^{(\eta-\eta') N}\lf[\frac{d(x,x')}{d(x,y)}\r]^\eta\lf[\frac{d(y,y')}{d(x,y)}\r]^\eta
\frac 1{V(x,y)}.
\end{align*}
Similarly, $\sum_{k>N} Q_0Q_k$ satisfies also the last three formulae displayed above. Therefore, we know that
the second and the third terms in the right-hand side of \eqref{eq-z0} are also standard
$\eta'$-Calder\'on-Zygmund operators with the kernel satisfying (a) of Theorem \ref{thm:Kbdd} and (d) and (e)
of Theorem \ref{thm:Kibdd}. This finishes the proof of Lemma \ref{lem:RNi}.
\end{proof}

To avoid the deficit that $\FR_N$ does not satisfy (b) and (c) of Theorem \ref{thm:Kbdd}, we have
the following lemma via an argument similar to that used in the proof of Lemma \ref{lem:ccrf3}, with $R_N$
replaced by $\FR_N$ and $R_{N,M}$ by $\FR_{N,M}$ defined below, the details being omitted.

\begin{lemma}\label{lem:RNMi}
Let $\{Q_k\}_{k=0}^\fz$ be an $\exp$-{\rm IATI} and, for any $N\in\nn$, $\FR_N$ be defined as in \eqref{eq:defTR2}.
For any $M\in\nn$, let
$$
\FR_{N,M}:=\sum_{k=0}^M\sum_{N<|l|\le M} Q_{k+l}Q_k.
$$
Then the following assertions hold true:
\begin{enumerate}
\item all  the conclusions of Lemma \ref{lem:RNi} remain true for $\FR_{N,M}$ with all the involved positive constants
independent of $M$;
\item  for any $x,\ y\in X$, $\int_X \FR_{N,M}(x,y')\,d\mu(y')=0=\int_X \FR_{N,M}(x',y)\,d\mu(x')$;
\item for any $f\in L^p(X)$ with $p\in[1,\fz]$ and any $x\in X$,
$\FR_{N,M}f(x)=\int_X \FR_{N,M}(x,y)f(y)\,d\mu(y)$;
\item for any $f\in\CG(\bz,\gz)$ with $\bz,\ \gz\in(0,\eta)$, the sequence $\{\FR_{N,M}f(x)\}_{M=1}^\fz$
converges locally uniformly to an element, denoted by $\wz{\FR_N}(f)(x)$, where $\wz{\FR_N}(f)$ differs from
$\FR_Nf(x)$ at most a set of $\mu$-measure $0$;
\item if we extend $\wz{\FR_N}$ to a bounded linear operator on
$L^2(X)$, still denoted  by $\wz{\FR_N}$, then, for any $f\in L^2(X)$, $\wz{\FR_N}f=\FR_Nf$ in $L^2(X)$ and
almost everywhere.
\end{enumerate}
\end{lemma}

Applying Theorems \ref{thm:Kbdd} and \ref{thm:Kibdd} to the operator $\FR_{N,M}$ and then passing limit to
$\FR_N$, we obtain the following conclusion, the details being omitted.

\begin{proposition}\label{prop-add-x}
Let  $x_1\in X$, $r\in(0,\fz)$, $\bz,\ \gz\in(0,\eta)$ and $\eta'\in(\max\{\bz,\gz\},\eta)$.
Then there exists a positive constant $C$, independent of $x_1$, $r$ and $N$, such that
$$
\|\FR_N\|_{L^2(X)\to L^2(X)}\le C\dz^{\eta'N}
$$
and
$$
\|\FR_N\|_{\mathring{\CG}(x_1,r,\bz,\gz)\to\mathring{\CG}(x_1,r,\bz,\gz)}
+\|\FR_N\|_{\CG(x_1,1,\bz,\gz)\to \CG(x_1,1,\bz,\gz)}\le C\dz^{(\eta-\eta')N}.
$$
\end{proposition}

Now we show the following inhomogeneous continuous Calder\'on reproducing formulae.

\begin{theorem}\label{thm:icrf}
Let $\bz,\ \gz\in(0,\eta)$ and $\{Q_k\}_{k=0}^\fz$ be an $\exp$-{\rm IATI}. Then there exist $N\in\nn$ and
a sequence $\{\wz{Q}_k\}_{k=0}^\fz$ of bounded linear operators on $L^2(X)$ such that, for any
$f\in\go{\bz,\gz}$ [or $L^p(X)$ with any given $p\in(1,\fz)$],
\begin{equation}\label{eq:icrf}
f=\sum_{k=0}^\fz \wz{Q}_kQ_kf,
\end{equation}
where the series converges in $\go{\bz,\gz}$ [or $L^p(X)$ with any given $p\in(1,\fz)$].
Moreover, for any $k\in\zz_+$, $\wz{Q}_k$ satisfies the size condition \eqref{eq:atisize}, the regularity
condition \eqref{eq:atisregx} only for the first variable and the following integration condition: for any
$x\in X$,
\begin{equation}\label{eq:iwzQint}
\int_X \wz{Q}_k(x,y)\,d\mu(y)=\int_X \wz{Q}_k(y,x)\,d\mu(y)=\begin{cases}
1& \text{if\ \ $k\in\{0,\ldots,N\}$,}\\
0& \text{if\ \  $k\in\{N+1,N+2,\ldots\}$.}
\end{cases}
\end{equation}
\end{theorem}

\begin{proof}
We briefly sketch the main ideas of the proof of this theorem due to its similarity to the proof of Theorem
\ref{thm:hcrf}. Based on Proposition \ref{prop-add-x}, we choose $N\in\nn$ large enough such that, for any
$x_1\in X$ and $r\in(0,\fz)$,
$$
\max\lf\{\|\FR_N\|_{L^2(X)\to L^2(X)},
\|\FR_N\|_{\mathring{\CG}(x_1,r,\bz,\gz)\to\mathring{\CG}(x_1,r,\bz,\gz)},
\|\FR_N\|_{\CG(x_1,1,\bz,\gz)\to \CG(x_1,1,\bz,\gz)}\r\}\le 1/2.
$$
Consequently, $\FT_N:=I-\FR_N$ is invertible on $L^2(X)$, $\CG(x_1,1,\bz,\gz)$ and $\GOX{x_1,r,\bz,\gz}$, with
all operator norms at most $2$. For any $k\in\zz_+$ and $x,\ y\in X$, let
$$
\wz{Q}_k(x,y):=\FT_N^{-1}\lf(Q_k^N(\cdot,y)\r)(x)
$$
with $Q_k^N$ as in \eqref{eq-z1}. Then it is easy to show that \eqref{eq:icrf} holds true in $L^2(X)$.

By Lemma \ref{lem:propQkN}, we know that, when $k\in\{N+1,N+2,\ldots\}$,
$Q_k^N(\cdot,y)=\sum_{l=k-N}^{k+N}Q_{l}Q_k(\cdot,y)\in\GOX{y,\dz^k,\bz,\gz}$
and, by Lemma \ref{lem:ccrf1} and Corollary \ref{cor:mixb} for $Q_lQ_k$, we also know that,
when $k\in\{0,\ldots,N\}$,
$Q_k^N(\cdot,y)=\sum_{l=0}^{k+N}Q_{l}Q_k(\cdot,y)\in\CG(y,1,\bz,\gz)$.
The boundedness of $\FT_N^{-1}$ implies that $\wz{Q}_k(\cdot,y)\in \GOX{y,\dz^k,\bz,\gz}$
when $k\in\{N+1,N+2,\ldots\}$,
and $\wz{Q}_k(\cdot,y)\in \CG(y,1,\bz,\gz)$ when  $k\in\{0,\ldots,N\}$.
If $k\in\{0,\ldots,N\}$, then $\CG(y,1,\bz,\gz)=\CG(y,\dz^k,\bz,\gz)$ with norms depending only on $N$.
Therefore, for any $k\in\zz_+$, the kernel of
$\wz{Q}_k$ satisfies the size condition \eqref{eq:atisize}, and the regularity
condition \eqref{eq:atisregx}  for the first variable.

Similarly to the proof of \eqref{eq:wzQcan}, for any $x\in X$, we have $\int_X \wz{Q}_k(x,y)\,d\mu(y)=0=\int_X
\wz{Q}_k(y,x)\,d\mu(y)$ when $k\in\{N+1,N+2,\ldots\}$. When $k\in\{0,\ldots,N\}$, we write
$$
\wz{Q}_k(x,y)=\FT_N^{-1}\lf(Q_k^N(\cdot,y)\r)(x)=\sum_{j=0}^\fz \lf(\lf(\FR_N\r)^jQ_k^N(\cdot,y)\r)(x).
$$
By Lemma \ref{lem:RNMi} and the dominated convergence theorem, we find that, for any $x\in X$,
$\FR_Ng\in\GOX{x_1,1,\bz,\gz}$ whenever $g\in\CG(x_1,1,\bz,\gz)$.
By this and the definition of $Q_k^N$, we conclude that, for any $y\in X$,
$\FR_NQ_k^N(\cdot,y)\in\GO{y,1,\bz,\gz}$. This, together with the boundedness of $\FR_{N}$ on
$\GO{y,1,\bz,\gz}$ (see Proposition \ref{prop-add-x}), implies that, for any $j\in\nn$ and $y\in X$,
$$
\int_X (\FR_N)^jQ_k^N(x,y)\,d\mu(x)=0.
$$
On the other hand, from the Fubini theorem, \eqref{eq:QkNint} and the cancellation of $\FR_{N,M}$ [see Lemma
\ref{lem:RNMi}(ii)], we deduce that, for any $M\in\nn$ and $x\in X$,
$$
\int_X \FR_{N,M}Q_k^N(x,y)\,d\mu(y)=\int_X\FR_{N,M}(x,z)\int_X Q_k^N(z,y)\,d\mu(y)\,d\mu(z)=0.
$$
Then, repeating the proof of \eqref{eq:RNjcan} with $R_{N,M}$ replaced by $\FR_{N,M}$ and $R_N$ by $\FR_N$, we find
that, for any $j\in\nn$ and $x\in X$,
$$
\int_X (\FR_N)^jQ_k^N(x,y)\,d\mu(y)=0.
$$
By these and \eqref{eq:QkNint}, we conclude that, for any $k\in\{0,\ldots,N\}$ and $x,\ y\in X$,
$$
\int_X \wz{Q}_k(x,y')\,d\mu(y')=\sum_{j=0}^\fz\int_X\lf(\FR_N\r)^jQ_k^N(x,y')\,d\mu(y')
=\int_XQ_k^N(x,y')\,d\mu(y')=1
$$
and, similarly,
$$
\int_X \wz{Q}_k(x',y)\,d\mu(x')=\sum_{j=0}^\fz\int_X\lf(\FR_N\r)^jQ_k^N(x',y)\,d\mu(x')
=\int_XQ_k^N(x',y)\,d\mu(x')=1.
$$
This finishes the proof of \eqref{eq:iwzQint}.

Now we prove that, when $\bz'\in(\bz,\eta)$ and $\gz'\in(\gz,\eta)$ and $f\in\CG(\bz',\gz')$,
\begin{equation}\label{eq:ilim}
\lim_{L\to\fz} \lf\|f-\sum_{k=0}^L\wz{Q}_kQ_kf\r\|_{\CG(\bz,\gz)}=0.
\end{equation}
Notice that $f=\FT^{-1}\FT_N f$ and $\FT_N^{-1}$ is bounded on $\CG(\bz,\gz)$. Then it suffices to show that
\begin{equation}\label{eq:isum}
\lim_{L\to\fz}\lf\|\FT_Nf-\sum_{k=0}^LQ_k^NQ_kf\r\|_{\CG(\bz,\gz)}=
\lim_{L\to\fz}\lf\|\sum_{k=L+1}^\fz Q_k^NQ_kf\r\|_{\CG(\bz,\gz)}=0.
\end{equation}
To simplify our discussion, we can assume that $L\ge N+1$, so that $Q_k^N$ has the cancellation properties.
Thus, the proofs of \eqref{eq:sumsize} and \eqref{eq:sumreg} for the case $k\ge L+1$ also imply
\eqref{eq:isum}, and hence \eqref{eq:ilim}.

Now we prove the convergence of \eqref{eq:icrf} in $\go{\bz,\gz}$. For any $L\in\nn$, define
$$
\FT_{N,L}:=\sum_{k=0}^L Q_k^NQ_k=\sum_{k=0}^L\sum_{|l|\le N}Q_{k+l}Q_k.
$$
Then, repeating the proof of Lemma
\ref{lem:RNMi} with $\FR_{N,M}$ replaced by $\FT_{N,L}$ and the sum $\sum_{N<|l|\le M}$ replaced by
$\sum_{|l|\le N}$, we find that $\FT_{N,L}$ satisfies all the assumptions of Theorem \ref{thm:Kibdd} with
$c_0:=1$, $r:=1$ and $C_T$ a positive constant independent of $L$. Therefore, by Theorem \ref{thm:Kibdd}, we
conclude that $\FT_{N,L}$ is bounded on $\CG(\bz,\gz)$ with its operator norm independent of $L$. Thus, from
a density argument used in Step 2) of the proof of Theorem \ref{thm:hcrf}, we deduce that $\FT_N$ is
bounded on $\go{\bz,\gz}$ and hence \eqref{eq:isum} converges on $\go{\bz,\gz}$. Moreover, since
$I=\FT_N+\FR_N$, it then follows that $\FR_N$ is bounded on $\go{\bz,\gz}$ with
$$
\|\FR_N\|_{\go{\bz,\gz}\to\go{\bz,\gz}}\le \|\FR_N\|_{\CG(\bz,\gz)\to\CG(\bz,\gz)}\le \frac 12.
$$
Therefore, $\FT_N^{-1}$ is bounded on $\go{\bz,\gz}$, which, together with the convergence of
\eqref{eq:isum} in $\go{\bz,\gz}$, implies that \eqref{eq:icrf} converges in $\go{\bz,\gz}$.

Next we prove that \eqref{eq:icrf} holds true in $L^p(X)$ with any given $p\in(1,\fz)$. For any $L\in\nn$, let
$\wz{\FT}_L:=\sum_{k=0}^L \wz{Q}_kQ_k$.
Following the arguments used in Step 3) of the proof of Theorem \ref{thm:hcrf},
with $T_L$ replaced by $\wz\FT_L$, we deduce that, when $L\in\{N+1,N+2,\ldots\}$,
$\sum_{k=N+1}^L \wz{Q}_kQ_k$ is a $\bz_1$-Calder\'on-Zygmund operator for some $\bz_1\in(0,\bz\wedge\gz)$,
so that it is bounded on $L^p(X)$. Meanwhile, by Proposition \ref{prop:basic}(iii),
we  know that $\sum_{k=0}^{N} \wz{Q}_kQ_k$ is also bounded on $L^p(X)$. Altogether, we have
$\|\wz{\FT}_L\|_{L^p(X)\to L^p(X)}\ls 1$,
where the implicit positive constant is independent of $L$.
From this and a density argument as that used in Step 3) of the proof of Theorem \ref{thm:hcrf}, we obtain the
convergence of \eqref{eq:icrf}  in $L^p(X)$. This finishes the proof of Theorem \ref{thm:icrf}.
\end{proof}

Similarly, we also have the  following two theorems, the details being omitted.
\begin{theorem}\label{thm:icrf2}
Suppose that $\bz,\ \gz\in(0,\eta)$ and $\{Q_k\}_{k=0}^\fz$ is an $\exp$-{\rm IATI}.
Then there exist $N\in\nn$ and a sequence $\{\overline{Q}_k\}_{k=0}^\fz$ of bounded linear operators on
$L^2(X)$ such that, for any $f\in\go{\bz,\gz}$ [resp., $L^p(X)$ with any given $p\in(1,\fz)$],
\begin{equation}\label{eq:icrf2}
f=\sum_{k=0}^\fz Q_k\overline{Q}_kf,
\end{equation}
where the series converges in $\go{\bz,\gz}$ [resp., $L^p(X)$ with any given $p\in(1,\fz)$]. Moreover, for any
$k\in\zz_+$, the kernel of $\overline{Q}_k$ satisfies the size condition \eqref{eq:atisize}, the regularity
condition \eqref{eq:atisregx} only for the second variable and
\eqref{eq:iwzQint}.
\end{theorem}

\begin{theorem}\label{thm:icrf3}
Let all the notation be as in Theorems \ref{thm:icrf} and \ref{thm:icrf2}. Then, for any $f\in(\go{\bz,\gz})'$,
\eqref{eq:icrf} and \eqref{eq:icrf2} hold true in $(\go{\bz,\gz})'$.
\end{theorem}

\begin{remark}\label{rem:icrf}
Similarly to Remark \ref{rem:r2}, we can show that, if $K$ is a compact subset of $(0,\eta)^2$, then
Theorems \ref{thm:icrf}, \ref{thm:icrf2} and \ref{thm:icrf3} hold true with the implicit positive constants
independent of $(\bz,\gz)\in K$, but depending on $K$.
\end{remark}

\subsection{Inhomogeneous discrete Calder\'{o}n reproducing formulae}\label{idrf}

In this subsection, we consider the inhomogeneous discrete Calder\'{o}n reproducing formulae.
We use the same notation as in Section \ref{hdrf}. Here we omit the details
because the proofs are combinations of those in Sections \ref{icrf} and \ref{hdrf}.

\begin{theorem}\label{thm:idrf}
Let $\{Q_k\}_{k=0}^\fz$ be an $\exp$-{\rm IATI} and $\bz,\ \gz\in(0,\eta)$. Assume that every $y_\az^{k,m}$
is an arbitrary point in $Q_\az^{k,m}$. Then there exist $N\in\nn$ and sequences $\{\wz{Q}^{(i)}_k\}_{k=0}^\fz$, $i\in\{1,2,3\}$, of bounded
linear operators on $L^2(X)$ such that, for any $f\in\go{\bz,\gz}$ [resp., $L^p(X)$ with any given
$p\in(1,\fz)$],
\begin{align}\label{eq:idrf}
f(\cdot)&=\sum_{k=0}^N\sum_{\az\in\CA_k}\sum_{m=1}^{N(k,\az)}\int_{Q_\az^{k,m}}\wz{Q}^{(1)}_k(\cdot,y)\,d\mu(y)
Q_{\az,1}^{k,m}(f)\\
&\quad+\sum_{k=N+1}^\fz\sum_{\az\in\CA_k}\sum_{m=1}^{N(k,\az)}\int_{Q_\az^{k,m}}\wz{Q}^{(1)}_k(\cdot,y)
\,d\mu(y)Q_kf\lf(y_\az^{k,m}\r)\noz\\
&=\sum_{\az\in\CA_0}\sum_{m=1}^{N(0,\az)}\int_{Q_\az^{k,m}}\wz{Q}^{(2)}_k(\cdot,y)\,d\mu(y)Q_{\az,1}^{0,m}(f)
\noz\\
&\quad+\sum_{k=1}^\fz\sum_{\az\in\CA_k}\sum_{m=1}^{N(k,\az)}\mu\lf(Q_\az^{k,m}\r)
\wz{Q}^{(2)}_k\lf(\cdot,y_\az^{k,m}\r)Q_{\az,1}^{k,m}(f)\noz\\
&=\sum_{\az\in\CA_0}\sum_{m=1}^{N(0,\az)}\int_{Q_\az^{k,m}}\wz{Q}^{(3)}_k(\cdot,y)\,d\mu(y)Q_{\az,1}^{0,m}(f)
\noz\\
&\quad+\sum_{k=1}^N\sum_{\az\in\CA_k}\sum_{m=1}^{N(k,\az)}\mu\lf(Q_\az^{k,m}\r)
\wz{Q}^{(3)}_k\lf(\cdot,y_\az^{k,m}\r)Q_{\az,1}^{k,m}(f)\noz\\
&\quad+\sum_{k=N+1}^\fz\sum_{\az\in\CA_k}\sum_{m=1}^{N(k,\az)}\mu\lf(Q_\az^{k,m}\r)
\wz{Q}^{(3)}_k\lf(\cdot,y_\az^{k,m}\r)Q_kf\lf(y_\az^{k,m}\r),\noz
\end{align}
where the series converge in $\go{\bz,\gz}$ [resp., $L^p(X)$ with any given $p\in(1,\fz)$], where, for any
$z\in X$,
\begin{equation}\label{eq:defQa1kn}
Q_{\az,1}^{k,m}(z):=\frac 1{\mu(Q_{\az}^{k,n})}\int_{Q_\az^{k,m}}Q_k(u,z)\,d\mu(u).
\end{equation}
Moreover, for any $i\in\{1,2,3\}$ and $k\in\zz_+$, $\wz{Q}^{(i)}_k$ has the properties same as
$\wz Q_k$ in Theorem \ref{thm:icrf}.
\end{theorem}

\begin{remark}\label{rem:pridrf}
We only explain the decomposition of $I$ to derive the first equality in \eqref{eq:idrf}. For any
$f\in L^2(X)$ and $x\in X$, by \eqref{eq:defTR2}, we write
\begin{align*}
f(x)&=\sum_{k=0}^\fz Q_k^NQ_kf(x)+\FR_Nf(x)\\
&=\sum_{k=0}^{N}\sum_{\az\in\CA_k}\sum_{m=1}^{N(k,\az)}\int_{Q_{\az}^{k,m}}Q_k^N(x,y)Q_kf(y)\,d\mu(y)
+\sum_{k=N+1}^{\fz}\sum_{\az\in\CA_k}\sum_{m=1}^{N(k,\az)}\int_{Q_{\az}^{k,m}}Q_k^N(x,y)Q_kf(y)\,d\mu(y)\\
&\quad+\FR_Nf(x)\\
&=\lf[\sum_{k=0}^{N}\sum_{\az\in\CA_k}\sum_{m=1}^{N(k,\az)}\int_{Q_{\az}^{k,m}}Q_k^N(x,y)\,d\mu(y)
Q_{\az,1}^{k,m}(f)\r.\\
&\quad+\lf.\sum_{k=N+1}^\fz\sum_{\az\in\CA_k}\sum_{m=1}^{N(k,\az)}\int_{Q_{\az}^{k,m}}Q_k^N(x,y)\,d\mu(y)
Q_kf\lf(y_{\az}^{k,m}\r)\r]+\FR_Nf(x)\\
&\quad+\sum_{k=0}^{N}\sum_{\az\in\CA_k}\sum_{m=1}^{N(k,\az)}\frac 1{\mu(Q_\az^{k,m})}
\int_{Q_{\az}^{k,m}}Q_k^N(x,y)\int_{Q_{\az}^{k,m}}[Q_kf(y)-Q_kf(u)]\,d\mu(u)\,d\mu(y)\\
&\quad+\sum_{k=N+1}^\fz\sum_{\az\in\CA_k}\sum_{m=1}^{N(k,\az)}\int_{Q_{\az}^{k,m}}Q_k^N(x,y)
\lf[Q_kf(y)-Q_kf\lf(y_{\az}^{k,m}\r)\r]\,d\mu(y)\\
&=:\mathscr{S}_Nf(x)+\FR_Nf(x)+\mathscr{R}_N^1f(x)+\mathscr{R}_N^2f(x).
\end{align*}
Then, we can use the method in Sections \ref{hdrf} and \ref{icrf} to consider the boundedness of
the remainders $\FR_N$, $\mathscr{R}_N^1$ and $\mathscr{R}_N^2$ on both $L^2(X)$ and $\CG(\bz,\gz)$, the
details being omitted.
\end{remark}

\begin{theorem}\label{thm:idrf2}
Let $\{Q_k\}_{k=0}^\fz$ be an $\exp$-{\rm IATI} and $\bz,\ \gz\in(0,\eta)$. Assume that $y_\az^{k,m}\in Q_\az^{k,m}$
is an arbitrary point. Then there exist $N\in\nn$ and sequences $\{\overline{Q}^{(i)}_k\}_{k=0}^\fz$,
$i\in\{1,2,3\}$, of bounded linear operators on $L^2(X)$ such that, for any $f\in\go{\bz,\gz}$ [resp., $L^p(X)$
with any given $p\in(1,\fz)$],
\begin{align}\label{eq:idrf2}
f(\cdot)&=\sum_{k=0}^N\sum_{\az\in\CA_k}\sum_{m=1}^{N(k,\az)}\int_{Q_\az^{k,m}}Q_k(\cdot,y)\,d\mu(y)
\overline{Q}_{\az,1}^{(1),k,m}(f)\\
&\quad+\sum_{k=N+1}^\fz\sum_{\az\in\CA_k}\sum_{m=1}^{N(k,\az)}\int_{Q_\az^{k,m}}Q_k(\cdot,y)\,d\mu(y)
\overline{Q}^{(1)}_kf\lf(y_\az^{k,m}\r)\noz\\
&=\sum_{\az\in\CA_0}\sum_{m=1}^{N(0,\az)}\int_{Q_\az^{k,m}}Q_k(\cdot,y)\,d\mu(y)
\overline{Q}_{\az,1}^{(2),0,m}(f)\noz\\
&\quad+\sum_{k=1}^\fz\sum_{\az\in\CA_k}\sum_{m=1}^{N(k,\az)}\mu\lf(Q_\az^{k,m}\r)
Q_k\lf(\cdot,y_\az^{k,m}\r)\overline{Q}_{\az,1}^{(2),k,m}(f)\noz\\
&=\sum_{\az\in\CA_0}\sum_{m=1}^{N(0,\az)}\int_{Q_\az^{k,m}}Q_k(\cdot,y)\,d\mu(y)\overline{Q}_{\az,1}^{(3),0,m}
(f)\noz\\
&\quad+\sum_{k=1}^N\sum_{\az\in\CA_k}\sum_{m=1}^{N(k,\az)}\mu\lf(Q_\az^{k,m}\r)
Q_k\lf(\cdot,y_\az^{k,m}\r)\overline{Q}_{\az,1}^{(3),k,m}(f)\noz\\
&\quad+\sum_{k=N+1}^\fz\sum_{\az\in\CA_k}\sum_{m=1}^{N(k,\az)}\mu\lf(Q_\az^{k,m}\r)
\overline{Q}^{(3)}_k\lf(x,y_\az^{k,m}\r)Q_kf\lf(y_\az^{k,m}\r),\noz
\end{align}
where the series converge in $\GOO{\bz,\gz}$ [resp., $L^p(X)$ with any given $p\in(1,\fz)$] and, for any
$i\in\{1,2,3\}$, $\overline{Q}_{\az,1}^{(i),k,m}$ is defined as in \eqref{eq:defQa1kn}, with $Q_k$ replaced by
$\overline{Q}^{(i)}_k$. Moreover, for
$i\in\{1,2,3\}$ and $k\in\zz_+$, $\overline{Q}^{(i)}_k$ has the properties same as $\overline{Q}_k$ in Theorem
\ref{thm:icrf2}.
\end{theorem}
\begin{theorem}\label{thm:idrf3}
Let all the notation be as in Theorems \ref{thm:idrf} and \ref{thm:idrf2}. Then, for any $f\in(\go{\bz,\gz})'$,
all equalities in \eqref{eq:idrf} and \eqref{eq:idrf2} hold true in $(\go{\bz,\gz})'$.
\end{theorem}

\begin{remark}\label{rem:idrf}
Similarly to Remark \ref{rem:r2}, if $K$ is a compact subset of $(0,\eta)^2$, then
Theorems \ref{thm:idrf} through \ref{thm:idrf3} hold true with the implicit  positive  constants
independent of $(\bz,\gz)\in K$, but depending on $K$.
\end{remark}

\bigskip

\noindent Ziyi He, Dachun Yang (Corresponding author) and Wen Yuan

\medskip

\noindent Laboratory of Mathematics and Complex Systems (Ministry of Education of China),
School of Mathematical Sciences, Beijing Normal University, Beijing 100875, People's Republic of China

\smallskip

\noindent{\it E-mails:} \texttt{ziyihe@mail.bnu.edu.cn} (Z. He)

\noindent\phantom{{\it E-mails:} }\texttt{dcyang@bnu.edu.cn} (D. Yang)

\noindent\phantom{{\it E-mails:} }\texttt{wenyuan@bnu.edu.cn} (W. Yuan)

\bigskip

\noindent Liguang Liu

\medskip

\noindent Department of Mathematics, School of Information, Renmin University of China, Beijing 100872,
People's Republic of China

\smallskip

\noindent{\it E-mail:} \texttt{liuliguang@ruc.edu.cn}
\end{document}